\documentclass[11pt]{amsart}

\usepackage[T1]{fontenc}
\usepackage[latin9]{inputenc}
\usepackage[margin=1.05in]{geometry}
\usepackage{cite}
\usepackage{verbatim}
\usepackage{amstext}
\usepackage{amsthm}
\usepackage{amssymb}
\usepackage{esint}

\usepackage{times} 
\usepackage{mathtools} 
\usepackage{color}
\usepackage{graphicx}
\usepackage{caption}
\usepackage{todonotes}
\usepackage{microtype}
\usepackage{enumerate}
\usepackage{bm} 
\usepackage{hyperref}

\newcommand{\R}{\mathbb{R}}

\newcommand{\etab}{\boldsymbol{\eta}}
\newcommand{\te}{\textrm}
\newcommand{\tacka}{\,\cdot\,}

\newcommand{\veps}{\varepsilon}

\DeclareMathOperator{\diam}{diam}
\DeclareMathOperator{\divv}{div}

\DeclareMathOperator{\argmin}{argmin}

\definecolor{mygreen}{rgb}{0.1,0.75,0.2}

\definecolor{dgrn}{rgb}{0.1,0.6,0.1}
\definecolor{dred}{rgb}{0.6,0,0}
\definecolor{lgray}{rgb}{0.5,0.5,0.5}
\newcommand{\dgrn}{\color{dgrn}}

\newcommand{\nc}{\normalcolor}
\newcommand{\De}[1]{{\color{dgrn}#1}}

 \newenvironment{listi}
  {\begin{list} 
 {(\roman{broj})}
{ \usecounter{broj}}
     \setlength{\labelwidth}{25pt}
  }
{   \end{list} }
\newcounter{broj}

\makeatletter
\ifcsname phantomsection\endcsname
    \newcommand*{\qrr@gobblenexttocentry}[5]{}
\else
    \newcommand*{\qrr@gobblenexttocentry}[4]{}
\fi
\newcommand*{\addsubsection}{%
    \addtocontents{toc}{\protect\qrr@gobblenexttocentry}%
    \subsection}
\makeatother

\makeatletter
\numberwithin{equation}{section}
\numberwithin{figure}{section}
\theoremstyle{plain}
\newtheorem{thm}{\protect\theoremname}[section]
  \theoremstyle{definition}
  \newtheorem{example}[thm]{\protect\examplename}
  \theoremstyle{definition}
  \newtheorem{defn}[thm]{\protect\definitionname}
  \theoremstyle{remark}
  \newtheorem{rem}[thm]{\protect\remarkname}
  \theoremstyle{plain}
  \newtheorem{prop}[thm]{\protect\propositionname}
  \newtheorem{cor}[thm]{Corollary}
  \newtheorem{fact}[thm]{Fact}
  \newtheorem{assumption}[thm]{Assumption}
\newtheorem{lemma}[thm]{Lemma}
\makeatother

\usepackage[english]{babel}
  \providecommand{\definitionname}{Definition}
  \providecommand{\examplename}{Example}
  \providecommand{\propositionname}{Proposition}
  \providecommand{\remarkname}{Remark}
\providecommand{\theoremname}{Theorem}

\title{Nonlocal Wasserstein Distance: Metric and Asymptotic Properties}
\begin{document}

\author{Dejan Slep\v{c}ev \and Andrew Warren }
\address{Department of Mathematical Sciences, Carnegie Mellon University, Pittsburgh, PA 15213, USA.}
\email{slepcev@math.cmu.edu}
\email{awarren1@andrew.cmu.edu}

\begin{abstract}
The seminal result of Benamou and Brenier provides a characterization of the Wasserstein distance as the path of the minimal action in the space of probability measures, where paths are solutions of the continuity equation and the action is the kinetic energy. Here we consider a fundamental modification of the framework where the paths are solutions of nonlocal (jump) continuity equations and the action is a nonlocal kinetic energy. The resulting nonlocal Wasserstein distances are relevant to fractional diffusions and Wasserstein distances on graphs. We characterize the basic properties of the distance and obtain sharp conditions on the (jump) kernel specifying the nonlocal transport that determine whether the topology metrized is the weak or the strong topology. A key result of the paper are the quantitative comparisons between the nonlocal and local Wasserstein distance. 
\end{abstract}

\date{\today}

\maketitle

\tableofcontents

%
%







%

\section{Introduction and Summary of Results}



This paper is devoted to the study of a nonlocal Wasserstein distance defined as the least nonlocal action needed to connect two measures via a nonlocal continuity equation. 
The standard Wasserstein
distance is  the minimal transportation cost to couple two measures, for the quadratic point-to-point cost:
\[
W_{2}(\mu_{0},\mu_{1}):=\left(\inf_{\pi\in\Pi(\mu_{0},\mu_{1})}\int_{\mathbb{R}^{d}\times\mathbb{R}^{d}}|x-y|^{2}d\pi(x,y)\right)^{1/2},
\]
where $\mu_{0}$ and $\mu_{1}$ lie in $\mathcal{P}_{2}(\mathbb{R}^{d})$, the space 
of probability measures on $\mathbb{R}^{d}$ with finite second moments,
and $\Pi(\mu_{0},\mu_{1})$ is the space of all couplings of $\mu_{0}$
and $\mu_{1}$. The celebrated work of Benamou
and Brenier \cite{benamou2000computational} establishes that the Wasserstein
metric has a \emph{dynamical} reformulation inspired by fluid mechanics.
There, one considers the space of all narrowly continuous curves $\rho_{t}:[0,1]\rightarrow\mathcal{P}_{2}(\mathbb{R}^{d})$ and \emph{velocity vector fields} $v_{t}:[0,1]\rightarrow L^{2}(\rho_{t};\mathbb{R}^{d})$
such that the \emph{continuity equation}
\begin{equation}
\partial_{t}\rho_{t}+\divv(\rho_{t}v_{t})=0\label{eq:ce}
\end{equation}
holds in the sense of distributions. In \cite{benamou2000computational}
it is shown that 
\begin{equation}
W_{2}^{2}(\mu_{0},\mu_{1})=\inf\left\{ \int_{0}^{1}\!\! \int_{\mathbb{R}^{d}}|v_{t}(x)|^{2}d\rho_{t}(x)dt:\rho_{0} = \mu_{0},\rho_{1}= \mu_{1},\text{ and }(\rho_{t},v_{t})_{t\in[0,1]}\text{ solve (\ref{eq:ce}})\right\} .\label{eq:bb}
\end{equation}
The integral $\int_{0}^{1}\int_{\mathbb{R}^{d}}|v_{t}(x)|^{2}d\rho_{t}(x)dt$
represents the \emph{total action} along $(\rho_{t},v_{t})_{t\in[0,1]}$;
in other words, the $2$-Wasserstein distance between $\mu_{0}$ and
$\mu_{1}$ can be reformulated as the \emph{least action} of a curve
in $\mathcal{P}_{2}(\mathbb{R}^{d})$ connecting $\mu_{0}$ to $\mu_{1}$
in which the flow of mass is continuous, in the sense of satisfying
equation \ref{eq:ce}.


In this article, our object of study is the class of \emph{nonlocal
transportation metrics} on $\mathcal{P}(\mathbb{R}^{d})$, introduced by Erbar
in \cite{erbar2014gradient}, which are defined in terms of an nonlocal action
minimization problem on the space of curves in $\mathcal{P}(\mathbb{R}^{d})$.
A key difference is that the curves connecting the measures are not solutions of the continuity equation \ref{eq:ce}, but solutions of \emph{nonlocal continuity equation}:
\begin{equation}
\partial_{t}\rho_{t}(x)+\int_{\mathbb{R}^{d}}v_{t}(x,y)\theta(\rho_{t}(x),\rho_{t}(y))\eta(x,y)dy=0, \label{eq:nce}
\end{equation}
where $v_{t}:\mathbb{R}^{d}\times\mathbb{R}^{d}\rightarrow\mathbb{R}$ is the nonlocal velocity, $\eta: \R^d \times \R^d \to [0, \infty)$ is the \emph{weight kernel}  which
encodes the ability to transport mass directly from $x$ to $y$,
and $\theta: [0, \infty) \times [0, \infty) \to [0, \infty)$ allows one to define
 an ``interpolated density'' $\theta(\rho_{t}(x),\rho_{t}(y))$,
which is a generalized average of $\rho_{t}(x)$ and $\rho_{t}(y)$.

The \emph{nonlocal total action} is formally given by 
\begin{equation}
\frac{1}{2}\int_{0}^{1}\int_{\mathbb{R}^{d}\times\mathbb{R}^{d}}v_{t}(x,y)^{2}\theta(\rho_{t}(x),\rho_{t}(y))\eta(x,y)dxdydt.\label{eq:nonlocal-action-nonrigorous}
\end{equation}
Together, \ref{eq:nce} and \ref{eq:nonlocal-action-nonrigorous}
allow us to consider  the family of \emph{nonlocal transportation
distances} (since the distance now depends on the choice of $\eta$
and $\theta$)\emph{ }
\begin{equation}
\mathcal{W}_{\eta,\theta}^{2}(\mu_{0},\mu_{1}):=\inf\left\{ \frac{1}{2}\int_{0}^{1}\! \int_{\mathbb{R}^{d}\times\mathbb{R}^{d}}v_{t}(x,y)^{2}\theta(\rho_{t}(x),\rho_{t}(y))\eta(x,y)dxdydt\right\} \label{eq:nlw-nonrigorous}
\end{equation}
where the infimum runs over all $(\rho_{t},v_{t})_{t\in[0,1]}$ which
solve (\ref{eq:nce}), such that $\rho_{0}=\mu_{0}$ and
$\rho_{1}=\mu_{1}$. This family of distances can be viewed simultaneously
as nonlocal analogues of the Benamou-Brenier formulation of the $W_{2}$
metric, and also as a ``continuum state space'' analogue of the
\emph{graph Wasserstein distance} defined in \cite{chow2012fokker,maas2011gradient,mielke2011gradient},
wherein the underlying space is a finite graph or irreducible Markov
chain, rather than $\mathbb{R}^{d}$ as in the case of $\mathcal{W}_{\eta,\theta}$.

In this paper we investigate topological and metric properties of the family of distances $\mathcal{W}_{\eta,\theta}$ and compare them to the Wasserstein metric $W_2$. 
Some properties have already been established in \cite{erbar2014gradient}: it is known that the $\mathcal{W}_{\eta,\theta}$
distance is lower semicontinuous with respect to narrow convergence,
and that if the kernel $\eta$ has finite second moments, then the
topology induced by $\mathcal{W}_{\eta,\theta}$ is at least as fine
as that of narrow convergence. 


Here we show that the topology metrized by $\mathcal{W}_{\eta,\theta}$ can be strictly stronger than that of narrow convergence. In particular we characterize for which kernels 
 $\mathcal{W}_{\eta,\theta}$ metrizes the narrow, the strong, or an even stronger topology on the space of measures supported within a compact domain. The key to establishing the result is the following proposition which loosely speaking characterizes the effort needed to spread mass from a point to the surrounding region. 

\begin{prop}
\label{prop:dirac separation} Let $\mathcal{W}_{\eta,\theta}$ be
defined as in Definition \ref{def:nlw}. Suppose that $\eta$ and
$\theta$ satisfy Assumptions \ref{assu:eta properties} (i-v) and \ref{assu:theta properties}
respectively. If $\nu$ is any compactly supported probability measure
singular to $\delta_{0}$ (the Dirac measure at the origin), then
depending on the choice of $\eta$ and $\theta$: 
\begin{listi}
\item If $\theta(1,0)=0$ and $\int_{B(0,1)}\eta(|y|)dy<\infty$ then $\mathcal{W}_{\eta,\theta}(\delta_{0},\nu)=\infty$.
\item If $\theta(1,0)>0$ and $\int_{B(0,1)}\eta(|y|)dy<\infty$ then  $\infty> \mathcal{W}_{\eta,\theta}(\delta_{0},\nu)\geq2\left(\int_{\mathbb{R}^{d}}\eta(|y|)dy\right)^{-1/2}$.
\item If instead $\eta$ has \emph{algebraic blow-up }at the origin, that
is, there exists some $s>0$, $\delta>0$, and constant $c$ such
that $\eta(|y|)\geq c|y|^{-d-s}$ when $|y|\leq\delta$, %
{} then instead we have the estimate 
\[
\mathcal{W}_{\eta}(\delta_{0},\mathfrak{m}_{B(0,\delta)})\leq C\delta^{s/2}
\]
 with explicit constant $C$, where $\mathfrak{m}_{B(0,\delta)}$
is the uniform probability measure on $B(0,\delta)$. In particular,
$\inf_{\nu\in\{\mathcal{P}(\mathbb{R}^{d}):\nu\bot\delta_{0}\}}\mathcal{W}_{\eta,\theta}(\delta_{0},\nu)=0$.
\end{listi}
\end{prop}

We remark that the proposition is not a trichotomy:  the case where $\int_{B(0,1)}\eta(|y|)dy=\infty$,
but $\eta(|y|)$ does not grow strictly faster than $|y|^{-d}$ near the origin, remains
open.

We prove Proposition \ref{prop:dirac separation} in Section \ref{subsec:Expel-problem}. 
While Proposition \ref{prop:dirac separation} shows that in case (i),
the topology induced by $\mathcal{W}_{\eta,\theta}$ on $\mathcal{P}(\mathbb{R}^{d})$
is highly disconnected, we give a further structural description of
the topology of $\mathcal{W}_{\eta,\theta}$ in the other two cases.
Moreover in case (ii) we establish a quantitative comparison to a combination of total variation and transportation distances. 
Let $W_1$ be the Monge distance, that is the optimal transportation distance with linear cost.
\begin{thm}
\label{thm:nlw topology structure} Let $\mathcal{W}_{\eta,\theta}$
be defined as in Definition \ref{def:nlw}. Suppose that $\eta$ and
$\theta$ satisfy Assumptions \ref{assu:eta properties} (i-v) and \ref{assu:theta properties}
respectively. If $\theta(1,0)>0$ and $\int_{B(0,1)}\eta(|y|)dy<\infty$,
then on a compact domain, there exists an explicit constant $C$ such
that for all $\mu,\nu\in\mathcal{P}(\mathbb{R}^{d})$, 
\[
\frac{1}{C}TV^{1/2}(\mu,\nu)\leq\mathcal{W}_{\eta,\theta}(\mu,\nu)\leq C\cdot TV^{1/2}(\mu,\nu).
\]
In particular, on a compact domain, $\mathcal{W}_{\eta,\theta}$ metrizes
the strong topology on probability measures. 

Conversely, in the case where $\int_{B(0,1)}\eta(|y|)dy=\infty$ but
$\eta$ has finite second moment, then we merely have the lower bound
\[
W_{1}(\mu,\nu)\leq C\mathcal{W}_{\eta,\theta}(\mu,\nu);
\]
and furthermore, if there exists some $s>0$, $\delta>0$, and constant
$c$ such that $\eta(|y|)\geq c|y|^{-d-s}$ when $|y|\leq\delta$,
then on a compact domain, $\mathcal{W}_{\eta,\theta}$ metrizes the
weak topology on probability measures. 
\end{thm}

The lower bounds asserted in Theorem \ref{thm:nlw topology structure}
are established in Section \ref{subsec:Global-lower-bounds}; the
corresponding upper bound makes use of the estimates from Section
\ref{subsec:Expel-problem}, and is proved in Section \ref{subsec:Global-upper-bounds}.
\medskip

We now turn to making the quantitative comparison between the 
Wasserstein distance and the nonlocal Wasserstein distances more precise. 
In particular we show that when the kernel of nonlocal transport $\eta$ is localized the nonlocal Wasserstein distance converges to the Wasserstein distance up to the appropriate scaling. We furthermore obtain explicit error bounds on the difference. 
\begin{thm}
\label{thm:nonlocal-to-local} Let $\mathcal{W}_{\eta,\theta}$ be
defined as in Definition \ref{def:nlw}.. Suppose that $\eta$ and
$\theta$ satisfy Assumptions \ref{assu:eta properties} and \ref{assu:theta properties}
respectively. Let $\varepsilon\in(0,1]$, and define $\eta_{\varepsilon}(|x-y|):=\varepsilon^{-d}\eta\left(\frac{|x-y|}{\varepsilon}\right)$.
Suppose that $\eta(|y|)$ has finite second moment $M_{2}(\eta):=\int|y|^{2}\eta(|y|)dx$.
Let $\rho_{0},\rho_{1}\in\mathcal{P}(\mathbb{R}^{d})$. Then the following
estimates hold: 
\begin{listi}
\item Suppose either that $\eta(|y|)$ is integrable and $\theta(1,0)>0$,
or that there exists some $s>0$, $\delta>0$, and constant $c$ such
that $\eta(|y|)\geq c|y|^{-d-s}$ when $|y|\leq\delta$. Then, there
exists a constant $C_{d,\theta,\eta}$ depending solely and explicitly
on $d$, $\theta$, and $\eta$, such that
\[
\sqrt{\frac{M_{2}(\eta)}{2d}}\varepsilon\mathcal{W}_{\eta_{\varepsilon},\theta}(\rho_{0},\rho_{1})\leq\left(1+\sqrt{\varepsilon}\right)^{2}W_{2}(\rho_{0},\rho_{1})+C_{d,\theta,\eta}\sqrt{\varepsilon}.
\]
\item Suppose that $\rho_{0}$ and $\rho_{1}$ are supported inside some
domain of radius $R$. Then, there is a constant $C_{R,\eta}$ depending
solely and explicitly on $R$ and $\eta$ such that 
\[
W_{2}^{2}(\rho_{0},\rho_{1})\leq\frac{M_{2}(\eta)}{2d}\varepsilon^{2}\mathcal{W}_{\eta_{\varepsilon},\theta}^{2}(\rho_{0},\rho_{1})+C_{R,\eta}\sqrt{\varepsilon}.
\]
\end{listi}
In particular, when restricting attention to probability measures supported on a compact domain, these estimates imply the Gromov-Hausdorff
convergence of $\sqrt{\frac{M_{2}(\eta)}{2d}}\varepsilon\mathcal{W}_{\eta_{\varepsilon},\theta}$
to $W_{2}$ as $\varepsilon\rightarrow0$.
\end{thm}

Part 1 of Theorem \ref{thm:nonlocal-to-local} is deduced, with an
explicit constant for $C_{d,\theta,\eta}$, as Corollary \ref{upper-bound};
likewise, Part 2 of Theorem \ref{thm:nonlocal-to-local} is deduced,
with an explicit constant for $C_{R,\eta}$, in Corollary \ref{lower-bound}.
\medskip

\subsection{Related work}
Having stated the main results of the article, let us give some further
motivating discussion regarding nonlocal Wasserstein distances, and
why their topological and asymptotic properties are of interest.

\subsubsection{Nonlocal Wasserstein metric, and associated gradient flows.}

Nonlocal Wasserstein distances were introduced in the work of Erbar \cite{erbar2014gradient}. 
A central result of \cite{erbar2014gradient}  is that the
$\mathcal{W}_{\eta, \theta}$ gradient flow of  entropy $E(\rho) = \int \ln \rho d \rho$
is the \emph{fractional heat equation }
\[
\partial_{t}\rho_{t}+(-\Delta)^{s/2}\rho=0
\]
when $\theta$ is chosen to be the logarithmic mean, and $\eta$
is chosen to be the jump kernel of an $s$-stable  Levy process:  
$\eta(x,y) = - c |x-y|^{-s-d}$.
More broadly, this result suggests that nonlocal parabolic equations may
be studied in an analogous fashion to those parabolic equations (such
as: the heat equation, the Fokker-Planck equation, the porous medium
equation)
which can be cast as $W_{2}$ gradient flows \cite{ambrosio2008gradient}.
Erbar has made important contributions in that direction by establishing the lower-semicontinuity of the nonlocal action, showing the topology of nonlocal Wasserstein distances is at least as strong as that of the Wasserstein distance, and showing that the entropy is geodesically convex with respect to the nonlocal Wasserstein distance.

As with the regular Wasserstein distance it is of interest to consider gradient flows of the functionals that combine some or all of: entropy, potential, and interaction energy. For $\beta_i \geq 0$ for $i=1,2,3$
\begin{equation}
G(\rho) = \beta_1 \int \ln \rho \, d\rho + \beta_2 \int U(x) \, d\rho(x) + \beta_3 \iint K(x,y) d\rho(x) d\rho(y).
\end{equation}
where $K(x,y) =K(y,x)$.
If $\beta_1 \beta_2 >0 $ and $\beta_3=0$ this would be a nonlocal Fokker-Planck equation, for $\beta_1 \beta_3 >0 $ this would be a nonlocal McKean-Vlasov equation.
The issue that arises in nonlocal Wasserstein gradient flows is that the behavior of the solutions at low temperature, for $\beta_1 \ll 1$, crucially depends on the kernel $\eta$ and interpolation $\theta$. In particular the result of Proposition \ref{prop:dirac separation}(i) indicates that when $\beta_1=0$, $\eta$ is integrable and $\theta$ is for example a logarithmic mean then gradient flow of the potential (or interaction) energy  is unable to move a delta mass! This follows from the fact that since the potential energy is finite the gradient flow curves have finite action. This highlights the need to better understand the influence of the choice of $\eta$ and $\theta$ on the nonlocal Wasserstein metric and the resulting gradient flows. 

To  overcome the issues with the freezing of support for the gradient flow of potential and interaction energies,   Esposito, Patacchini, Schlichting, and Slep\v{c}ev \cite{esposito2019nonlocal} studied a modification of the nonlocal transportation framework, inspired by upwind numerical schemes, which allows for the interpolation $\theta$ to depend on the velocity. 
For antisymmetric $v$
\[ \theta(\rho(x), \rho(y), v(x,y)) = \begin{cases}
\rho(x) \quad  & \te{if } v(x,y)\geq 0 \\
\rho(y) & \te{otherwise.}
\end{cases}\]

They study gradient flow of the nonlocal interaction energy ($\beta_1=\beta_2=0$)
with respect to nonlocal Wasserstein distances both on graphs and in the continuum. 


The resulting upwind nonlocal Wasserstein ``distance'' is not symmetric, and is shown to be a 
quasimetric. It  provides a formal \emph{Finslerian} (rather than Riemannian)
differential structure on $\mathcal{P}(\mathbb{R}^{d})$, which is
nonetheless sufficient to develop gradient flows as curves of maximal slope. 

Lastly: in \cite{esposito2021novel}, the authors show that the 1D aggregation
equation 
\[
\partial_{t}f_{t}=\partial_{v}(f_{t}\partial_{v}W*f_{t})\qquad W(v)=c|v|^{3}
\]
can be cast as the gradient flow of the kinetic energy with respect
to a nonlocal transportation metric, which they call the \emph{nonlocal
collision metric}. This metric falls outside the scope of this article
because the action they consider is 2-homogeneous (rather than 1-homogeneous).
While the aggregation equation is also known to be a 2-Wasserstein
gradient flow (of the nonlocal interaction energy, rather than the
kinetic energy), it is nonetheless notable that a nonlocal transportation
metric is recently shown to be physically relevant in kinetic theory.

\subsubsection{Graph Wasserstein distances}

Maas  \cite{maas2011gradient}, Mielke \cite{mielke2011gradient}, and 
Chow, Huang, Li, and Zhou \cite{chow2012fokker} have independently introduced 
a metric structure for probability measures on discrete spaces (finite
graphs or Markov chains) modeled on the Benamou-Brenier formulation
of the $W_{2}$ metric.
 Our setup  largely follows that of Maas.

Let $\mathcal{X}$ be a finite set. Let $\pi$ be some distinguished
probability measure on $\mathcal{X}$. Define 
\[
\mathcal{P}_{\pi}(\mathcal{X}):=\left\{ \rho:\mathcal{X}\rightarrow\mathbb{R}_{+}\mid\sum_{x\in\mathcal{X}}\rho(x)\pi(x)=1\right\} .
\]
In other words, $\mathcal{P}_{\pi}(\mathcal{X})$ is the set of probability
densities on $\mathcal{X}$, w.r.t. $\pi$. 

Consider some irreducible Markov kernel $K:\mathcal{X}\times\mathcal{X}\rightarrow\mathbb{R}_{+}$
such that $\pi$ is the unique stationary measure for $K$, that is,
\[
\pi(y)=\sum_{x\in\mathcal{X}}\pi(x)K(x,y).
\]
Furthermore, assume that $K$ is reversible, that is, the \emph{detailed
balance condition}
\[
K(x,y)\pi(x)=K(y,x)\pi(y)
\]
holds. Equivalently \cite[Chapter 9]{levin2017markov}, one may consider
a connected weighted graph on $\mathcal{X}$ with weights $w(x,y)$,
where $\pi$ is the stationary distribution for the uniform random
walk on $(\mathcal{X},w)$. 

Let $\theta(x,y)$ be an interpolation function\footnote{By this, we informally mean: a function which is a well-behaved, but
possibly nonlinear, average of $x$ and $y$. Rigorously, we mean a
function from $\mathbb{R}_{+}\times\mathbb{R}_{+}\rightarrow\mathbb{R}$
satisfying all the conditions of Assumption \ref{assu:theta properties}
below.}; define the shorthand $\hat{\rho}(x,y):=\theta(\rho(x),\rho(y))$.
We introduce the \emph{graph continuity equation}
\[
\dot{\rho}_{t}(x)+\sum_{y\in\mathcal{X}}v_{t}(x,y)\hat{\rho}_{t}(x,y)K(x,y)=0
\]
where $v_{t}:\mathcal{X}\times\mathcal{X}\rightarrow\mathbb{R}$ is
thought of as a ``vector field'' on $\mathcal{X}$, analogous to
the vector field $v_{t}$ appearing in the (continuum) continuity
equation. The term $\sum_{y\in\mathcal{X}}v_{t}(x,y)\hat{\rho}_{t}(x,y)K(x,y)$
can be interpreted as a graph analogue of the term $\divv(\rho v)$
from the continuity equation on $\mathbb{R}^{d}$. 

The action of a density-potential pair is given by 
\[
\mathcal{A}(\rho,v):=\frac{1}{2}\sum_{x,y\in\mathcal{X}}(v_{t}(x,y))^{2}\hat{\rho}(x,y)K(x,y)\pi(x).
\]
From here, one can define a geodesic metric on $\mathcal{P}_{\pi}(\mathcal{X})$
in a variational fashion, by setting 
\[
\mathcal{W}_{\theta, \eta, \pi}(\bar{\rho_{0}},\bar{\rho_{1}})^{2}:=\inf\left\{ \int_{0}^{1}\mathcal{A}(\rho_{t},v_{t})dt\right\} 
\]
where the infimum runs over all pairs $(\rho_{t},v_{t})_{t\in[0,1]}$
satisfying the graph continuity equation, with
$\rho_{0}=\bar{\rho_{0}}$ and $\rho_{1}=\bar{\rho}_{1}$. %
This Benamou-Brenier-type formulation of a distance on a discrete base space
is more technically straightforward than its continuum
ancestor $W_{2}$. For one, it is shown in \cite{maas2011gradient}
that, at least on the ``interior'' of $\mathcal{P}_{\pi}(\mathcal{X})$
(namely, the subset  $\{\rho(x)\in\mathcal{P}_{\pi}(\mathcal{X}):\forall x\in\mathcal{X},\rho(x)>0\}$)
we can interpret $\mathcal{W}$
as a geodesic metric arising from a \emph{bona fide} Riemannian metric
structure; this is in contrast to the continuum setting, where the
space $\mathcal{P}_{2}(\mathbb{R}^{d})$ can only be understood
``formally'' as a Riemannian manifold. Indeed, for the metric $\mathcal{W}_{\theta, \eta, \pi}$ and related gradient flows,
numerous heuristic arguments from the Otto calculus can be translated
to rigorous arguments in the discrete setting. This has been exploited
to study a variety of evolution equations on discrete spaces as $\mathcal{W}_{\theta, \eta, \pi}$
gradient flows, for instance discrete analogs of the
porous medium equation \cite{erbarmaas2014gradient} and the McKean-Vlasov
equation \cite{erbar2016gradient}.

The reason for the need to introduce an interpolation $\theta$  in discrete setting is rather clear.
Indeed, because mass configurations
are defined on the set of nodes, and vector fields are defined on
edges, the discrete analogue of the flux of the continuity equation, $\rho v$, must combine node- and edge-defined quantities in a noncanonical
fashion. Indeed, in the definition of $\mathcal{W}_{\theta, \eta, \pi}$ the
role of the flux $\rho v$ is played by the quantity $\theta(\rho(x),\rho(y))v(x,y)$,
where the interpolation $\theta$ is introduced in order to define an edge-based quantity (flux) based on vertex-defined quantities (mass). 
While there are many
possible  choices for $\theta$, which $\theta$ one chooses can significantly alter the geometry of $\mathcal{W}_{\theta, \eta, \pi}$. 

\subsubsection*{From graphs to continuum.}


A number of works has investigated the asymptotic properties of graph Wasserstein distances as the graphs converge to a continuum limit. This is of particular interest in data science and for  mesh-free numerical  schemes.
For instance, if one considers a sequence of finite graphs
$G_{n}$ equipped with the shortest-path metric, converging in some sense to a continuous
domain $\mathcal{G}\subset\mathbb{R}^{d}$,

is it the case that $(\mathcal{P}(G_{n}),\mathcal{W}_{\theta, \eta, \pi})$ 
also converges to $(\mathcal{P}(\mathcal{G}),W_{2})$? Similar Gromov-Hausdorff-type
stability results are well-established for a sequence of continuous
domains where each respective space of probability measures is equipped
with the $W_{2}$ metric \cite{villani2008optimal}. However, the
problem of discrete-to-continuum stability for the graph Wasserstein
distance turns out to be considerably more delicate. In \cite{gigli2013gromov},
it is shown that if we consider a sequence $\mathcal{X}_{n}$ of finer
and finer $d$-dimensional regular lattices on the flat $d$-torus
$\mathbb{T}^{d}$, and take for our Markov chain the uniform random
walk on said lattice, then under appropriate rescaling the sequence
of spaces $(\mathcal{P}(\mathcal{X}_{n}),\mathcal{W}_{\theta, \eta, \pi})$ converges
to $(\mathcal{P}(\mathbb{T}^{d}),W_{2})$ in the sense of Gromov-Hausdorff. On the other hand, such convergence does \emph{not} hold for an arbitrary
sequence of regular meshes \cite{gladbach2020scaling}. 
Despite this failure of convergence for general sequences of meshes, Garcia-Trillos has shown \cite{garciatrillos2020gromov} that $\mathcal{W}_{\theta, \eta, \pi}$ corresponding to weighted random geometric graphs (e.g. where vertices are random i.i.d. samples from the Lebesgue measure on the torus) converges in the sense of Gromov-Hausdorff to $(\mathcal{P}(\mathbb{T}^{d}),W_{2})$ as the 
number of vertices goes to infinity and the graph bandwidth converges to zero at appropriate rate (which is such that unweighted graph degrees go to infinity).

Our own Theorem
\ref{thm:nonlocal-to-local} provides another result in this vein,
whereby the 2-Wasserstein distance is recovered in the limit; but
ours gives a \emph{nonlocal-to-local} convergence result, rather than
discrete-to-continuum. 

We should also draw attention to one other question regarding the
graph Wasserstein metrics: what can be said about the geodesics on
$(\mathcal{P}(\mathcal{X}),\mathcal{W}_{\theta, \eta, \pi})$? Two works \cite{erbar2019geometry,gangbo2019geodesics}
have independently investigated graph Wasserstein geodesics via their
\emph{dual }description in terms of solutions to a suitable discrete
\emph{Hamilton-Jacobi equation}. In the present article, an analogous
Hamilton-Jacobi duality result for the nonlocal Wasserstein metric
is exploited in Section \ref{sec:nonlocal hj} to prove Part 2 of Theorem
\ref{thm:nonlocal-to-local}.

\subsubsection{Other related work}

In \cite{peletier2020jump}, Peletier, Rossi, Savar\'e and Tse consider a far-reaching generalization
of the results of \cite{erbar2014gradient}. While \cite{erbar2014gradient}
shows that the fractional heat equation can be viewed as the gradient
flow of $\text{KL}(\cdot\mid\text{Leb})$ with respect to a Wasserstein-like
metric, \cite{peletier2020jump} investigates a large class of reversible
Markov jump processes whose Kolmogorov forward equations may be cast
as so-called \emph{generalized gradient systems}, which are a further
abstraction of the gradient flows in metric spaces studied in \cite{ambrosio2008gradient}.

Finally, let us remark on some related work on the subject of \emph{nonlocal
conservation laws}. We draw attention to two distinct uses of the term
in the literature (although the works cited in the following discussion refer to some other uses of the term in the literature). On the one hand, certain authors use ``nonlocal
conservation law'' to refer to solutions to the (local) continuity equation
$\partial_{t}u_{t}+\divv(u_{t}v_{t})=0$ where $v_{t}$ is itself
a nonlocal functional of $u_{t}$; see \cite{blandin2016well,crippa2013existence}
for a general discussion of this class of equations and an overview
of related literature. In particular, we draw attention to recent
work \cite{colombo2019singular,colombo2020local} investigating the
local limit (namely, as $\varepsilon\rightarrow0$) of nonlocal conservation
laws of the form 
\[
\partial_{t}u_{t}^{\varepsilon}+\divv\left(u_{t}^{\varepsilon}b\left(\eta_{\varepsilon}*u_{t}^{\varepsilon}\right)\right)=0
\]
where $b:\mathbb{R}_{+}\rightarrow\mathbb{R}^{d}$, and $\eta:\mathbb{R}^{d}\rightarrow\mathbb{R}_{+}$
is some convolution kernel and $\eta_{\varepsilon}(x):=\frac{1}{\varepsilon^{d}}\eta\left(\frac{x}{\varepsilon}\right)$.
Formally, the singular limit is given by the conservation law $\partial_{t}u_{t}+\divv(u_{t}b(u_{t}))=0$,
but \cite{colombo2019singular} exhibits counterexamples where $u_{t}^{\varepsilon}\not\rightarrow u_{t}$
(e.g. in $L^{p}$ for $p>1$) even when $b$ and $\eta$ are regular.
More recently, sufficient conditions for nonlocal-to-local convergence
in one dimension have been given in \cite{colombo2020local}; but
it remains the case that formally ``obvious'' nonlocal-to-local
convergence problems can present unexpected technical phenomena.

In a distinct line of work, the articles \cite{du2012new,du2017nonlocal} introduce new classes of "nonlocal conservation laws" where one replaces the divergence term in the continuity equation with a nonlocal divergence-type operator. The nonlocal divergence and gradient in these papers, as well as in 
 \cite{du2012new, gilboa2009nonlocal}), share the properties of the objects we study, but also have important differences. 
 More importantly the authors are interested in conservation laws where the flux $j$ is a nonlinear function of $\rho$. 
Several ways to encode the nonlinear and nonlocal dependence on $\rho$ are developed.  
In \cite{du2017nonlocal}  the authors present, 
in the same vein as in \cite{colombo2020local}, sufficient conditions  which allow one to recover a local conservation law in the limit, under a suitable rescaling of the nonlocal divergence operator.

\section{Preliminaries}\label{sec:Preliminaries}

\addsubsection*{Wasserstein metric}
\label{subsec:Wasserstein-metric}

Let $\mathcal{P}_{2}(\mathbb{R}^{d})$ denote the space of probability
measures on $\mathbb{R}^{d}$ with finite second moments. The 2-Wasserstein
distance $W_{2}$ on $\mathcal{P}_{2}(\mathbb{R}^{d})$ is defined
by 
\[
W_{2}^{2}(\mu,\nu):=\inf_{\pi\in\Pi(\mu,\nu)}\int|x-y|^{2}d\pi(x,y)
\]
where $\Pi(\mu,\nu)$ is the set  of all transport plans (couplings) of $\mu$ and
$\nu$. The 2-Wasserstein distance also has a well-known dynamical
formulation, due to Benamou and Brenier \cite{benamou2000computational}:
\[
W_{2}^{2}(\mu,\nu)=\inf\left\{ \int_{0}^{1}\int_{\mathbb{R}^{d}}|v_{t}(x)|^{2}d\rho_{t}(x)dt \:: \: \partial_{t}\rho_{t}+\divv(\rho_{t}v_{t})=0,\;\rho_{0}=\mu,\;\rho_{1}=\nu\right\} .
\]
Here, the continuity equation $\partial_{t}\rho_{t}+\divv(\rho_{t}v_{t})=0$
is interpreted in a suitable distributional sense, in particular to
allow $\rho_{t}$ to be a probability measure (rather than a smooth
density). 

The Benamou-Brenier formulation of the $W_{2}$ metric can be interpreted
as showing that the $W_{2}$ distance between $\mu$ and $\nu$ is
given by the minimal total kinetic energy of a unit-time flow
of mass with initial and terminal distribution specified by $\mu$
and $\nu$ respectively. Classically, kinetic energy is either formulated
in terms of position and velocity, or position and momentum; accordingly,
as was observed in \cite{benamou2000computational} (but see also
further discussion and extensions in \cite{brenier2003extended,dolbeault2009new}),
one can also rewrite the
Benamou-Brenier formulation of $W_{2}$ in ``mass-flux'' coordinates:
\[
W_{2}^{2}(\mu,\nu)=\inf\left\{ \int_{0}^{1}\int_{\mathbb{R}^{d}}\left|\frac{d\mathbf{v}_{t}}{d\rho_{t}}\right|^{2}d\rho_{t}(x)dt \: : \: \partial_{t}\rho_{t}+\divv\mathbf{v}_{t}=0,\rho_{0}=\mu,\rho_{1}=\nu\right\} 
\]
where $\mathbf{v}_{t}$ is a locally finite signed measure, which
formally takes the place of $\rho_{t}v_{t}$. This presentation of
$W_{2}$ has the technical advantage that the action $\int_{\mathbb{R}^{d}}\left|\frac{d\mathbf{v}}{d\rho}(x)\right|^{2}d\rho(x)$
is jointly convex and lower semicontinuous in $\rho$ and $\mathbf{v}$
(this is a well-known consequence of Reshetnyak's theorem; for completeness,
we provide a proof of this result which covers the case where $\mathbf{v}\in\mathcal{M}_{loc}(\mathbb{R}^{d})$
in Theorem \ref{thm:locally finite reshetnyak}). In particular, see
Remark \ref{rem:W2 convolution stability} for a useful consequence
of these properties.

\subsection{Nonlocal structure: weight kernel, interpolation, vector calculus}

We equip $\mathbb{R}^{d}$ with an ``underlying nonlocal structure'', as follows.
\begin{assumption}[Weight kernel]
\label{assu:eta properties} The function $\eta:\{(x,y)\in\mathbb{R}^{d}\times\mathbb{R}^{d}\backslash\{x=y\}\}\rightarrow[0,\infty)$
satisfies the following properties: 
\begin{listi}
\item $\eta$ is \emph{continuous} on the set $\{(x,y)\in\mathbb{R}^{d}\times\mathbb{R}^{d}\backslash\{x=y\}:\eta(x,y)>0\};$ 
\item $\eta$ is \emph{isotropic}, that is, there exists radial profile $\etab: (0, \infty) \to [0,\infty)$ such that 
$\eta(x,y) = \etab(|x-y|)$.
\item The radial profile $\etab$ is \emph{non-increasing.}
\item $\eta$ satisfies the \emph{tail moment bound} $\int_{\mathbb{R}^{d}}1\wedge|y|^{2} \, \etab(|y|)dy<\infty$.

\hspace{-4.5em}Additionally, for some results we require that 
\item The support of $\etab$ contains $(0,1]$,

\hspace{-4.5em}Or, furthermore, that
\item The support of $\etab$ is \textit{equal} to $(0,1]$. 
\end{listi}
\end{assumption}

The assumption of isotropy is largely imposed to simplify the statements
of our results. When combined with the tail moment bound, these suffice
to guarantee that Assumption 1.1 from \cite{erbar2014gradient} is
satisfied; the relevance for us is that we make use of several results
from \cite{erbar2014gradient} which require this assumption. On the other hand, the arguments of Sections \ref{exact nonlocalization} and \ref{sec:nonlocal hj} make  use of the assumption that
$\eta$ is compactly supported. Note also that under the assumption $\eta$ is isotropic,
if $\eta$ is compactly supported then the suport of $\etab(|y|)$
is equal to $\bar{B}(0,R)$ for some $R>0$; in assumption (vi), we
fix $R=1$ merely as a convention. Likewise, in view of assumption (iii)  assumption (v) may be viewed as merely a convention unless $\etab$ is identically zero.

Since $\eta$ is taken to be isotropic, we often write $\eta(|x-y|)$
rather than $\etab(|x-y|)$; in other words, we abusively identify $\eta$
with its radial profile.

We write $M_{p}(\eta):=\int_{\mathbb{R}^{d}}|x-y|^{p}\eta(|x-y|)dy$
to denote the $p$-th central moment of $\eta(x,y)$ with $x$ fixed. Note
that due to isotropy, $M_{p}(\eta)$ does not depend on the choice
of $x$. Unless otherwise stated, we do not explicitly assume $p$th
moment bounds on $\eta$, but note that Theorem \ref{thm:nonlocal-to-local}
assumes that $M_{2}(\eta)<\infty$. 

In what follows, given a choice of weight kernel $\eta$, we denote
$G:=\{(x,y)\in\mathbb{R}^{d}\times\mathbb{R}^{d}\backslash\{x=y\}:\eta(x,y)>0\}$.
The intended interpretation is that $G$ is the set of edges we have
placed on $\mathbb{R}^{d}$, with $\eta(x,y)$ being the edge weight
between $x$ and $y$. 

We also assume that the \emph{interpolation function} $\theta$ satisfies the following:
\begin{assumption}[Interpolation function]
\label{assu:theta properties}$\theta:[0,\infty)\times[0,\infty)\rightarrow[0,\infty)$
satisfies the following properties:
\begin{listi}
\item Regularity: $\theta$ is continuous on $[0,\infty)\times[0,\infty)$
and $C^{1}$ on $(0,\infty)\times(0,\infty)$;
\item Symmetry: $\theta(s,t)=\theta(t,s)$ for $s,t\geq0$; 
\item Positivity, normalisation: $\theta(s,t)>0$ for $s,t>0$ and $\theta(1,1)=1$;
\item Monotonicity: $\theta(r,t)\leq\theta(s,t)$ for all $0\leq r\leq s$
and $t\geq0$;
\item Positive 1-homogeneity: $\theta(\lambda s,\lambda t)=\lambda\theta(s,t)$
for $\lambda>0$ and $s\geq t\geq0$;
\item Concavity: the function $\theta:[0,\infty)\times[0,\infty)\rightarrow[0,\infty)$
is concave;
\item Connectedness: $C_{\theta}:=\int_{0}^{1}\frac{dr}{\theta(1-r,1+r)}\in[0,\infty).$
\end{listi}
\end{assumption}

\begin{rem}
Points (i-vi) in the preceding assumption are identical to \cite[Assumption 2.1]{erbar2014gradient},
except that Erbar also assumes that $\theta$ is \emph{zero on the
boundary}, namely $\theta(0,t)=0$ for all $t\geq0$. However, a careful
reading of \cite{erbar2014gradient} indicates that this extra assumption
is never used (except in the sense that Erbar proves some results
only for the logarithmic mean, which is indeed zero on the boundary).
We do \emph{not} assume $\theta$ is zero on the boundary; moreover,
we will see below that whether or not $\theta$ is zero on the boundary
has significant topological consequences for the nonlocal Wasserstein
distance.

We call point (vii) ``connectedness'' because, as discussed
in \cite{maas2011gradient} if instead $C_{\theta}=\infty$, the discrete
$\mathcal{W}_{\theta}$ distance on the space of probability measures
supported on a symmetric graph with two points becomes topologically
disconnected. The assumption that $C_{\theta}<\infty$ is required
for several arguments in Section \ref{sec:metric structural results}.
\end{rem}

\begin{lemma}
\label{lem:arithmetic mean upper bound}Any $\theta$ satisfying
Assumption \ref{assu:theta properties} also satisfies $\theta(r,s)\leq\frac{r+s}{2}$.%
\end{lemma}
\begin{proof}
First, note that by 1-homogeneity and the normalization $\theta(1,1)=1$,
we have $\theta(r,r)=r$ for all $r\in\mathbb{R}_{+}$. At the same
time, symmetry implies that at any point along the line $(r,r)$ in
$\mathbb{R}_{+}\times\mathbb{R}_{+}$ (excluding $(0,0)$, the directional
derivative of $\theta(r,r)$ in the direction orthogonal to the vector
$(1,1)$ must be zero. Therefore, by concavity, $\theta(r,s)$ is
upper-bounded by the hyperplane which takes the value $r$ at $(r,r)$
and has directional derivative zero in the direction orthogonal to
$(1,1)$ at every point $(r,r)$, and the only such hyperplane is
given by $\frac{r+s}{2}$.
\end{proof}

\begin{defn}[Nonlocal gradient and divergence]
 For any function $\phi:\mathbb{R}^{d}\rightarrow\mathbb{R}$ we
define its \emph{nonlocal gradient} $\overline{\nabla}\phi:G\rightarrow\mathbb{R}$
by 
\[
\overline{\nabla}\phi(x,y)=\phi(y)-\phi(x)\text{ for all }(x,y)\in G.
\]
For any $\mathbf{j}\in\mathcal{M}_{loc}(G)$, its \emph{nonlocal divergence}
$\overline{\nabla}\cdot\mathbf{j}\in\mathcal{M}_{loc}(\mathbb{R}^{d})$
is defined as the $\eta$-weighted adjoint of $\overline{\nabla}$,
i.e., 
\[
\int_{\mathbb{R}^{d}}\phi d\overline{\nabla}\cdot\mathbf{j}=-\frac{1}{2}\iint_{G}\overline{\nabla}\phi(x,y)\eta(x,y)d\mathbf{j}(x,y).
\]
\end{defn}

\begin{rem}
It is the nonlocal divergence operator $\overline{\nabla}\cdot$ which
replaces the usual divergence operator in the Definition \ref{def:nce}.
However,  we will simply write out the intergral operator
$\frac{1}{2}\iint_{G}\overline{\nabla}\phi(x,y)\eta(x,y)d\mathbf{j}(x,y)$
explicitly in the sequel; the definition here is just presented for easier comparison
with articles such as \cite{erbar2014gradient,esposito2019nonlocal}.
Additionally, we caution the reader that other conventions for the definition
of nonlocal gradient and divergence
operator are present in the literature,
in particular our definition is not the same as the one presented
in \cite{du2012analysis}.
\end{rem}

\subsection{Action}
We rigorously define the action in the ``flux form''.

\begin{defn}[Action]
\label{def:action}Let $\eta$ satisfy Assumption \ref{assu:eta properties}
and $\theta$ satisfy Assumption \ref{assu:theta properties}. Let $\mu\in\mathcal{P}(\mathbb{R}^{d})$
and $\mathbf{j}\in\mathcal{M}_{loc}(G)$. Let $m\in\mathcal{M}_{loc}(\mathbb{R}^{d})$
be any \emph{reference measure}. Define the \emph{action }of the pair
$(\mu,\mathbf{j})$ by
\[
\mathcal{A}_{\theta,\eta}(\mu,\mathbf{j};m):=\frac12 \, \iint_{G}\frac{\left(\frac{d\mathbf{j}}{d\lambda}(x,y)\right)^{2}}{\theta\left(\frac{d(\mu\otimes m)}{d\lambda}(x,y),\frac{d(m\otimes\mu)}{d\lambda}(x,y)\right)}\eta(x,y)d\lambda(x,y)
\]
where $\lambda$ is taken to be any nonnegative measure in $\mathcal{M}_{loc}^{+}(G)$
such that $|\mathbf{j}|,\mu\otimes m,m\otimes\mu\ll\lambda$. Here,
the fraction in the integrand is understood with the convention that
$\frac{0}{0}=0$.

If $\theta$ and $\eta$ are obvious from context, and $m$ is chosen
to be $\text{Leb}$, the Lebesgue measure on $\mathbb{R}^{d}$, we
simply write $\mathcal{A}(\mu,\mathbf{j})$. 
\end{defn}

\begin{rem}
Note that $\mathcal{A}_{\theta,\eta}(\mu,\mathbf{j};m)$ does not
depend on the choice of $\lambda$ satisfying this domination condition
$|\mathbf{j}|,\mu\otimes m,m\otimes\mu\ll\lambda$, since $\theta$
is 1-homogeneous.
\end{rem}

\begin{rem}
The choice of ``reference measure'' $m$ in Definiton \ref{def:action}
allows us to encode alternate geometries on $\mathbb{R}^{d}$ besides
the usual Euclidean one; phrased differently, Definiton \ref{def:action}
makes sense when working with any metric measure space $(\mathbb{R}^{d},d,m)$
(however if we use a metric $d$ other than the Euclidean one, note
this would change what it means for $\eta$ to be isotropic). In particular,
if we consider a ``weighted measured graph''  $G=(V,E,w,m_{n})$ with vertices in $\mathbb{R}^{d}$,
where $w(x_{i},x_{j})=\eta(x_{i},x_{j})$ for any $(x_{i},x_{j})\in E$,
and $m_{n}$ is a measure supported on $V$, then selecting $m_{n}$
as the reference measure in Definiton \ref{def:action} causes $\mathcal{A}_{\eta,\theta,m_{n}}$
to coincide with the action associated to the \emph{graph Wasserstein
distance }discussed in the introduction, if we also restrict to the
case where $\mu\ll m_{n}$ and $\mathbf{j}\ll\sum_{i,j}w(x_{i},x_{j})m_{n}(x_{i})m_{n}(x_{j})$.
\end{rem}


\begin{lemma}
\label{lem:action convex lsc}The action $\mathcal{A}_{\theta,\eta}(\mu,\mathbf{j};m)$
is jointly convex in $(\mu,\mathbf{j})$, and is jointly lower semicontinuous
in $(\mu,\mathbf{j})$ with respect to the narrow topology on $\mathcal{P}(\mathbb{R}^{d})$
and the weak{*} topology on $\mathcal{M}_{loc}(G)$.
\end{lemma}

\begin{proof}
This is proved in Corollary \ref{cor:action convex l.s.c.}. 
\end{proof}
\begin{rem}[Comparison with action given in Erbar] \label{rem:action comparison} Our Definition \ref{def:action} is superficially different from
the definition of the action given in \cite[Section 2]{erbar2014gradient}.
Nonetheless, it can be seen that the two definitions are, in fact,
equivalent, and so our choice of an alternate presentation of the
action is largely one of taste. 

Specialized to our setting, Erbar defined the action $\mathcal{A}^{\prime}(\mu,\mathbf{v})$
of a pair $(\mu,\mathbf{v})\in\mathcal{P}(\mathbb{R}^{d})\times\mathcal{M}_{loc}(G)$
by 
\[
\mathcal{A}_{\theta,\eta}^{\prime}(\mu,\mathbf{v}):=\iint_{G}\frac{\left(\frac{d\mathbf{v}}{d\lambda}(x,y)\right)^{2}}{2\theta\left(\frac{d\mu^{1}}{d\lambda}(x,y),\frac{d\mu^{2}}{d\lambda}(x,y)\right)}d\lambda(x,y)
\]
where $d\mu^{1}=\eta(x,y)d\mu(x)m(y)$ and $d\mu^{2}=\eta(x,y)dm(x)d\mu(y)$,
and $\lambda$ is any measure in $\mathcal{M}_{loc}(G)$ dominating
$\mathbf{v}$, $\mu^{1}$, and $\mu^{2}$ (and the fraction in the
integrand is understood with the convention that $\frac{0}{0}=0$).
Now, given a pair $(\mu,\mathbf{j})\in\mathcal{P}(\mathbb{R}^{d})\times\mathcal{M}_{loc}(G)$,
it is routine to check (by using the chain rule for Radon-Nikodym
derivatives and the 1-homogeneity of $\theta$) that 
\[
\mathcal{A}_{\theta,\eta}(\mu,\mathbf{j})=\mathcal{A}_{\theta,\eta}^{\prime}(\mu,\mathbf{v})
\]
provided that $\mathbf{v}\in\mathcal{M}_{loc}(G)$ is defined so that
$\frac{d\mathbf{v}}{d\mathbf{j}}(x,y)=\eta(x,y)$. And conversely,
if first given $(\mu,\mathbf{v})$, it holds that $\mathcal{A}_{\theta,\eta}(\mu,\mathbf{j})=\mathcal{A}_{\theta,\eta}^{\prime}(\mu,\mathbf{v})$
provided that $\mathbf{j}\in\mathcal{M}_{loc}(G)$ is defined so that
$\frac{d\mathbf{j}}{d\mathbf{v}}=\frac{1}{\eta(x,y)}$; note that
this latter definition (of $\mathbf{j}$) is unproblematic on $G$,
since by definition $\eta(x,y)>0$ everywhere on $G$. 
\end{rem}



\subsection{Nonlocal continuity equation.}
We define the weak solutions of the nonlocal continuity equation \eqref{eq:nce}, in flux form, in the same way as \cite[Section 2.3]{esposito2019nonlocal}.

\begin{defn}[Nonlocal continuity equation]
\label{def:nce}Let $T>0$. We say that $(\mu_{t},\mathbf{j}_{t})_{t\in[0,T]}$
\emph{solves the nonlocal continuity equation} provided that
\begin{listi}
\item $\mu_{(\cdot)}:[0,T]\rightarrow\mathcal{P}(\mathbb{R}^{d})$ is continuous
when $\mathcal{P}(\mathbb{R}^{d})$ is equipped with the narrow topology,
\item $\mathbf{j}_{(\cdot)}:[0,T]\rightarrow\mathcal{M}_{loc}(G)$ is Borel
when $\mathcal{M}_{loc}(G)$ is equipped with the weak{*} topology,
\item 
$\forall\varphi\in C_{c}^{\infty}([0,T]\times\mathbb{R}^{d})$
\begin{equation}
\int_{0}^{T}\int_{\mathbb{R}^{d}}\partial_{t}\varphi_{t}(x)d\mu_{t}(x)dt+\frac{1}{2}\int_{0}^{T}\iint_{G}\overline{\nabla}\varphi_{t}(x,y)\eta(x,y)d\mathbf{j}_{t}(x,y)dt=0.\label{eq:ncewf}
\end{equation}
\end{listi}
We write $(\mu_{t},\mathbf{j}_{t})\in\mathcal{CE}_{T}$ to indicate
that $(\mu_{t},\mathbf{j}_{t})_{t\in[0,T]}$ satisfies conditions
(i), (ii), and (iii) above. Furthermore, we write $\mathcal{CE}$ to
denote $\mathcal{CE}_{1}$, and write $(\mu_{t},\mathbf{j}_{t})\in\mathcal{CE}_{T}(\nu,\sigma)$
to indicate that $(\mu_{t},\mathbf{j}_{t})\in\ensuremath{\mathcal{CE}_{T}}$
and $\mu_{0}=\nu$ and $\mu_{1}=\sigma$. 
\end{defn}

\begin{rem}[Comparison with nonlocal continuity equation given in \cite{erbar2014gradient}]
\label{rem:nce comparison} In \cite{erbar2014gradient}, a slightly different nonlocal continuity
equation is considered. There, equation \ref{eq:nce} is replaced
with the following equation:
\begin{equation}
\forall\varphi\in C_{c}^{\infty}([0,T]\times\mathbb{R}^{d})\qquad\int_{0}^{T}\int_{\mathbb{R}^{d}}\partial_{t}\varphi_{t}(x)d\mu_{t}(x)dt+\frac{1}{2}\int_{0}^{T}\iint_{G}\overline{\nabla}\varphi_{t}(x,y)d\mathbf{v}_{t}(x,y)dt=0.\label{eq:nce-erbar}
\end{equation}
Here, $\mathbf{v}_{(\cdot)}:[0,T]\rightarrow\mathcal{M}_{loc}(G)$
is likewise assumed to be weak{*} Borel. It is evident that if $(\mu_{t},\mathbf{j}_{t})_{t\in[0,T]}$
satisfies the nonlocal continuity equation in the sense of Definition
\ref{def:nce}, then $(\mu_{t},\mathbf{v}_{t})$ satisfies \ref{eq:nce-erbar},
provided that $\mathbf{v}_{t}\in\mathcal{M}_{loc}(G)$ is defined
so that $\frac{d\mathbf{v}_{t}}{d\mathbf{j}_{t}}(x,y)=\eta(x,y)$.
Conversely, if $(\mu_{t},\mathbf{v}_{t})$ satisfies \ref{eq:nce-erbar},
then defining $\mathbf{j}_{t}$ so that $\frac{d\mathbf{j}_{t}}{d\mathbf{v}_{t}}(x,y)=\frac{1}{\eta(x,y)}$,
we see that $(\mu_{t},\mathbf{j}_{t})_{t\in[0,T]}$ satisfies the
nonlocal continuity equation in the sense of Definition \ref{def:nce}.
Note that this latter definition (of $\mathbf{j}_{t}$) is unproblematic
on $G$, since by definition $\eta(x,y)>0$ everywhere on $G$. 
\end{rem}

It is sometimes advantageous to work with a stronger notion of solution
to the nonlocal continuity equation:
\begin{defn}
\label{def:nce-classical}Let $\rho_{t}(x):[0,1]\times\mathbb{R}^{d}\rightarrow\mathbb{R}_{+}$
be a probability density which is differentiable in $t$, and let
$j_{t}(x,y):[0,1]\times G\rightarrow\mathbb{R}$.
We say that $(\rho_{t},j_{t})_{t\in[0,1]}$ is a \emph{classical solution}
to the nonlocal continuity equation provided that $j_{t}(x,y)$ is antisymmetric and 
\[
\partial_{t}\rho_{t}(x)+\int_{\mathbb{R}^{d}}j_{t}(x,y)\eta(x,y)dy=0
\]
holds pointwise in $t$ and $x$. 
\end{defn}

Note that if $(\rho_{t},j_{t})_{t\in[0,1]}$ is a classical solution
to the nonlocal continuity equation, then it holds that the time-dependent
measures $(\rho_{t}dx,j_{t}dxdy)_{t\in[0,1]}$ are again a solution
to the nonlocal continuity equation in the sense of Definition \ref{def:nce}.

\subsection{Nonlocal Wasserstein metric}
\begin{defn} \label{def:nlw}
Let $\eta$ satisfy Assumption \ref{assu:eta properties} (i-iv)
and $\theta$ satisfy Assumption \ref{assu:theta properties}, and let
$m\in\mathcal{M}_{loc}^{+}(\mathbb{R}^{d})$. The \emph{nonlocal Wasserstein
distance} $\mathcal{W}_{\eta,\theta,m}$ on $\mathcal{P}(\mathbb{R}^{d})$
is defined by 
\[
\mathcal{W}_{\eta,\theta,m}^{2}(\nu,\sigma):=\inf\left\{ \int_{0}^{1}\mathcal{A}_{\eta,\theta,m}(\mu_{t},\mathbf{j}_{t})dt:(\mu_{t},\mathbf{j}_{t})\in\mathcal{CE}(\nu,\sigma)\right\} .
\]
We will write $\mathcal{W}_{\eta,\theta}$ to denote the case where
$m=\text{Leb}$. Furthermore, we will usually drop the explicit reference
to the choice of $\theta$, and simply write $\mathcal{W}_{\eta}$.
\end{defn}

\begin{rem}
In view of Remarks \ref{rem:action comparison} and \ref{rem:nce comparison},
our definition of the nonlocal Wasserstein distance is equivalent
to a special case of the nonlocal Wasserstein distance defined in
\cite{erbar2014gradient}.
\end{rem}

\begin{fact}
$\mathcal{W}_{\eta,\theta,m}$ is a pseudometric on $\mathcal{P}(\mathbb{R}^{d})$.
On $\mathcal{P}(\mathbb{R}^{d})$, $\mathcal{W}_{\eta,\theta,m}^{2}$
is jointly convex, and $\mathcal{W}_{\eta,\theta m}$ is jointly narrowly
lower semicontinuous. 
\end{fact}

The proof of this fact is exactly as in \cite{erbar2014gradient}
and is therefore omitted.

\begin{lemma}[antisymmetric flux]
\label{lem:antisymmetric-tangent-flux} Let $(\mu_{t},\mathbf{j}_{t})_{t\in[0,1]}$
solve the nonlocal continuity equation. Let $\mathbf{j}_{t}^{as}:=(\mathbf{j}-\mathbf{j}^{\top})/2$.
Then
\begin{listi}
\item $(\mu_{t},\mathbf{j}_{t}^{as})_{t\in[0,1]}$ also solves the nonlocal
continuity equation.
\item $\mathcal{A}(\mu_{t},\mathbf{j}_{t}^{as})\leq\mathcal{A}(\mu_{t},\mathbf{j}_{t})$
for all $t\in[0,1]$. 
\end{listi}

In particular, given any NLW geodesic $(\mu_{t})_{t\in[0,1]}$, there
exists a tangent flux which is antisymmetric. 
\end{lemma}

\begin{proof}
The proof of (i) is identical to the argument presented in \cite[Corollary 2.8]{esposito2019nonlocal}
and is therefore omitted.

(ii) This likewise follows from a minor modification of arguments given
in \cite[Lemma 2.6 and Corollary 2.8]{esposito2019nonlocal}. Namely,
we reason from Lemma \ref{lem:action convex lsc}, in particular from
the convexity of the action in the $\mathbf{j}$ variable:
\[
\mathcal{A}(\mu_{t},\mathbf{j}_{t}^{as})\leq\frac{1}{2}\mathcal{A}(\mu_{t},\mathbf{j}_{t})+\frac{1}{2}\mathcal{A}(\mu_{t},-\mathbf{j}_{t}^{T}).
\]
Now, selecting $\lambda_{t}\in\mathcal{M}(G)$ so that $\lambda_{t}\gg-\mathbf{j}_{t}^{T}$, 
\[
\mathcal{A}_{\theta,\eta}(\mu_{t},-\mathbf{j}_{t}^{T})=\iint_{G}\frac{\left(\frac{d\left(-\mathbf{j}_{t}^{T} (x,y)\right)}{d\lambda_{t}}\right)^{2}}{2\theta\left(\frac{d(\mu\otimes m)}{d\lambda}(x,y),\frac{d(m\otimes\mu)}{d\lambda}(x,y)\right)}\eta(x,y)d\lambda_{t}(x,y).
\]
Note that without loss of generality we can take $\lambda_{t}$ to
also dominate $\mathbf{j}_{t}$, and furthermore we can take $\lambda_{t}$
to be symmetric (by replacing $\lambda_{t}$ with $(\lambda_{t}+\lambda_{t}^{T})/2$.
Then, computing that 
\[
\frac{d(-\mathbf{j}_{t}^{T})}{d\lambda_{t}}=-\frac{d\mathbf{j}_{t}^{T}}{d\lambda_{t}}=-\left(\frac{d\mathbf{j}_{t}}{d\lambda_{t}}\right)^{T}
\]
and using the fact that $\theta\left(\frac{d(\mu\otimes m)}{d\lambda}(x,y),\frac{d(m\otimes\mu)}{d\lambda}(x,y)\right)$, $\eta (x,y)$,
and $\lambda_{t}$ are all symmetric, we see that
\[
\iint_{G}\frac{\left(\frac{d\left(-\mathbf{j}_{t}^{T}\right)}{d\lambda_{t}}\right)^{2}}{2\theta\left(\frac{d(\mu\otimes m)}{d\lambda} ,\frac{d(m\otimes\mu)}{d\lambda} \right)} \, \eta d\lambda_{t} =\iint_{G}\frac{\left(\frac{d\mathbf{j}_{t}}{d\lambda_{t}}\right)^{2}}{2\theta\left(\frac{d(\mu\otimes m)}{d\lambda} ,\frac{d(m\otimes\mu)}{d\lambda} \right)} \, \eta d\lambda_{t} .
\]
Consequently, $\mathcal{A}(\mu_{t},-\mathbf{j}_{t}^{T})=\mathcal{A}(\mu_{t},\mathbf{j}_{t})$,
and so $\mathcal{A}(\mu_{t},\mathbf{j}_{t}^{as})\leq\mathcal{A}(\mu_{t},\mathbf{j}_{t})$.
\end{proof}
\begin{rem}
Thus far, we have only defined $\mathcal{W}_{\eta,\theta,m}(\mu,\nu)$
in the case where $\mu,\nu\in\mathcal{P}(\mathbb{R}^{d})$. However,
it is occasionally useful to consider the nonlocal Wasserstein distance
between nonnegative Radon measures \emph{of equal mass}: we do so,
in particular, in the proof of Proposition \ref{prop:W+TV upper bound}
below. In particular, the action $\mathcal{A}(\rho,\mathbf{j})$ is
still well-defined for $\rho\in\mathcal{M}^{+}(\mathbb{R}^{d})$,
and it is trivial to modify the definition of a solution to the nonlocal
continuity equation we have given in Definition \ref{def:nce}, to
allow for the case where $\rho_{t}:[0,1]\rightarrow\mathcal{M}^{+}(\mathbb{R}^{d})$
but $\rho_{t}$ has fixed total mass for all $t\in[0,T]$. Therefore,
in the case where $\mu,\nu\in\mathcal{M}^{+}(\mathbb{R}^{d})$ and
$\Vert\mu\Vert=\Vert\nu\Vert$, $\mathcal{W}_{\eta,\theta,m}(\mu,\nu)$
can be defined exactly as in Definition \ref{def:nlw}. More precisely,
we have the following well-definedness and homogeneity result:
\end{rem}

\begin{prop}[Extension of $\mathcal{W}_{\eta,\theta,m}$ to general nonnegative
measures]
\label{prop:nlw equal mass extension} Let $\mu,\nu\in\mathcal{M}^{+}(\mathbb{R}^{d})$.
Suppose that $\Vert\mu\Vert_{TV}=\Vert\nu\Vert_{TV}=M>0$. Then, 
\[
\mathcal{W}_{\eta,\theta,m}^{2}(\mu,\nu)=M\mathcal{W}_{\eta,\theta,m}^{2}\left(\frac{\mu}{M},\frac{\nu}{M}\right).
\]
In particular, $\mathcal{W}_{\eta,\theta,m}^{2}(\mu,\nu)<\infty$
iff $\,\mathcal{W}_{\eta,\theta,m}^{2}\left(\frac{\mu}{M},\frac{\nu}{M}\right)<\infty$.
\end{prop}

\begin{proof}
Let $(\rho_{t},\mathbf{j}_{t})_{t\in[0,1]}\in\mathcal{CE}\left(\frac{\mu}{M},\frac{\nu}{M}\right)$,
and suppose that $\mathcal{W}_{\eta,\theta,m}^{2}\left(\frac{\mu}{M},\frac{\nu}{M}\right)=\int_{0}^{1}\mathcal{A}(\rho_{t},\mathbf{j}_{t})dt$.
Then, $\left(M\rho_{t},M\mathbf{j}_{t}\right)_{t\in[0,1]}$ is also
a solution to the nonlocal continuity equation, with endpoints $\mu$
and $\nu$; consequently, using the 1-homogeneity of the action, 
\[
\mathcal{W}_{\eta,\theta,m}^{2}(\mu,\nu)\leq\int_{0}^{1}\mathcal{A}(M\rho_{t},M\mathbf{j}_{t})dt=M\int_{0}^{1}\mathcal{A}(\rho_{t},\mathbf{j}_{t})dt=M\mathcal{W}_{\eta,\theta,m}^{2}\left(\frac{\mu}{M},\frac{\nu}{M}\right).
\]
By identical reasoning, we also deduce that $\mathcal{W}_{\eta,\theta,m}^{2}\left(\frac{\mu}{M},\frac{\nu}{M}\right)\leq\frac{1}{M}\mathcal{W}_{\eta,\theta,m}^{2}(\mu,\nu)$. 
\end{proof}

\subsection{Convolutions}

We make frequent use of convolution estimates in Sections \ref{exact nonlocalization}
and \ref{sec:nonlocal hj}. A number of elementary computations relating to
convolutions are deferred to Appendix \ref{sec:additional}; here, we fix
notation and state some basic convolution stability results concerning
the $\mathcal{W}_{\eta}$ distances. Lastly, we show that a specific
convolution kernel, the \emph{Laplace kernel} $K:=c_{K}e^{-|x-y|}$,
has the property that smoothed densities $K*\mu$ have \emph{relative
Lipschitz regularity} of the form $\frac{K*\mu(x)}{K*\mu(x^{\prime})}\leq1+C|x-x^{\prime}|$,
a fact which we exploit repeatedly in Sections \ref{exact nonlocalization}
and \ref{sec:nonlocal hj}.
\begin{defn}
We say that $k:(0,\infty)\rightarrow[0,\infty)$ is a \emph{convolution
kernel} if it is $C^{1}$ on its support and normalized so that $\int_{\mathbb{R}^{d}}k(|x|)dx=1$. 
\end{defn}

Given a convolution kernel $k$ and a (possibly signed, possibly vector-valued)
measure $\mu$, we denote 
\[
k*\mu(x):=\int k(|x-y|)d\mu(y).
\]
We also use the following notation: given a convolution kernel $k$
and a measure $\mu$, we write $\boldsymbol{k}*\mu$ to denote the
measure whose Lebesgue density is given by 
\[
\frac{d(\boldsymbol{k}*\mu)}{dx}=k*\mu.
\]

Separately, for measures $\mathbf{j}\in\mathcal{M}_{loc}(G)$, we
follow \cite{erbar2014gradient} and define the convolution $\boldsymbol{k}*\mathbf{j}\in\mathcal{M}_{loc}(G)$
of a convolution kernel $k$ (on $\mathbb{R}^{d}$) with $\mathbf{j}$
as follows:
\[
\boldsymbol{k}*\mathbf{j}=\int_{\mathbb{R}^{d}}k(|z|)d\mathbf{j}(x-z,y-z)dz
\]
in other words, for all $\varphi\in C_{c}^{\infty}(G)$, 
\[
\iint_{G}\varphi(x,y)d(\boldsymbol{k}*\mathbf{j})(x,y)=\int_{\mathbb{R}^{d}}k(|z|)\varphi(x+z,y+z)d\mathbf{j}(x,y)dz.
\]
Note that this definition may be understood as a special case of convolution
with respect to a translation-invariant group action on a space: in
this case, $\mathbb{R}^{d}$ acts on $G$ by the translation $(x,y)\mapsto(x+z,y+z)$,
and this action is indeed translation-invariant since $\eta(|x-y|)$
is preserved under this translation.

\begin{prop}[Stability of action and metric under convolution]
\label{prop:convolution action stability} Let $k$ be any convolution
kernel. 
\begin{listi}
\item For any $\mu\in\mathcal{P}(\mathbb{R}^{d})$ and $\mathbf{j}\in\mathcal{M}_{loc}(G)$,
$\mathcal{A}_{\theta,\eta}(\boldsymbol{k}*\mu,\boldsymbol{k}*\mathbf{j})\leq\mathcal{A}_{\theta,\eta}(\mu,\mathbf{j})$.
\item Let $(\mu_{t},\mathbf{j}_{t})_{t\in[0,1]}\in\mathcal{CE}$. Then,
$(\boldsymbol{k}*\mu_{t},\boldsymbol{k}*\mathbf{j}_{t})_{t\in[0,1]}\in\mathcal{CE}$
also. 
\item (\cite[Proposition 4.8]{erbar2014gradient}) For any $\mu_{0},\mu_{1}\in\mathcal{P}_{2}(\mathbb{R}^{d})$,
$\mathcal{W}_{\theta,\eta}(\boldsymbol{k}*\mu_{0},\boldsymbol{k}*\mu_{1})\leq\mathcal{W}_{\theta,\eta}(\mu_{0},\mu_{1})$.
\end{listi}
\end{prop}

\begin{proof}
(i) In Lemma \ref{lem:mass-flux convolution inequality}, we show
that
\[
\mathcal{A}_{\eta,\theta}(\boldsymbol{k}*\mu,\boldsymbol{k}*\mathbf{j})\leq\int_{\mathbb{R}^{d}}\mathcal{A}_{\eta,\theta}(\mu_{z},\mathbf{j}_{z})k(z)dz.
\]
The result now follows according to the reasoning given in \cite[proof of Proposition 2.8]{erbar2014gradient}.
(ii) follows by identical reasoning to \cite[proof of Proposition 4.8]{erbar2014gradient}.
Finally, (iii) follows by combining (i) and (ii): indeed, letting
$(\mu_{t},\mathbf{j}_{t})_{t\in[0,1]}$ be an action-minimizing solution
to the nonlocal continuity equation with endpoints $\mu_{0}$ and
$\mu_{1}$, we see that 
\[
\mathcal{W}_{\theta,\eta}^{2}(\mu_{0},\mu_{1})=\int_{0}^{1}\mathcal{A}(\mu_{t},\mathbf{j}_{t})dt\geq\int_{0}^{1}\mathcal{A}(\boldsymbol{k}*\mu_{t},\boldsymbol{k}*\mathbf{j}_{t})dt\geq\mathcal{W}_{\theta,\eta}(\boldsymbol{k}*\mu_{0},\boldsymbol{k}*\mu_{1}).
\]
\end{proof}

\begin{rem}
\label{rem:W2 convolution stability}Similarly, it is known \cite[Lemmas 8.1.9 and 8.1.0]{ambrosio2008gradient}
that if $(\rho_{t},\mathbf{v}_{t})_{t\in[0,1]}$ solves the (local)
continuity equation $\partial_{t}\rho_{t}+\divv\mathbf{v}_{t}=0$
in the sense of distributions, and $k$ is a convolution kernel, then
$(\boldsymbol{k}*\rho_{t},\boldsymbol{k}*\mathbf{v}_{t})_{t\in[0,1]}$
is again a solution of $\partial_{t}\rho_{t}+\divv\mathbf{v}_{t}=0$
in the sense of distributions; and similarly, 
\[
\int_{\mathbb{R}^{d}}\left|\frac{d(\boldsymbol{k}*\mathbf{v}_{t})}{d(\boldsymbol{k}*\mu_{t})}\right|^{2}d(\boldsymbol{k}*\mu_{t})\leq\int_{\mathbb{R}^{d}}\left|\frac{d\mathbf{v}_{t}}{d\mu_{t}}\right|^{2}d\mu_{t}.
\]
Therefore, applying the mass-flux presentation of the $W_{2}$ metric
described in Section \ref{subsec:Wasserstein-metric} above, we can
reason exactly as in the proof of part (iii) of the previous proposition
to deduce that for any convolution kernel $k$ and probability measures
$\mu_{0}$ and $\mu_{1}$,
\[
W_{2}(\boldsymbol{k}*\mu_{0},\boldsymbol{k}*\mu_{1})\leq W_{2}(\mu_{0},\mu_{1}).
\]
\end{rem}

\subsubsection{Relative Lipschitz estimate for the right convolution.}

Let $K(x) = c_K \, e^{-|x|}$ where $\frac{1}{c_K} =\int_{\R^d} e^{-|x|} dx$. Let $K_\delta(x) = \frac{1}{\delta^d} K \left( \frac{x}{\delta}\right)$. 

\begin{lemma}
\label{lem:relative Lipschitz}
Consider $\mu \in \mathcal P_2(\R^d)$. 
Let $\mu_\delta = K_\delta * \mu$
Then 

\begin{equation*}
\ln \frac{\mu_\delta(y)}{\mu_\delta(x)}  \leq \frac{1}{\delta} |y-x|.
\end{equation*}
Furthermore if $|x-y|\leq \delta$ then 
\begin{equation*}
    \mu_\delta(y) \leq \mu_\delta(x) \left(1 + \frac{3}{\delta} |y-x| \right)
\end{equation*}
\end{lemma}
\begin{proof} Let $h = y-x$. We can assume $h \neq 0$
\begin{align*}
 |\ln \mu_\delta(x+h) - \ln \mu_\delta(x) | & \leq 
 \ln \frac{ \int e^{-|x-z|/\delta} e^{|h|/\delta} d\mu(z)}{ \int e^{-|x-z|/\delta} d\mu(z)} = \frac{|h|}{\delta}
\end{align*}
Therefore 
\[ \ln \frac{\mu_\delta(y)}{\mu_\delta(x)}  \leq \frac{1}{\delta} |y-x|. \]
Hence, for $|y-x|<\delta$,
\[  \mu_\delta(y) \leq \mu_\delta(x) e^{|y-x|/\delta} \leq \mu_\delta(x) \left(1 + \frac{3}{\delta} |y-x| \right). \]
\end{proof}

\section{Metric structure of $\mathcal{W}_{\eta,\theta}$ distances}\label{sec:metric structural results}

\subsection{General lower bounds for nonlocal Wasserstein distances}\label{subsec:Global-lower-bounds}

In this subsection, we consider nonlocal Wassersein distances $\mathcal{W}_{\eta,\theta,m}$
with general reference measure $m$, since this complicates our analysis
only minimally. As in \ref{def:nlw}, we assume that $\eta$ satsfies Assumption
\ref{assu:eta properties} (i-iv), and that $\theta$ satisfies Assumption
\ref{assu:theta properties}.

{} We first recall the following result of Erbar (which we specialize
somewhat), which gives a partial characterization of the topology
induced by the nonlocal Wasserstein distance.
\begin{prop}
(\cite[Proposition 4.5]{erbar2014gradient}) \label{prop:W1 lower bound}
Suppose that $\rho_{0},\rho_{1}\in\mathcal{P}(\mathbb{R}^{d})$. The nonlocal Wasserstein distance $\mathcal{W}_{\eta,\theta,m}$
with arbitrary reference
measure $m$ of Definition \ref{def:nlw} satisfies
\[
\sqrt{\frac{2}{\tilde{C}}}W_{1}(\rho_{0},\rho_{1})\leq\mathcal{W}_{\eta,\theta,m}(\rho_{0},\rho_{1}).
\]
Here $\tilde{C}=\sup_{x\in\mathbb{R}^{d}}\int_{\mathbb{R}^{d}}|x-y|^{2}\eta(|x-y|)dm(y)$.
\end{prop}
In particular, observe that this $W_{1}$ lower bound is vacuous in
the case where $m=\text{Leb}$ and the second moment of $\eta$ is
infinite.

When the lower bound in the previous proposition is non-vacuous, this
shows, in particular, that the topology induced by $\mathcal{W}_{\eta,\theta}$
as strong or stronger than the narrow topology on $\mathcal{P}(\mathbb{R}^{d})$. What we show below is that, when $\eta$ is integrable, the topology is strictly stronger.
More precisely we show that the
topology induced by $\mathcal{W}_{\eta,\theta,m}$ is at least as
strong as the \emph{strong topology} on $\mathcal{P}(\mathbb{R}^{d})$. This indicates that the nonlocal Wasserstein distances are fundamentally different from the standard Wasserstein distances. 

\begin{prop}
\label{prop:TV lower bound}Suppose that $\rho_{0},\rho_{1}\in\mathcal{P}(\mathbb{R}^{d})$.
The nonlocal Wasserstein distance $\mathcal{W}_{\eta,\theta,m}$
with arbitrary reference
measure $m$ of Definition \ref{def:nlw} satisfies
\[
\sqrt{\frac{2}{C}}TV(\rho_{0},\rho_{1})\leq\mathcal{W}_{\eta,\theta,m}(\rho_{0},\rho_{1}).
\]
Here $C=\sup_{x\in\mathbb{R}^{d}}\int_{\mathbb{R}^{d}}\eta(|x-y|)dm(y)$.
\end{prop}

This lower bound is vacuous when $\eta$ is non-integrable. This suggests
that the upwind nonlocal transportation distance induces a weaker
topology when $\eta$ is non-integrable; we address this point later
on in Lemma \ref{lem:crude W2 upper bound}.
\begin{proof}
Recall that one of the several equivalent definitions of the TV norm
is as follows:
\[
TV(\rho_{0},\rho_{1})=\sup_{A\in\mathcal{B}(\mathbb{R}^{d})}|\rho_{0}(A)-\rho_{1}(A)|.
\]
Let $A$ be some measurable set such that $|\rho_{0}(A)-\rho_{1}(A)|>TV(\rho_{0},\rho_{1})-\frac{\varepsilon}{2}.$
Any mass-flux pair $(\rho_{t},\mathbf{j}_{t})$ connecting $\rho_{0}$
to $\rho_{1}$ must therefore move at least $TV(\rho_{0},\rho_{1})-\varepsilon$
of mass from $A$ to $A^{C}$ (or vice versa). Without loss of generality,
we can take $A$ to be compact.

Let $(\rho_{t},\mathbf{j}_{t})\in\mathcal{CE}(\rho_{0},\rho_{1})$
be an action-minimizing mass-flux pair, so that $\mathcal{W}_{\eta,\theta,m}^{2}(\rho_{0},\rho_{1})=\int_{0}^{1}\mathcal{A}(\rho_{t},\mathbf{j}_{t})dt$;
without loss of generality we can assume, by \cite[Proposition 4.3]{erbar2014gradient},
that $(\rho_{t},\mathbf{j}_{t})$ is \emph{unit speed} in the sense
that $\mathcal{W}_{\eta,\theta,m}^{2}(\rho_{0},\rho_{1})=\mathcal{A}(\rho_{t},\mathbf{j}_{t})$
for almost all $t\in[0,1]$. We may also assume, without loss of generality,
that $\mathbf{j}_{t}$ is antisymmetric for almost all $t$, by Lemma
\ref{lem:antisymmetric-tangent-flux}. 

Let $\lambda_{t}$ be some measure which dominates all of $\mathbf{j}_{t}$,
$m\otimes\rho_{t}$, and $\rho_{t}\otimes m$. The action for $(\rho_{t},\mathbf{j}_{t})$
is 
\begin{align*}
\mathcal{A}_{\theta}(\rho_{t},\mathbf{j}_{t}) & =\frac{1}{2}\iint_{G}\left(\frac{\left(\frac{d\mathbf{j}_{t}}{d\lambda_{t}}(x,y)\right)^{2}}{\theta\left(\frac{d(\rho_{t}\otimes m)}{d\lambda_{t}}(x,y),\frac{d(m\otimes\rho_{t})}{d\lambda_{t}}(x,y)\right)}\right)\eta(x,y)d\lambda_{t}(x,y).
\end{align*}
Applying the reverse H\"older inequality, and using the fact that $\left|\frac{d\mathbf{j}_{t}}{d\lambda_{t}}\right|=\frac{d|\mathbf{j}_{t}|}{d\lambda_{t}}$,
we get 
\begin{align*}
\mathcal{A}_{\theta}(\rho_{t},\mathbf{j}_{t}) & \geq\frac{1}{2}\left(\iint_{G}\frac{d|\mathbf{j}_{t}|}{d\lambda_{t}}(x,y)\eta(x,y)d\lambda_{t}(x,y)\right)^{2}\\
 & \qquad\times\left(\iint_{G}\theta\left(\frac{d(\rho_{t}\otimes m)}{d\lambda_{t}}(x,y),\frac{d(m\otimes\rho_{t})}{d\lambda_{t}}(x,y)\right)\eta(x,y)d\lambda_{t}(x,y)\right)^{-1}.
\end{align*}
By Lemma \ref{lem:arithmetic mean upper bound}, $\theta$ automatically satisfies $\theta(r,s)\leq(r+s)/2$, so we find that 
\begin{multline*}
\left(\iint_{G}\theta\left(\frac{d(\rho_{t}\otimes m)}{d\lambda_{t}}(x,y),\frac{d(m\otimes\rho_{t})}{d\lambda_{t}}(x,y)\right)\eta(x,y)d\lambda_{t}(x,y)\right)\\
\leq\frac{1}{2}\left(\iint_{G}\left(\frac{d(\rho_{t}\otimes m)}{d\lambda_{t}}(x,y)+\frac{d(m\otimes\rho_{t})}{d\lambda_{t}}(x,y)\right)\eta(x,y)d\lambda_{t}(x,y)\right).
\end{multline*}
By the Radon-Nikodym theorem, we have the estimate 
\begin{align*}
\iint_{G}\frac{d\rho_{t}\otimes m}{d\lambda_{t}}(x,y)\eta(x,y)d\lambda_{t}(x,y) & =\iint_{G}\eta(x,y)d(\rho_{t}\otimes m)(x,y)\\
 & \leq\int_{\mathbb{R}^{d}}\left[\sup_{x\in\mathbb{R}^{d}}\int_{\mathbb{R}^{d}}\eta(|x-y|)dm(y)\right]d\rho_{t}(x)\\
 & =\sup_{x\in\mathbb{R}^{d}}\int_{\mathbb{R}^{d}}\eta(|x-y|)dm(y) =: C.
\end{align*}
The same estimate works if we replace $\rho_{t}(y)$ with $\rho_{t}(x)$;
hence, we conclude that 
\[
\iint_{G}\theta\left(\frac{d(\rho_{t}\otimes m)}{d\lambda_{t}}(x,y),\frac{d(m\otimes\rho_{t})}{d\lambda_{t}}(x,y)\right)\eta(x,y)d\lambda_{t}(x,y)\leq C.
\]
Therefore, (applying the Radon-Nikodym theorem once more)
\[
\mathcal{A}_{\theta}(\rho_{t},\mathbf{j}_{t})\geq\frac{1}{2C}\left(\iint_{G}\frac{d|\mathbf{j}_{t}|}{d\lambda_{t}}(x,y)\eta(x,y)d\lambda_{t}\right)^{2}=\frac{1}{2C}\left(\iint_{G}\eta(x,y)d|\mathbf{j}_{t}|(x,y)\right)^{2}.
\]

Let $\xi_{\delta}$ be a\emph{ cutoff function} for the set $A$,
more precisely $\xi_{\delta}=1$ on $A$, $0$ on $A_{\delta}^{c}$,
and continuous on $\mathbb{R}^{d}$ (the existence of such a $\xi_{\delta}$
is guaranteed by Urysohn's lemma). We use $\xi_{\delta}(x)$ as a
test function in the nonlocal continuity equation: by \cite[Lemma 2.15]{esposito2019nonlocal}
we find that 
\[
\int_{\mathbb{R}^{d}}\xi_{\delta}(x)d\rho_{1}(x)-\int_{\mathbb{R}^{d}}\xi_{\delta}(x)d\rho_{0}=-\frac{1}{2}\int_{0}^{1}\iint_{G}(\xi_{\delta}(y)-\xi_{\delta}(x))\eta(x,y)d\mathbf{j}_{t}(x,y)dt.
\]
Note that $|\xi_{\delta}(y)-\xi_{\delta}(x)|\leq1$, so 
\begin{align*}
\left|\int_{\mathbb{R}^{d}}\xi_{\delta}(x)d\rho_{1}(x)-\int_{\mathbb{R}^{d}}\xi_{\delta}(x)d\rho_{0}\right| & =\frac{1}{2}\left|\int_{0}^{1}\iint_{G}(\xi_{\delta}(y)-\xi_{\delta}(x))\eta(x,y)d\mathbf{j}_{t}(x,y)dt\right|\\
 & \leq\frac{1}{2}\int_{0}^{1}\iint_{G}\eta(x,y)d|\mathbf{j}_{t}|(x,y)dt.
\end{align*}
Now, selecting $\delta>0$ so that 
\[
\left|\left(\int_{\mathbb{R}^{d}}\xi_{\delta}(x)d\rho_{1}(x)-\int_{\mathbb{R}^{d}}\xi_{\delta}(x)d\rho_{0}\right)-\left(\rho_{1}(A)-\rho_{0}(A)\right)\right|<\frac{\varepsilon}{2}
\]
we can compute that 
\begin{align*}
TV(\rho_{0},\rho_{1}) & \leq|\rho_{1}(A)-\rho_{0}(A)|+\frac{\varepsilon}{2}\\
 & \leq\left|\int_{\mathbb{R}^{d}}\xi_{\delta}(x)d\rho_{1}(x)-\int_{\mathbb{R}^{d}}\xi_{\delta}(x)d\rho_{0}\right|+\varepsilon\\
 & \leq\frac{1}{2}\int_{0}^{1}\iint_{G}\eta(x,y)d|\mathbf{j}_{t}|(x,y)dt+\varepsilon\\
 & \leq\frac{1}{2}\sqrt{2C}\int_{0}^{1}\sqrt{\mathcal{A}(\rho_{t},\mathbf{j}_{t})}dt+\varepsilon.
\end{align*}

Since $\mathcal{W}_{\eta,\theta,m}(\rho_{0},\rho_{1}):=\int_{0}^{1}\sqrt{\mathcal{A}(\rho_{t},\mathbf{j}_{t})}dt$,
and $\varepsilon>0$ was arbitrary, we conclude that 
\[
\mathcal{W}_{\eta,\theta,m}(\rho_{0},\rho_{1})\geq\sqrt{\frac{2}{C}}TV(\rho_{0},\rho_{1})\qquad C=\sup_{x\in\mathbb{R}^{d}}\int_{\mathbb{R}^{d}}\eta(|x-y|)dm(y)
\]
as desired.
\end{proof}

We also have the following technical corollary, which will be used
in Proposition \ref{prop:expel infinite}.
\begin{cor}
\label{cor:H\"older-half-cts}Suppose that $C:=\sup_{x\in\mathbb{R}^{d}}\int_{\mathbb{R}^{d}}\eta(|x-y|)dm(y)<\infty$.
Let $\rho_{0},\rho_{1}\in\mathcal{P}(\mathbb{R}^{d})$, and suppose
that $\mathcal{W}_{\eta,\theta,m}(\rho_{0},\rho_{1})<\infty$. Let
$(\rho_{t},\mathbf{j}_{t})\in\mathcal{CE}(\rho_{0},\rho_{1})$ be
a constant-speed action-minimizing mass-flux pair, so that $\mathcal{W}_{\eta,\theta,m}^{2}(\rho_{t_{0}},\rho_{t_{1}})=\int_{t_{0}}^{t_{1}}\mathcal{A}(\rho_{t},\mathbf{j}_{t})dt$
for all $0\leq t_{0}<t_{1}\leq1$. Then, for any Borel $A\subset\mathbb{R}^{d}$,
the function $t\mapsto\rho_{t}(A)$ is $\frac{1}{2}$-H\"older continuous.
\end{cor}

\begin{proof}
Let $0\leq t_{0}<t_{1}\leq1$. Consider $(\rho_{t},\mathbf{j}_{t})_{t\in[t_{0},t_{1}]}$,
the $t_{0}$-to-$t_{1}$ restriction of $(\rho_{t},\mathbf{j}_{t})\in\mathcal{CE}(\rho_{0},\rho_{1})$.
Let $(\tilde{\rho}_{t},\tilde{\mathbf{j}_{t}})_{t\in[0,1]}$ denote
the uniform reparametrization of $(\rho_{t},\mathbf{j}_{t})_{t\in[t_{0},t_{1}]}$
into a unit-time solution to the nonlocal continuity equation; compute
that 
\[
\int_{t_{0}}^{t_{1}}\mathcal{A}(\rho_{t},\mathbf{j}_{t})dt=(t_{1}-t_{0})\int_{0}^{1}\mathcal{A}(\tilde{\rho}_{t},\tilde{\mathbf{j}_{t}})dt.
\]
But, $(\tilde{\rho}_{t},\tilde{\mathbf{j}_{t}})_{t\in[0,1]}$ is also
an action-minimizing solution to the nonlocal continuity equation
--- otherwise, we could locally replace $(\rho_{t},\mathbf{j}_{t})_{t\in[t_{0},t_{1}]}$
and get a lower-action mass-flux pair connecting $\rho_{0}$ and $\rho_{1}$,
which is ruled out by assumption. Therefore, by Proposition \ref{prop:TV lower bound},
\[
(t_{1}-t_{0})^{1/2}\mathcal{W}_{\theta,\eta,m}(\rho_{t_{0}},\rho_{t_{1}})\geq\sqrt{\frac{2}{C}}TV(\rho_{t_{0}},\rho_{t_{1}}).
\]
Now, let $A$ be any Borel set. Since $TV(\rho_{t_{0}},\rho_{t_{1}})=\sup_{A\in\mathcal{B}(\mathbb{R}^{d})}|\rho_{t_{0}}(A)-\rho_{t_{1}}(A)|$,
we find that
\[
(t_{1}-t_{0})^{1/2}\mathcal{W}_{\theta,\eta,m}(\rho_{t_{0}},\rho_{t_{1}})\geq\sqrt{\frac{2}{C}}|\rho_{t_{0}}(A)-\rho_{t_{1}}(A)|.
\]
Finally, since $\mathcal{W}_{\theta,\eta,m}(\rho_{t_{0}},\rho_{t_{1}})\leq\mathcal{W}_{\theta,\eta,m}(\rho_{0},\rho_{1})$,
we find that 
\[
|\rho_{t_{0}}(A)-\rho_{t_{1}}(A)|\leq\sqrt{\frac{2}{C}}\mathcal{W}_{\theta,\eta,m}(\rho_{0},\rho_{1})(t_{1}-t_{0})^{1/2}
\]
which shows that $t\mapsto\rho_{t}(A)$ is $\frac{1}{2}$-H\"older continuous,
as desired.
\end{proof}

\subsection{Expel problem for $\mathcal{W}_{\eta,\theta}$} \label{subsec:Expel-problem}

In this subsection we consider the \emph{expel problem} for nonlocal
Wasserstein distances. That is, given a Dirac mass, say at the origin $\delta_{0}$ for concreteness, we wish to estimate
$\inf_{\nu\bot\delta_{0}}\mathcal{W}_{\eta,\theta}(\delta_{0},\nu)$.
Throughout, we only consider the case where the reference measure
is the Lebesgue measure.
In this subsection, we also assume that $\eta$ satisfies Assumption
\ref{assu:eta properties} (i-v), and that $\theta$ satisfies Assumption
\ref{assu:theta properties}. 

We shall make repeated use of an adaptation of a specific computation
from \cite{maas2011gradient}, which we present separately as Lemma \ref{lem:2 point space estimate}.

\begin{lemma}[expel cost upper bound]
\label{lem:expel non-integrable}
 Let $0<\delta<\varepsilon$ where
$\varepsilon\leq1$, and let $x_{0}\in\mathbb{R}^{d}$. Suppose that
there is some constant $c_{s}$ such that $\eta(x,y)\geq c_{s}|x-y|^{-d-s}$
when $|x-y|\leq\frac{\delta}{\varepsilon}$. Let $\mathfrak{m}_{B(x_{0},\delta)}$ denote
the uniform probability measure on the ball $B(x_{0},\delta)$. Then, 
\[
\mathcal{W}_{\eta,\varepsilon}(\delta_{x_{0}},\mathfrak{m}_{B(x_{0},\delta)})\leq\frac{C_{\theta}}{C_{d,s}}\left(\frac{\delta}{\varepsilon}\right)^{s/2}
\]
where $C_{\theta}:=\int_{0}^{1}\frac{1}{\sqrt{\theta(1-r,1+r)}}dr$
and $C_{d,s}$ is given explicitly in the proof.
\end{lemma}

An important consequence of the result above is that $\mathcal{W}_{\eta,\varepsilon}(\delta_{x_{0}},\mathfrak{m}_{B(x_{0},\delta)}) \to 0$ as $\delta \to 0$ and hence the expel cost is zero: $\inf_{\nu\bot\delta_{0}}\mathcal{W}_{\eta,\theta}(\delta_{0},\nu)=0$.
\begin{proof}
Given $x_{0}\in\mathbb{R}^{d}$ and $r>s>0$ we write $\mathfrak{A}(x_{0},r)$
to denote $B(x_{0},r)\backslash B(x_{0},\frac{r}{2})$, that is, the
 annulus of outer radius $r$ and inner radius $\frac{r}{2}$.
We let $\mathfrak{m}_{\mathfrak{A}(x_{0},r)}$ denote the uniform
probability measure on $\mathfrak{A}(x_{0},r)$. 

Fix $\delta>0$. Applying Lemma \ref{lem:2 point space estimate}
with $A=\mathfrak{A}(x_{0},\delta2^{-n})$ and $B=\mathfrak{A}(x_{0},2^{-n-1})$,
we find that 
\[
\mathcal{W}_{\eta,\varepsilon}(\mathfrak{m}_{\mathfrak{A}(x_{0},\delta2^{-n})},\mathfrak{m}_{\mathfrak{A}(x_{0},\delta2^{-n-1})})\leq\frac{C_{\theta}}{4\sqrt{|\mathfrak{A}(x_{0},\delta2^{-n})|\eta_{\varepsilon}(\delta\frac{3}{2}2^{-n})}}.
\]
Let us write this upper bound in a more explicit fashion. We know
that 
\[
|\mathfrak{A}(x_{0},\delta2^{-n})|=\alpha_{d}\left(\delta2^{-n}\right)^{d}-\alpha_{d}\left(\delta2^{-n-1}\right)^{d}=\alpha_{d}\delta^{d}\left(\frac{1}{2^{nd}}-\frac{1}{2^{nd+d}}\right)
\]
where $\alpha_{d}$ is the volume of the $d$-dimensional unit ball.
On the other hand, $\eta_{\varepsilon}\left(\delta\frac{3}{2}2^{-n}\right)=\frac{1}{\varepsilon^{d}}\eta(\frac{\delta}{\varepsilon}\frac{3}{2}2^{-n})$.
Suppose now that on $B(x_{0},\frac{\delta}{\varepsilon})$, $\eta(x,y)\geq c_{s}|x-y|^{-d-s}$
where $s>0$. Then 
\begin{equation}
\eta_{\varepsilon}\left(\frac{\delta}{\varepsilon}\frac{3}{2}2^{-n}\right)\geq\frac{1}{\varepsilon^{d}}c_{s}\left(\frac{\delta}{\varepsilon}\right)^{-d-s}\left(\frac{3}{2}2^{-n}\right)^{-d-s}\label{eq:eta delta over epsilon lower bound}
\end{equation}
so that 
\begin{align*}
|\mathfrak{A}(x_{0},\delta2^{-n})|\eta_{\varepsilon}\left(\delta\frac{3}{2}2^{-n}\right) & \geq\alpha_{d}\delta^{d}\left(\frac{1}{2^{nd}}-\frac{1}{2^{nd+d}}\right)\cdot\frac{1}{\varepsilon^{d}}c_{s}\left(\frac{\delta}{\varepsilon}\right)^{-d-s}\left(\frac{3}{2}2^{n}\right)^{d+s}\\
 & =\alpha_{d}\left(\frac{\delta}{\varepsilon}\right)^{-s}2^{ns}\left(1-\frac{1}{2^{d}}\right)c_{s}\left(\frac{3}{2}\right)^{-d-s}.
\end{align*}
Putting $\tilde{C}_{d,s}=2\left(\alpha_{d}c_{s}\left(\frac{3}{2}\right)^{-d-s}\right)^{1/2}$,
this shows that 
\[
\mathcal{W}_{\eta,\varepsilon}(\mathfrak{m}_{\mathfrak{A}(x_{0},\delta2^{-n})},\mathfrak{m}_{\mathfrak{A}(x_{0},\delta2^{-n-1})})\leq\frac{C_{\theta}}{\tilde{C}_{d,s}}\left(\frac{\delta}{\varepsilon}\right)^{s/2}2^{-ns/2}.
\]
Summing the geometric series, we find 
\[
\sum_{n=0}^{\infty}\mathcal{W}_{\eta,\varepsilon}(\mathfrak{m}_{\mathfrak{A}(x_{0},\delta2^{-n})},\mathfrak{m}_{\mathfrak{A}(x_{0},\delta2^{-n-1})})\leq\frac{C_{\theta}}{\tilde{C}_{d,s}}\frac{1}{1-2^{s/2}}\left(\frac{\delta}{\varepsilon}\right)^{s/2}.
\]
It follows that $\mathcal{W}_{\eta,\varepsilon}(\delta_{x_{0}},\mathfrak{m}_{\mathfrak{A}(x_{0},\delta)})\leq\frac{C_{\theta}}{\tilde{C}_{d,s}}\frac{1}{1-2^{s/2}}\left(\frac{\delta}{\varepsilon}\right)^{s/2}$.
To see why, observe that $(\mathfrak{m}_{\mathfrak{A}(x_{0},\delta2^{-n})})_{n\in\mathbb{N}}$
converges to $\delta_{x_{0}}$ in $W_{1}$ and thus in the narrow
topology. Since $\mathcal{W}_{\eta,\varepsilon}$ is jointly l.s.c.
with respcet to the narrow topology, we find that 
\begin{align*}
\mathcal{W}_{\eta,\varepsilon}(\delta_{x_{0}},\mathfrak{m}_{\mathfrak{A}(x_{0},\delta)}) & \leq\liminf_{k\rightarrow\infty}\mathcal{W}_{\eta,\varepsilon}(\mathfrak{m}_{\mathfrak{A}(x_{0},\delta2^{-k})},\mathfrak{m}_{\mathfrak{A}(x_{0},\delta)})\\
 & \leq\liminf_{k\rightarrow\infty}\sum_{n=0}^{k}\mathcal{W}_{\eta,\varepsilon}(\mathfrak{m}_{\mathfrak{A}(x_{0},\delta2^{-n})},\mathfrak{m}_{\mathfrak{A}(x_{0},\delta2^{-n-1})})\\
 & \leq\frac{C_{\theta}}{\tilde{C}_{d,s}}\frac{1}{1-2^{s/2}}\left(\frac{\delta}{\varepsilon}\right)^{s/2}.
\end{align*}

Finally, we can easily upper bound $\mathcal{W}_{\eta,\varepsilon}(\mathfrak{m}_{\mathfrak{A}(x_{0},\delta)},\mathfrak{m}_{B(x_{0},\delta)})$.
We use yet another construction based on the $\mathcal{W}$ geodesic
in the two-point space: we use exactly the same computation as in
the proof of Lemma \ref{lem:2 point space estimate}%
. Indeed, consider the curve $\rho_{t}:[0,1]\rightarrow\mathcal{P}(\mathbb{R}^d)$
defined by 
\[
\frac{d\rho_{t}}{d\text{Leb}}(x)=\begin{cases}
\frac{1-\gamma_{t}}{2|\mathfrak{A}(x_{0},\delta)|} & x\in\mathfrak{A}(x_{0},\delta)\\
\frac{1+\gamma_{t}}{2|B(x_{0},\delta2^{-1})|} & x\in B(x_{0},\delta2^{-1})\\
0 & \text{else}.
\end{cases}
\]
Additionally, let $\mathbf{j}_{t}$ be chosen exactly as in the proof
of Lemma \ref{lem:2 point space estimate}.%
{} This $(\rho_{t},\mathbf{j}_{t})$ is constructed so that $\rho_{0}=\mathfrak{m}_{\mathfrak{A}(x_{0},\delta)}$
and $\rho_{1}=\mathfrak{m}_{B(x_{0},\frac{\delta}{2})}$, and so that
the mass on $\mathfrak{A}\left(x_{0},\delta\right)$ is decreasing
uniformly on the set, and continuously in time; therefore, there is
a $t_{0}\in(0,1)$ such that $\rho_{t_{0}}$ has uniform distribution
on $B(x_{0},\delta)$. %

In particular, it follows that 
\[
\mathcal{W}_{\eta,\varepsilon}(\mathfrak{m}_{\mathfrak{A}(x_{0},\delta)},\mathfrak{m}_{B(x_{0},\delta)})\leq\sqrt{\int_{0}^{1}\mathcal{A}(\rho_{t},\mathbf{j}_{t})dt}.
\]
Note however that $\sqrt{\int_{0}^{1}\mathcal{A}(\rho_{t},\mathbf{j}_{t})dt}$
is none other than the upper bound for $\mathcal{W}_{\eta,\varepsilon}(\mathfrak{m}_{\mathfrak{A}(x_{0},\delta)},\mathfrak{m}_{B(x_{0},\frac{\delta}{2})})$,
so we have that 
\begin{align*}
\mathcal{W}_{\eta,\varepsilon}(\mathfrak{m}_{\mathfrak{A}(x_{0},\delta)},\mathfrak{m}_{B(x_{0},\delta)})\leq\frac{C_{\theta}}{4\sqrt{|B\left(x_{0},\frac{\delta}{2}\right)|\eta_{\varepsilon}\left(\frac{3}{2}\delta\right)}} & \leq\frac{C_{\theta}}{4\sqrt{\alpha_{d}\left(\frac{\delta}{2}\right)^{d}\frac{1}{\varepsilon^{d}}c_{s}\left(\frac{3}{2}\frac{\delta}{\varepsilon}\right)^{-d-s}}}\\
 & =\frac{C_{\theta}}{2^{1-d/2}\tilde{C}_{d,s}}\left(\frac{\delta}{\varepsilon}\right)^{s/2}.
\end{align*}
Therefore, by the triangle inequality, we have that 
\[
\mathcal{W}_{\eta,\varepsilon}(\delta_{x_{0}},\mathfrak{m}_{B(x_{0},\delta)})\leq\frac{C_{\theta}}{C_{d,s}}\left(\frac{\delta}{\varepsilon}\right)^{s/2}
\]
where $C_{d,s}=\tilde{C}_{d,s}\left(\frac{1}{(1-2^{s/2})}+2^{d/2-1}\right)^{-1}.$
\end{proof}

The previous lemma computed an upper bound on the expel cost for general
interpolation $\theta$, in the case where $\eta(|\tacka|)$ is non-integrable
in $B(0,\delta)$. It is also possible to provide an expel upper bound
in the case where $\theta$ is \emph{nonzero on the boundary} --
this condition is satisfied, for instance, by the arithmetic mean,
but \emph{not} by the logarithmic mean.
\begin{lemma}
\label{lem:expel arithmetic mean}
Suppose that $\theta(1,0)=\kappa_{\theta}>0$. (Note that if $\theta(r,s)=\frac{r+s}{2}$,
then $\kappa_{\theta}=\frac{1}{2}$.) In this case, 
\[
\mathcal{W}_{\eta,\varepsilon}(\delta_{x_{0}},\mathfrak{m}_{B(x_{0},\delta)})\leq\frac{1}{\sqrt{\kappa_{\theta}\alpha_{d}\left(\frac{\delta}{\varepsilon}\right)^{d}\eta\left(\frac{\delta}{\varepsilon}\right)}}.
\]
\end{lemma}
\
\begin{proof}
Let $0<\delta<\varepsilon$. Let $g:[0,1]\rightarrow[0,1]$ be a function
to be determined later, such that $g(0)=0$ and $g(1)=1$. Let $\gamma:=\text{Leb}+\delta_{x_{0}}$.
Consider the curve $\rho_{t}:[0,1]\rightarrow\mathcal{P}(\mathbb{R}^d)$ defined
by
\[
\frac{d\rho_{t}}{d\gamma}(x)=\begin{cases}
g(t) & x=x_{0}\\
\frac{1-g(t)}{|B(x_{0},\delta)|} & x\in B(x_{0},\delta)\backslash\{x_{0}\}\\
0 & \text{else}.
\end{cases}
\]
Note that with our given boundary conditions on $g(t)$, $\rho_{0}$
is the uniform measure on $B(x_{0},\delta_{0})$, and $\rho_{1}=\delta_{x_{0}}$.
Note also that by construction, 
\[
\frac{d}{dt}\frac{d\rho_{t}}{d\gamma}(x)=\begin{cases}
g^{\prime}(t) & x=x_{0}\\
\frac{-g^{\prime}(t)}{|B(x_{0},\delta)|} & x\in B(x_{0},\delta)\backslash\\
0 & \text{else}.
\end{cases}\{x_{0}\}
\]
 Let $\mathbf{j}_{t}(x,y)$ be a flux so that $(\rho_{t},\mathbf{j}_{t})$
solves the nonlocal continuity equation; in particular we set
\[
\frac{d\mathbf{j}_{t}}{d(\gamma\otimes\gamma)}(x,y)=\begin{cases}
-\frac{g^{\prime}(t)}{2\eta_{\varepsilon}(x,y)|B(x_{0},\delta)|} & (x,y)\in\{x_{0}\}\times B(x_{0},\delta)\backslash\{x_{0}\}\\
\frac{g^{\prime}(t)}{2\eta_{\varepsilon}(x,y)|B|} & (x,y)\in B(x_{0},\delta)\backslash\{x_{0}\}\times\{x_{0}\}\\
0 & \text{else}.
\end{cases}
\]
Together, since $\gamma\otimes\gamma$ dominates all of $\mathbf{j}_{t}$,
$\rho_{t}\otimes\text{Leb}$, and $\text{Leb}\otimes\rho_{t}$, the
action of $(\rho_{t},\mathbf{j}_{t})$ is then 
\begin{align*}
\mathcal{A}(\rho_{t},\mathbf{j}_{t}) & :=\iint_{G}\frac{\left(\frac{d\mathbf{j}_{t}}{d(\gamma\otimes\gamma)}(x,y)\right)^{2}}{2\theta\left(\frac{d(\rho_{t}\otimes\text{Leb})}{d(\gamma\otimes\gamma)}(x,y),\frac{d(\text{Leb}\otimes\rho_{t})}{d(\gamma\otimes\gamma)}(x,y)\right)}\eta_{\varepsilon}(x,y)d\gamma(x)d\gamma(y).
\end{align*}
Observe that 
\[
\frac{d(\rho_{t}\otimes\text{Leb})}{d(\gamma\otimes\gamma)}=\frac{d(\rho_{t}\otimes\text{Leb})}{d((\text{Leb}+\delta_{x_{0}})\otimes(\text{Leb}+\delta_{x_{0}}))}(x,y)=\begin{cases}
\frac{d\rho_{t}}{d(\text{Leb}+\delta_{x_{0}})}(x) & y\neq x_{0}\\
0 & y=x_{0}
\end{cases}
\]
and similarly for $\frac{d(\text{Leb}\otimes\rho_{t})}{d(\gamma\otimes\gamma)}$.
Moreover, note that $\frac{d\mathbf{j}_{t}}{d(\gamma\otimes\gamma)}(x,y)=0$
off of $\{x_{0}\}\times B(x_{0},\delta)\backslash\{x_{0}\}\cup B(x_{0},\delta)\backslash\{x_{0}\}\times\{x_{0}\}$.
Consequently, 
\begin{align*}
\mathcal{A}(\rho_{t},\mathbf{j}_{t}) & =\iint_{\{x_{0}\}\times B(x_{0},\delta)\backslash\{x_{0}\}}\frac{\left(-\frac{g^{\prime}(t)}{2\eta_{\varepsilon}(x,y)|B(x_{0},\delta)|}\right)^{2}}{2\theta\left(\frac{d\rho_{t}}{d(\text{Leb}+\delta_{x_{0}})}(x)\mathbf{1}_{y\neq x_{0}},\frac{d\rho_{t}}{d(\text{Leb}+\delta_{x_{0}})}(y)\mathbf{1}_{x\neq x_{0}}\right)}\eta_{\varepsilon}(x,y)d\gamma(x)d\gamma(y)\\
 & \qquad+\iint_{B(x_{0},\delta)\backslash\{x_{0}\}\times\{x_{0}\}}\frac{\left(\frac{g^{\prime}(t)}{2\eta_{\varepsilon}(x,y)|B(x_{0},\delta)|}\right)^{2}}{2\theta\left(\frac{d\rho_{t}}{d(\text{Leb}+\delta_{x_{0}})}(x)\mathbf{1}_{y\neq x_{0}},\frac{d\rho_{t}}{d(\text{Leb}+\delta_{x_{0}})}(y)\mathbf{1}_{x\neq x_{0}}\right)}\eta_{\varepsilon}(x,y)d\gamma(x)d\gamma(y)\\
 & =\int_{B(x_{0},\delta)\backslash\{x_{0}\}}\frac{\left(\frac{g^{\prime}(t)}{2\eta_{\varepsilon}(x_{0},y)|B(x_{0},\delta)|}\right)^{2}}{\theta\left(g(t),0\right)}\eta_{\varepsilon}(x_{0},y)dy.
\end{align*}
Clearly, if it were the case that $\theta\left(g(t),0\right)=0$,
then the action of $(\rho_{t},\mathbf{j}_{t})$ would be infinite.
However, since we have instead assumed that $\theta(1,0)=\kappa_{\theta}>0$,
and so by 1-homogeneity, $\theta\left(g(t),0\right)=\kappa_{\theta}g(t)$
for all $t$, and hence
\begin{align*}
\mathcal{A}(\rho_{t},\mathbf{j}_{t}) & =\int_{B(x_{0},\delta)\backslash\{x_{0}\}}\frac{\left(\frac{g^{\prime}(t)}{2\eta_{\varepsilon}(x_{0},y)|B(x_{0},\delta)|}\right)^{2}}{\kappa_{\theta}g(t)}\eta_{\varepsilon}(x_{0},y)dy\\
 & =\frac{1}{4\kappa_{\theta}|B(x_{0},\delta)|^{2}}\int_{B(x_{0},\delta)\backslash\{x_{0}\}}\frac{(g^{\prime}(t))^{2}}{g(t)\eta_{\varepsilon}(x_{0},y)}dy\\
 & \leq\frac{1}{4\kappa_{\theta}|B(x_{0},\delta)|\eta_{\varepsilon}(\delta)}\frac{(g^{\prime}(t))^{2}}{g(t)}.
\end{align*}
Consequently, 
\[
\int_{0}^{1}\mathcal{A}(\rho_{t},\mathbf{j}_{t})\leq\frac{1}{4\kappa_{\theta}|B(x_{0},\delta)|\varepsilon^{-d}\eta\left(\frac{\delta}{\varepsilon}\right)}\int_{0}^{1}\frac{(g^{\prime}(t))^{2}}{g(t)}dt.
\]
Finally, we select $g(t)=t^{2}$. With this choice, $\int_{0}^{1}\frac{(g^{\prime}(t))^{2}}{g(t)}dt=4$.
And since $|B(x_{0},\delta)|=\alpha_{d}\delta^{d}$, we conclude
that 
\[
\mathcal{W}_{\eta,\varepsilon}(\delta_{x_{0}},\mathfrak{m}_{B(x_{0},\delta)})\leq\frac{1}{\sqrt{\kappa_{\theta}\alpha_{d}\left(\frac{\delta}{\varepsilon}\right)^{d}\eta\left(\frac{\delta}{\varepsilon}\right)}}
\]
as desired.
\end{proof}

\begin{prop}
\label{prop:expel infinite}If $\int_{\mathbb{R}^{d}}\eta(|y|)dy<\infty$,
and $\theta(1,0)=0$, then for all probability measures $\nu\in\mathcal{P}(\mathbb{R}^{d})$
which are singular to $\delta_{x_{0}}$, $\mathcal{W}_{\eta,\theta}(\delta_{x_{0}},\nu)=\infty$.
\end{prop}

\begin{proof}
Let $(\rho_{t},\mathbf{j}_{t})_{t\in[0,1]}$ solve the nonlocal continuity
equation, and let $\rho_{0}=\delta_{0}$ and $\rho_{1}=\nu$. We assume
for the sake of contradiction that $\int_{0}^{1}\mathcal{A}(\rho_{t},\mathbf{j}_{t})dt<\infty$.

We define the set 
\[
\mathfrak{T}=\{t\in[0,1]\::\: \rho_{t}(\{0\})>0 \; \wedge \; |\mathbf{j}_{t}|(\{0\}\times\mathbb{R}^{d}\backslash\{0\})>0\}.
\]
Note that this set is measurable since $\rho_{t}:[0,1]\rightarrow\mathcal{P}(\mathbb{R}^{d})$
is narrowly continuous and $\mathbf{j}_{t}:[0,1]\rightarrow\mathcal{M}_{loc}(G)$
is a Borel function. We claim that for any $t\in\mathfrak{T}$, it
holds that 
\[
\mathcal{A}(\rho_{t},\mathbf{j}_{t})=\infty.
\]
To see this, let $\lambda_{t}\in\mathcal{M}^{+}(G)$ be any measure
such that $\rho_{t}\otimes\text{Leb}+\text{Leb}\otimes\rho_{t}+|\mathbf{j}_{t}|\ll\lambda_{t}$.
In particular, if $\rho_{t}(\{0\})\neq0$ (and, so, for any $t\in\mathfrak{T}$),
the fact that $\rho_{t}\otimes\text{Leb}\ll\lambda_{t}$ implies that
$\lambda_{t}\upharpoonright\{0\}\times\mathbb{R}^{d}$ is not identically
zero. 

At the same time, compute that 
\begin{align*}
\frac{d(\text{Leb}\otimes\rho_{t})}{d\lambda_{t}}(x,y) & =\frac{d(\text{Leb}\otimes\rho_{t})}{d(\text{Leb}\otimes\rho_{t}+\rho_{t}\otimes\text{Leb})}(x,y)\frac{d(\text{Leb}\otimes\rho_{t}+\rho_{t}\otimes\text{Leb})}{d\lambda_{t}}(x,y)
\end{align*}
Note that for all $t\in\mathfrak{T}$ we may select a representative
of $\frac{d(\text{Leb}\otimes\rho_{t})}{d(\text{Leb}\otimes\rho_{t}+\rho_{t}\otimes\text{Leb})}$
so that
\[
\frac{d(\text{Leb}\otimes\rho_{t})}{d(\text{Leb}\otimes\rho_{t}+\rho_{t}\otimes\text{Leb})}(0,y)=0\quad\forall y\in\mathbb{R}^{d}
\]
since for all $t\in\mathfrak{T}$ and open, bounded $U\subset\mathbb{R}^{d}$,
\[
\rho_{t}\otimes\text{Leb}(\{0\}\times U)=\rho_{t}(\{0\})\text{Leb}(U)>0
\]
 and so $(\text{Leb}\otimes\rho_{t}+\rho_{t}\otimes\text{Leb})(\{0\}\times U)>0$,
but $\text{Leb}\otimes\rho_{t}(\{0\}\times U)=0$. This implies that
(up to a choice of a.e.-equivalent representative) $\frac{d\text{Leb}\otimes\rho_{t}}{d\lambda_{t}}(0,y)=0$
for all $y\in\mathbb{R}^{d}$. 

Therefore, compute as follows: for all $t\in\mathfrak{T}$,
\begin{align*}
\mathcal{A}_{\eta,\varepsilon}(\rho_{t},\mathbf{j}_{t}) & =\iint_{G}\frac{\left(\frac{d\mathbf{j}_{t}}{d\lambda_{t}}(x,y)\right)^{2}}{2\theta\left(\frac{d(\rho_{t}\otimes\text{Leb})}{d\lambda_{t}}(x,y),\frac{d(\text{Leb}\otimes\rho_{t})}{d\lambda_{t}}(x,y)\right)}\eta_{\varepsilon}(x,y)d\lambda_{t}(x,y)\\
 & \geq\int_{\{0\}\times\mathbb{R}^{d}\backslash\{0\}}\frac{\left(\frac{d\mathbf{j}_{t}}{d\lambda_{t}}(x,y)\right)^{2}}{2\theta\left(\frac{d(\rho_{t}\otimes\text{Leb})}{d\lambda_{t}}(x,y),\frac{d(\text{Leb}\otimes\rho_{t})}{d(\lambda_{t})}(x,y)\right)}\eta_{\varepsilon}(x,y)d\lambda_{t}(x,y)\\
 & =\int_{\{0\}\times\mathbb{R}^{d}\backslash\{0\}}\frac{\left(\frac{d\mathbf{j}_{t}}{d\lambda_{t}}(x,y)\right)^{2}}{2\theta\left(\frac{d(\rho_{t}\otimes\text{Leb})}{d\lambda_{t}}(x,y),0\right)}\eta_{\varepsilon}(x,y)d\lambda_{t}(x,y)\\
 & =\infty.
\end{align*}
Therefore, if $\int_{0}^{1}\mathcal{A}(\rho_{t},\mathbf{j}_{t})<\infty$,
it must be the case that $\text{Leb}(\mathfrak{T})=0$. 

However, we claim that $\text{Leb}(\mathfrak{T})>0$. Indeed, consider
the following. Let $\xi \in C^\infty_c(\R^d, [0,1])$ such that $\xi(0)=1$ and let $\xi_{\delta}(x):=\xi\left(\frac{x}{\delta}\right)$.
Note that $0\leq\xi_{\delta}(x)\leq1$, and that as $\delta\rightarrow0$, $\xi_{\delta}(x)$ converges pointwise
to the indicator $1_{\{x=0\}}$. Plugging $\xi_{\delta}(x)$
into the continuity equation, we find that 
\[
\int_{\mathbb{R}^{d}}\xi_{\delta}(x)d\rho_{t}(x)-\int_{\mathbb{R}^{d}}\xi_{\delta}(x)d\rho_{0}(x)=-\int_{0}^{t}\iint_{G}(\xi_{\delta}(y)-\xi_{\delta}(x))\eta(x,y)d\mathbf{j}_{t}(x,y)dt
\]
and so 
\begin{align*}
\left|\int_{\mathbb{R}^{d}}\xi_{\delta}(x)d\rho_{t}(x)-\int_{\mathbb{R}^{d}}\xi_{\delta}(x)d\rho_{0}(x)\right| & \leq\int_{0}^{t}\iint_{G}|\xi_{\delta}(y)-\xi_{\delta}(x)|\eta(x,y)d|\mathbf{j}_{t}|(x,y)dt.\\
 & :=\iint_{G}|\xi_{\delta}(y)-\xi_{\delta}(x)|\frac{d|\mathbf{j}_{t}|}{d\lambda_{t}}(x,y)\eta(x,y)d\lambda_{t}(x,y)dt.
\end{align*}
Using the fact that $\frac{d|\mathbf{j}_{t}|}{d\lambda_{t}}=\left|\frac{d\mathbf{j}_{t}}{d\lambda_{t}}\right|$,
and observing that 
\[
\left|\frac{d\mathbf{j}_{t}}{d\lambda_{t}}\right|=\sqrt{2\theta\left(\frac{d\rho_{t}\otimes\text{Leb}}{d\lambda_{t}},\frac{d\text{Leb}\otimes\rho_{t}}{d\lambda_{t}}\right)}\sqrt{\frac{\left(\frac{d\mathbf{j}_{t}}{d\lambda_{t}}\right)^{2}}{\theta\left(\frac{d(\rho_{t}\otimes\text{Leb})}{d\lambda_{t}},\frac{d(m\otimes\rho_{t})}{d\lambda_{t}}\right)}}
\]
whenever the right hand side is well defined, we deduce (now using
the convention that $0\cdot\infty=\infty$) that 
\begin{multline*}
\int_{0}^{t}\iint_{G}|\xi_{\delta}(y)-\xi_{\delta}(x)|\frac{d|\mathbf{j}_{t}|}{d\lambda_{t}}\eta d\lambda_{t}(x,y) dt\leq\\
\int_{0}^{t}\iint_{G}|\xi_{\delta}(y)-\xi_{\delta}(x)|\sqrt{2\theta\left(\frac{d\rho_{t}\otimes\text{Leb}}{d\lambda_{t}},\frac{d\text{Leb}\otimes\rho_{t}}{d\lambda_{t}}\right)}\sqrt{\frac{\left(\frac{d\mathbf{j}_{t}}{d\lambda_{t}}\right)^{2}}{\theta\left(\frac{d(\rho_{t}\otimes\text{Leb})}{d\lambda_{t}},\frac{d(\text{Leb}\otimes\rho_{t})}{d\lambda_{t}}\right)}}\eta\, d\lambda_{t}(x,y) dt.
\end{multline*}
Applying H\"older's inequality, and using the fact that $\theta(r,s)\leq\frac{r+s}{2}$,
we find that this last expression is bounded above by 
\begin{multline*}
\left(\int_{0}^{t}\iint_{G}\left(\frac{d\rho_{t}\otimes\text{Leb}}{d\lambda_{t}}(x,y)+\frac{d\text{Leb}\otimes\rho_{t}}{d\lambda_{t}}(x,y)\right)\eta(x,y)d\lambda_{t}(x,y)dt\right)^{1/2}\\
\times\left(\int_{0}^{t}\iint_{G}\frac{(\xi_{\delta}(y)-\xi_{\delta}(x))^{2}\left(\frac{d\mathbf{j}_{t}}{d\lambda_{t}}(x,y)\right)^{2}}{2\theta\left(\frac{d(\rho_{t}\otimes\text{Leb})}{d\lambda_{t}}(x,y),\frac{d(\text{Leb}\otimes\rho_{t})}{d\lambda_{t}}(x,y)\right)}\eta_{\varepsilon}(x,y)d\lambda_{t}(x,y)dt\right)^{1/2}.
\end{multline*}
The term in the left parentheses is bounded above, in turn, by $2t\int_{\mathbb{R}^{d}}\eta(|y|)dy$.
To see this, compute that 
\begin{align*}
\int_{0}^{t}\iint_{G}\frac{d\rho_{t}\otimes\text{Leb}}{d\lambda_{t}}(x,y)\eta(x,y)d\lambda_{t}(x,y)dt & =\int_{0}^{t}\iint_{G}\eta(x,y)d(\rho_{t}\otimes\text{Leb})(x,y)dt\\
 & \leq\int_{0}^{t}\int_{\mathbb{R}^{d}}\left(\int_{\mathbb{R}^{d}}\eta(x,y)dy\right)d\rho_{t}(x)dt\\
 & =tC.
\end{align*}
The computation for $\int_{0}^{t}\iint_{G}\frac{\text{Leb}\otimes\rho_{t}}{d\lambda_{t}}(x,y)\eta(x,y)d\lambda_{t}(x,y)dt$
is identical. Therefore, we find that 
\begin{multline*}
\left|\int_{\mathbb{R}^{d}}\xi_{\delta}(x)d\rho_{t}(x)-\int_{\mathbb{R}^{d}}\xi_{\delta}(x)d\rho_{0}(x)\right|\leq\\
\sqrt{2tC}\left(\int_{0}^{t}\iint_{G}\frac{(\xi_{\delta}(y)-\xi_{\delta}(x))^{2}\left(\frac{d\mathbf{j}_{t}}{d\lambda_{t}}(x,y)\right)^{2}}{2\theta\left(\frac{d(\rho_{t}\otimes\text{Leb})}{d\lambda_{t}}(x,y),\frac{d(\text{Leb}\otimes\rho_{t})}{d\lambda_{t}}(x,y)\right)}\eta_{\varepsilon}(x,y)d\lambda_{t}(x,y)dt\right)^{1/2}.
\end{multline*}
Now, since $(\xi_{\delta}(y)-\xi_{\delta}(x))^{2}\leq1$, and $\frac{\left(\frac{d\mathbf{j}_{t}}{d\lambda_{t}}(x,y)\right)^{2}}{2\theta\left(\frac{d(\rho_{t}\otimes\text{Leb})}{d\lambda_{t}}(x,y),\frac{d(\text{Leb}\otimes\rho_{t})}{d\lambda_{t}}(x,y)\right)}\eta_{\varepsilon}(x,y)$
is integrable with respect to $d\lambda_{t}(x,y)dt$ on $G\times[0,t]$
(assuming that $\int_{0}^{t}\mathcal{A}(\rho_{t},\mathbf{j}_{t})dt<\infty$),
and $(\xi_{\delta}(y)-\xi_{\delta}(x))^{2}$ converges pointwise to
$1_{\{0\}\times\mathbb{R}^{d}\backslash\{0\}}$, we can apply the
dominated convergence theorem to deduce that as $\delta\rightarrow0$,

\begin{multline*}
\int_{0}^{t}\iint_{G}\frac{(\xi_{\delta}(y)-\xi_{\delta}(x))^{2}\left(\frac{d\mathbf{j}_{t}}{d\lambda_{t}}(x,y)\right)^{2}}{2\theta\left(\frac{d(\rho_{t}\otimes\text{Leb})}{d\lambda_{t}}(x,y),\frac{d(\text{Leb}\otimes\rho_{t})}{d\lambda_{t}}(x,y)\right)}\eta_{\varepsilon}(x,y)d\lambda_{t}(x,y)dt\\
\longrightarrow \int_{0}^{t}\iint_{\{0\}\times\mathbb{R}^{d}\backslash\{0\}}\frac{\left(\frac{d\mathbf{j}_{t}}{d\lambda_{t}}(x,y)\right)^{2}}{2\theta\left(\frac{d(\rho_{t}\otimes\text{Leb})}{d\lambda_{t}}(x,y),\frac{d(\text{Leb}\otimes\rho_{t})}{d\lambda_{t}}(x,y)\right)}\eta_{\varepsilon}(x,y)d\lambda_{t}(x,y)dt.
\end{multline*}
At the same time, as $\delta\rightarrow0$, $\left|\int_{\mathbb{R}^{d}}\xi_{\delta}(x)d\rho_{t}(x)-\int_{\mathbb{R}^{d}}\xi_{\delta}(x)d\rho_{0}(x)\right|\rightarrow\left|\rho_{t}(\{0\})-\rho_{0}(\{0\})\right|$.
So in the limit we find that 
\[
\left|\rho_{t}(\{0\})-\rho_{0}(\{0\})\right|\leq\sqrt{2tC}\left(\int_{0}^{t}\iint_{\{0\}\times\mathbb{R}^{d}\backslash\{0\}}\frac{\left(\frac{d\mathbf{j}_{t}}{d\lambda_{t}} \right)^{2}}{2\theta\left(\frac{d(\rho_{t}\otimes\text{Leb})}{d\lambda_{t}} ,\frac{d(\text{Leb}\otimes\rho_{t})}{d\lambda_{t}} \right)}\eta_{\varepsilon}(x,y)d\lambda_{t}(x,y)dt\right)^{1/2}.
\]
In particular, if $t$ is taken to be the first time such that $\rho_{t}(\{0\})=0$ (note that such a first $t$ exists, thanks to Corollary \ref{cor:H\"older-half-cts}), then we find that 
\[
\frac{1}{\sqrt{2tC}}<\int_{0}^{t}\iint_{\{0\}\times\mathbb{R}^{d}\backslash\{0\}}\frac{\left(\frac{d\mathbf{j}_{t}}{d\lambda_{t}}(x,y)\right)^{2}}{2\theta\left(\frac{d(\rho_{t}\otimes\text{Leb})}{d\lambda_{t}}(x,y),\frac{d(\text{Leb}\otimes\rho_{t})}{d\lambda_{t}}(x,y)\right)}\eta_{\varepsilon}(x,y)d\lambda_{t}(x,y)dt
\]
which implies, in particular, that on a subset of $[0,t)$ of positive
measure, 
\[
\frac{1}{\sqrt{2tC}}<\iint_{\{0\}\times\mathbb{R}^{d}\backslash\{0\}}\frac{\left(\frac{d\mathbf{j}_{t}}{d\lambda_{t}}(x,y)\right)^{2}}{2\theta\left(\frac{d(\rho_{t}\otimes\text{Leb})}{d\lambda_{t}}(x,y),\frac{d(\text{Leb}\otimes\rho_{t})}{d\lambda_{t}}(x,y)\right)}\eta_{\varepsilon}(x,y)d\lambda_{t}(x,y)
\]
which means that on a subset of $[0,t)$ of positive measure, $|\mathbf{j}_{t}|(\{0\}\times\mathbb{R}^{d}\backslash\{0\})>0$.
However, for \emph{all} of $[0,t)$, we know that $\rho_{t}(\{0\})>0$.
Therefore, $\text{Leb}(\mathfrak{T})>0$.

\end{proof}

\subsection{Global upper bounds for $\mathcal{W}_{\eta,\theta}$}\label{subsec:Global-upper-bounds} Throughout this subsection, we assume that $\eta$ satisfies Assumption
\ref{assu:eta properties} (i-v), and that $\theta$ satisfies Assumption
\ref{assu:theta properties}. 

We record the following estimate, which will ultimately be used to show that on a compact domain, the topology induced by $\mathcal{W}_{\eta,\theta}$ is no stronger than the narrow topology.

\begin{lemma}
\label{lem:topological upper bound non-integrable}Suppose that there
is some constant $c_{s}$ such that $\eta(x,y)\geq c_{s}|x-y|^{-d-s}$
when $|x-y|\leq\frac{1}{2}$. Then, 
\[
\mathcal{W}_{\eta,\varepsilon}(\delta_{x},\delta_{y})\leq\Phi\left(\frac{|x-y|}{\varepsilon}\right)
\]
where 
\[
\Phi(t)=\begin{cases}
C_{d,\theta,s}t^{s/2} & 0\leq t<\frac{3}{8}\\
C_{d,\theta,s}\left(\frac{3}{8}\right)^{s/2}\frac{8}{3}t & \frac{3}{8}\leq t
\end{cases}
\]
and $C_{d,\theta,s}$ depends only on $d,\theta,$ and $s$, and is
explicitly given in the proof.
\end{lemma}

\begin{proof}
Let $\mathfrak{m}_{B(x_{0},\delta)}$ denote the uniform probability
measure on the ball $B(x_{0},\delta)$. Then, 
\[
\mathcal{W}_{\eta,\varepsilon}(\delta_{x},\delta_{y})\leq\mathcal{W}_{\eta,\varepsilon}(\delta_{x},\mathfrak{m}_{B(x,\delta)})+\mathcal{W}_{\eta,\varepsilon}(\mathfrak{m}_{B(x,\delta)},\mathfrak{m}_{B(y,\delta)})+\mathcal{W}_{\eta,\varepsilon}(\mathfrak{m}_{B(y,\delta)},\delta_{y}).
\]
In order to estimate $\mathcal{W}_{\eta,\varepsilon}(\mathfrak{m}_{B(x,\delta)},\mathfrak{m}_{B(y,\delta)})$,
we use Lemma \ref{lem:2 point space estimate}. Indeed, applying
Lemma \ref{lem:2 point space estimate} in the case where $A=B(x_{0},\delta)$
and $B=B(x_{1},\delta)$, and $|x_{0}-x_{1}|+2\delta\leq\frac{\varepsilon}{2}$,
it follows that
\[
\mathcal{W}_{\varepsilon,\eta}(\mathfrak{m}_{B(x_{0},\delta)},\mathfrak{m}_{B(x_{1},\delta)})\leq\frac{C_{\theta}}{4\sqrt{\alpha_{d}\delta^{d}\eta_{\varepsilon}(|x_{0}-x_{1}|+2\delta)}}.
\]
In particular, since $\eta(|x-y|)\geq c_{s}|x-y|^{-d-s}$, we find
that 
\[
\mathcal{W}_{\varepsilon,\eta}(\mathfrak{m}_{B(x_{0},\delta)},\mathfrak{m}_{B(x_{1},\delta)})\leq\frac{C_{\theta}}{4\sqrt{\alpha_{d}\left(\frac{\delta}{\varepsilon}\right)^{d}c_{s}\left(\frac{|x_{0}-x_{1}|+2\delta}{\varepsilon}\right)^{-d-s}}}.
\]
Note that $2\delta<|x_{0}-x_{1}|\leq\frac{\varepsilon}{2}-2\delta$
but otherwise $\delta$ is arbitrary. In particular, if we pick $\delta=\frac{1}{6}|x_{0}-x_{1}|$,
this leads to the constraint on $x_{0}$ and $x_{1}$ that $|x_{0}-x_{1}|<\frac{3}{8}\varepsilon$,
and the estimate 
\[
\mathcal{W}_{\varepsilon,\eta}(\mathfrak{m}_{B(x_{0},\delta)},\mathfrak{m}_{B(x_{1},\delta)})\leq\frac{C_{\theta}}{4\sqrt{\alpha_{d}c_{s}8^{-d-s}}}\left(\frac{\frac{4}{3}|x_{0}-x_{1}|}{\varepsilon}\right)^{s/2}\text{ when }|x_{0}-x_{1}|<\frac{3}{8}\varepsilon.
\]
At the same time, we know from Lemma \ref{lem:expel non-integrable}
that
\[
\mathcal{W}_{\eta,\varepsilon}(\delta_{x_{0}},\mathfrak{m}_{B(x_{0},\delta)})\leq\frac{C_{\theta}}{C_{d,s}}\left(\frac{\frac{1}{6}|x_{0}-x_{1}|}{\varepsilon}\right)^{s/2}
\]
and similarly for $x_{1}$, so altogether, 
\[
\mathcal{W}_{\eta,\varepsilon}(\delta_{x_{0}},\delta_{x_{1}})\leq\frac{1}{2}C_{d,\theta,s}\left(\frac{|x_{0}-x_{1}|}{\varepsilon}\right)^{s/2}\text{ when }|x_{0}-x_{1}|<\frac{3}{8}\varepsilon
\]
where 
\[
C_{d,\theta,s}=\frac{C_{\theta}}{2\sqrt{\alpha_{d}c_{s}8^{-d-s}}}\left(\frac{4}{3}\right)^{s/2}+4\frac{C_{\theta}}{C_{d,s}}\frac{1}{6^{s/2}}.
\]

For arbitrary $x,y\in\mathbb{R}^{d}$, it suffices to repeatedly apply
the triangle inequality. Namely, construct a sequence $x=x_{0},x_{1},\ldots,x_{k}=y$
so that $|x_{i}-x_{i+1}|<\frac{3}{8}\varepsilon$ for each $i\in\{0,\ldots,k-1\}$;
note in particular we can take 
\[
k=\left\lceil \frac{|x-y|}{\frac{3\varepsilon}{8}}\right\rceil .
\]
Thus, for $|x-y|\geq\frac{3}{8}\varepsilon$, we have that 
\begin{align*}
\mathcal{W}_{\eta,\varepsilon}(\delta_{x},\delta_{y})\leq\sum_{i=1}^{k-1}\mathcal{W}_{\eta,\varepsilon}(\delta_{x_{i}},\delta_{x_{i+1}})\leq k\frac{C_{d,\theta,s}}{2}\left(\frac{3}{8}\right)^{s/2} & \leq\frac{C_{d,\theta,s}}{2}\left(\frac{3}{8}\right)^{s/2}\left(\frac{|x-y|}{\frac{3\varepsilon}{8}}+1\right)\\
 & \leq C_{d,\theta,s}\left(\frac{3}{8}\right)^{s/2}\frac{8}{3}\frac{|x-y|}{\varepsilon}.
\end{align*}
Note that when $|x-y|=\frac{3}{8}\varepsilon$, we have $C_{d,\theta,s}\left(\frac{3}{8}\right)^{s/2}\frac{8}{3}\frac{|x-y|}{\varepsilon}=C_{d,\theta,s}\left(\frac{3}{8}\right)^{s/2}$,
while $\frac{1}{2}C_{d,\theta,s}\left(\frac{|x_{0}-x_{1}|}{\varepsilon}\right)^{s/2}=\frac{1}{2}C_{d,\theta,s}\left(\frac{3}{8}\right)^{s/2}$;
so by defining the continuous, nondecreasing function
\[
\Phi_{\varepsilon}:[0,\infty)\rightarrow[0,\infty)
\]
\[
\Phi(t)=\begin{cases}
C_{d,\theta,s}t^{s/2} & 0\leq t<\frac{3}{8}\\
C_{d,\theta,s}\left(\frac{3}{8}\right)^{s/2}\frac{8}{3}t & \frac{3}{8}\leq t
\end{cases}
\]
we see that 
\[
\mathcal{W}_{\eta,\varepsilon}(\delta_{x},\delta_{y})\leq\Phi\left(\frac{|x-y|}{\varepsilon}\right).
\]

\end{proof}
More generally, one has the following upper bound:
\begin{lemma}
\label{lem:nonlocal dirac crude upper bound} Let $\varepsilon > 0$. Suppose that  $\eta(x,y)\geq c_{s}|x-y|^{-d-s}$
when $|x-y|\leq\frac{1}{6}$ or $\theta(1,0)=\kappa_{\theta}>0$. Then for all $x,y\in\mathbb{R}^{d}$, 
\[
\mathcal{W}_{\eta,\varepsilon}(\delta_{x},\delta_{y})\leq\frac{C_{d,\theta}}{\sqrt{\eta\left(\frac{1}{2}\right)}}\frac{1}{\varepsilon}|x-y|+C_{d,\theta,\eta}
\]
where $C_{d,\theta,\eta}$ is an explicit constant given in the proof.
\end{lemma}
\begin{proof}

To estimate the distance between delta masses at $x$ and $y$ when $|x-y| \gg \veps$,  we first spread the mass to a ball of width $\delta$, comparable to $\epsilon$, and then jump between identical balls placed at distance comparable to $\varepsilon$ along the line segment between $x$ and $y$. The number of the balls is comparable to $|x-y|/\veps$, which explains the scaling of the right hand side in our estimate.

Let $\mathfrak{m}_{B(x_{0},\delta)}$ denote the uniform probability
measure on the ball $B(x_{0},\delta)$. Then, 
\[
\mathcal{W}_{\eta,\varepsilon}(\delta_{x},\delta_{y})\leq\mathcal{W}_{\eta,\varepsilon}(\delta_{x},\mathfrak{m}_{B(x,\delta)})+\mathcal{W}_{\eta,\varepsilon}(\mathfrak{m}_{B(x,\delta)},\mathfrak{m}_{B(y,\delta)})+\mathcal{W}_{\eta,\varepsilon}(\mathfrak{m}_{B(y,\delta)},\delta_{y}).
\]
In other to estimate $\mathcal{W}_{\eta,\varepsilon}(\mathfrak{m}_{B(x,\delta)},\mathfrak{m}_{B(y,\delta)})$,
we use Lemma \ref{lem:2 point space estimate}. Indeed, applying
Lemma \ref{lem:2 point space estimate} in the case where $A=B(x_{0},\delta)$
and $B=B(x_{1},\delta)$, and $|x_{0}-x_{1}|+2\delta\leq\frac{\varepsilon}{2}$,
it follows that
\[
\mathcal{W}_{\eta,\varepsilon}(\mathfrak{m}_{B(x_{0},\delta)},\mathfrak{m}_{B(x_{1},\delta)})\leq\frac{C_{\theta}}{4\sqrt{\alpha_{d}\delta^{d}\eta_{\varepsilon}(|x_{0}-x_{1}|+2\delta)}}\leq\frac{C_{\theta}}{4\sqrt{\alpha_{d}\delta^{d}\eta_{\varepsilon}(\varepsilon/2)}}.
\]
Then, in order to estimate $\mathcal{W}_{\eta,\varepsilon}(\mathfrak{m}_{B(x,\delta)},\mathfrak{m}_{B(y,\delta)})$,
it suffices to select a sequence $x=x_{0},x_{1},\ldots,x_{k}=y$ so
that $\delta<|x_{i}-x_{i+1}|\leq\frac{\varepsilon}{2}-2\delta$ for
all $i\in\{0,\ldots,k-1\}$; in particular, we can take 
\[
k=\left\lceil \frac{|x-y|}{\frac{\varepsilon}{2}-2\delta}\right\rceil .
\]
Hence, 
\[
\mathcal{W}_{\eta,\varepsilon}(\mathfrak{m}_{B(x,\delta)},\mathfrak{m}_{B(y,\delta)})\leq k\frac{C_{\theta}}{4\sqrt{\alpha_{d}\delta^{d}\eta_{\varepsilon}(\varepsilon/2)}}=\left(\left\lceil \frac{|x-y|}{\frac{\varepsilon}{2}-2\delta}\right\rceil \right)\frac{C_{\theta}}{4\sqrt{\alpha_{d}\delta^{d}\eta_{\varepsilon}(\varepsilon/2)}}.
\]
Now, note that 
\[
\eta_{\varepsilon}(\frac{\varepsilon}{2})=\frac{1}{\varepsilon^{d}}\eta\left(\frac{\frac{\varepsilon}{2}}{\varepsilon}\right)=\frac{1}{\varepsilon^{d}}\eta\left(\frac{1}{2}\right).
\]
For convenience, we also select $\delta=\frac{\varepsilon}{6}$. Plugging
this in, we get 
\begin{align*}
\mathcal{W}_{\eta,\varepsilon}(\mathfrak{m}_{B(x,\delta)},\mathfrak{m}_{B(y,\delta)}) & \leq\left(\left\lceil \frac{|x-y|}{\frac{\varepsilon}{6}}\right\rceil \right)\frac{1}{4\sqrt{\alpha_{d}}\left(\frac{1}{6}\right)^{d/2}}\frac{C_{\theta}}{\sqrt{\eta\left(\frac{1}{2}\right)}}.\\
 & \leq\left(\frac{|x-y|+\frac{\varepsilon}{6}}{\frac{\varepsilon}{6}}\right)\frac{1}{4\sqrt{\alpha_{d}}\left(\frac{1}{6}\right)^{d/2}}\frac{C_{\theta}}{\sqrt{\eta\left(\frac{1}{2}\right)}}
\end{align*}
In other words, putting 
\[
C_{d,\theta}=\frac{C_{\theta}}{\frac{2}{3}\sqrt{\alpha_{d}}\left(\frac{1}{6}\right)^{d/2}}
\]
we see that 
\begin{align*}
\mathcal{W}_{\eta,\varepsilon}(\mathfrak{m}_{B(x,\delta)},\mathfrak{m}_{B(y,\delta)}) & \leq\frac{C_{d,\theta}}{\sqrt{\eta\left(\frac{1}{2}\right)}}\frac{1}{\varepsilon}|x-y|+\frac{1}{6}\frac{C_{d,\theta}}{\sqrt{\eta\left(\frac{1}{2}\right)}}.
\end{align*}

Finally, we can compute as follows: 
\begin{align*}
\mathcal{W}_{\eta,\varepsilon}(\delta_{x},\delta_{y}) & \leq\mathcal{W}_{\eta,\varepsilon}(\delta_{x},\mathfrak{m}_{B(x,\delta)})+\mathcal{W}_{\eta,\varepsilon}(\mathfrak{m}_{B(x,\delta)},\mathfrak{m}_{B(y,\delta)})+\mathcal{W}_{\eta,\varepsilon}(\mathfrak{m}_{B(y,\delta)},\delta_{y})\\
 & \leq\mathcal{W}_{\eta,\varepsilon}(\mathfrak{m}_{B(x,\delta)},\mathfrak{m}_{B(y,\delta)})+2\cdot\mathcal{W}_{\eta,\varepsilon}(\delta_{x},\mathfrak{m}_{B(x,\delta)})\\
 & \leq\frac{C_{d,\theta}}{\sqrt{\eta\left(\frac{1}{2}\right)}}\frac{1}{\varepsilon}|x-y|+\frac{1}{6}\frac{C_{d,\theta}}{\sqrt{\eta\left(\frac{1}{2}\right)}}+2\mathcal{W}_{\eta,\varepsilon}\left(\delta_{x},\mathfrak{m}_{B\left(x,\frac{\varepsilon}{6}\right)}\right)
\end{align*}
where we have used the fact that $\delta=\frac{\varepsilon}{6}$.
In particular, applying Lemma \ref{lem:expel non-integrable}, we
find that if $\eta(x,y)\geq c_{s}|x-y|^{-d-s}$ when $|x-y|\leq\frac{1}{6}$,
it holds that 
\[
\mathcal{W}_{\eta,\varepsilon}(\delta_{x},\delta_{y})\leq\frac{C_{d,\theta}}{\sqrt{\eta\left(\frac{1}{2}\right)}}\frac{1}{\varepsilon}|x-y|+\frac{1}{6}\frac{C_{d,\theta}}{\sqrt{\eta\left(\frac{1}{2}\right)}}+2\frac{C_{\theta}}{C_{d,s}}\left(\frac{1}{6}\right)^{s/2};
\]
applying Lemma \ref{lem:expel arithmetic mean}, we find that if $\theta(1,0)=\kappa_{\theta}>0$,
then 
\[
\mathcal{W}_{\eta,\varepsilon}(\delta_{x},\delta_{y})\leq\frac{C_{d,\theta}}{\sqrt{\eta\left(\frac{1}{2}\right)}}\frac{1}{\varepsilon}|x-y|+\frac{1}{6}\frac{C_{d,\theta}}{\sqrt{\eta\left(\frac{1}{2}\right)}}+\frac{2}{\sqrt{\kappa_{\theta}\alpha_{d}\left(\frac{1}{6}\right)^{d}\eta\left(\frac{1}{6}\right)}}.
\]
So altogether, we have that 
\[
\mathcal{W}_{\eta,\varepsilon}(\delta_{x},\delta_{y})\leq\frac{C_{d,\theta}}{\sqrt{\eta\left(\frac{1}{2}\right)}}\frac{1}{\varepsilon}|x-y|+C_{d,\theta,\eta}
\]
where we have the case-wise definition of $C_{d,\theta,\eta}$ (if
both conditions obtain, either case can be chosen for the value of
$C_{d,\theta,\eta}$): 
\[
C_{d,\theta,\eta}:=\begin{cases}
\frac{1}{6}\frac{C_{d,\theta}}{\sqrt{\eta\left(\frac{1}{2}\right)}}+2\frac{C_{\theta}}{C_{d,s}}\left(\frac{1}{6}\right)^{s/2} & \begin{aligned}\eta(x,y)\geq c_{s}|x-y|^{-d-s}\\
\text{ when }|x-y|\leq\frac{1}{6};
\end{aligned}
\\
\frac{1}{6}\frac{C_{d,\theta}}{\sqrt{\eta\left(\frac{1}{2}\right)}}+\frac{2}{\sqrt{\kappa_{\theta}\alpha_{d}\left(\frac{1}{6}\right)^{d}\eta\left(\frac{1}{6}\right)}} & \theta(1,0)=\kappa_{\theta}>0.
\end{cases}
\]
\end{proof}
In order to proceed, we prove the following disintegration inequality
for the $\mathcal{W}_{\eta}$ metric, which is of independent interest,
in addition to being needed in the proof of Lemma \ref{lem:crude W2 upper bound}
below. An analogous result was established in \cite[Proposition 2.14]{erbar2012ricci}
in the discrete case; however, their proof does not readily adapt
to our continuum setting.
\begin{thm}[Disintegration inequality]
\label{thm:disintegration-inequality} Let $\mu,\nu\in\mathcal{P}(\mathbb{R}^{d})$.
Then, 
\[
\mathcal{W}_{\eta}^{2}(\mu,\nu)\leq\min_{\pi\in\Pi(\mu,\nu)}\int_{\mathbb{R}^{d}\times\mathbb{R}^{d}}\mathcal{W}_{\eta}^{2}(\delta_{x},\delta_{y})d\pi(x,y).
\]
\end{thm}

Morally speaking, this theorem is just an instance of Jensen's inequality,
since $\mathcal{W}^{2}$ is a convex l.s.c. function, and $\mathcal{W}_{\eta}^{2}\left(\mu,\nu\right)=\mathcal{W}_{\eta}^{2}\left(\int\delta_{x}d\pi(x,y),\int\delta_{y}d\pi(x,y)\right)$.
And indeed, in the discrete case, the proof of \cite[Proposition 2.14]{erbar2012ricci}
proceeds rather directly from Jensen's inequality, albeit applied
to the action $\mathcal{A}(\rho,\mathbf{j})$ rather than to $\mathcal{W}$.
However, standard proofs of Jensen's inequality (see for instance
\cite{perlman1974jensen}) require the underlying space (in this case
$\mathcal{P}(\mathbb{R}^{d})^{2}$) to carry the structure of a topological
vector space, which we do not have here. 
We are aware of one more abstract version of Jensen's inequality \cite{teran2014jensen} which
does not require a t.v.s. structure, but in our situation a direct
proof turns out to be readily available.
\begin{proof}
Let $\pi\in\Pi(\mu,\nu)$. Let $(X(\omega),Y(\omega))$ and $(X_{i}(\omega),Y_{i}(\omega))$,
$i=1,2,\ldots$ be i.i.d. random variables distributed according to
$\pi$. In particular, $(X,Y)_{\#}\mathbb{P}=\pi$. By the joint convexity
of $\mathcal{W}_{\eta}^{2}$, we have that 
\[
\mathcal{W}_{\eta}^{2}\left(\frac{1}{n}\sum_{i=1}^{n}\delta_{X_{i}(\omega)},\frac{1}{n}\sum_{i=1}^{n}\delta_{Y_{i}(\omega)}\right)\leq\frac{1}{n}\sum_{i=1}^{n}\mathcal{W}_{\eta}^{2}(\delta_{X_{i}(\omega)},\delta_{Y_{i}(\omega)}).
\]
At the same time, by change of variables, we observe that 
\[
\int\mathcal{W}_{\eta}^{2}(\delta_{X(\omega)},\delta_{Y(\omega)})d\mathbb{P}(\omega)=\int\mathcal{W}_{\eta}^{2}(\delta_{x},\delta_{y})d\pi(x,y).
\]
Now, suppose that $\mathcal{W}_{\eta}^{2}(\delta_{X(\omega)},\delta_{Y(\omega)})$
is an integrable random variable; since $\mathcal{W}_{\eta}^{2}$
is nonnegative, otherwise $\int\mathcal{W}_{\eta}^{2}(\delta_{x},\delta_{y})d\pi(x,y)=\infty$,
in which case the theorem holds trivially. Applying the strong law
of large numbers to the i.i.d. random variables $\mathcal{W}_{\eta}^{2}(\delta_{X_{i}(\omega)},\delta_{Y_{i}(\omega)})$,
we see that with probability 1, 
\[
\frac{1}{n}\sum_{i=1}^{n}\mathcal{W}_{\eta}^{2}(\delta_{X_{i}(\omega)},\delta_{Y_{i}(\omega)})\rightarrow\int\mathcal{W}_{\eta}^{2}(\delta_{X(\omega)},\delta_{Y(\omega)})d\mathbb{P}(\omega).
\]
Therefore, 
\[
\liminf_{n\rightarrow\infty}\mathcal{W}_{\eta}^{2}\left(\frac{1}{n}\sum_{i=1}^{n}\delta_{X_{i}(\omega)},\frac{1}{n}\sum_{i=1}^{n}\delta_{Y_{i}(\omega)}\right)\leq\int\mathcal{W}_{\eta}^{2}(\delta_{x},\delta_{y})d\pi(x,y).
\]
The Glivenko-Cantelli theorem tells us that with probability 1, $\frac{1}{n}\sum_{i=1}^{n}\delta_{X_{i}(\omega)}\rightharpoonup^{*}\mu$
and $\frac{1}{n}\sum_{i=1}^{n}\delta_{Y_{i}(\omega)}\rightharpoonup^{*}\nu$,
so since $\mathcal{W}_{\eta}^{2}$ is jointly l.s.c. with respect
to narrow convergence,
\[
\mathcal{W}_{\eta}^{2}(\mu,\nu)\leq\liminf_{n\rightarrow\infty}\mathcal{W}_{\eta}^{2}\left(\frac{1}{n}\sum_{i=1}^{n}\delta_{X_{i}(\omega)},\frac{1}{n}\sum_{i=1}^{n}\delta_{Y_{i}(\omega)}\right).
\]
This shows that $\mathcal{W}_{\eta}^{2}(\mu,\nu)\leq\int\mathcal{W}_{\eta}^{2}(\delta_{x},\delta_{y})d\pi(x,y)$.
But since $\pi\in\Pi(\mu,\nu)$ was arbitrary, we find that $\mathcal{W}_{\eta}^{2}(\nu_{0},\nu_{1})\leq\inf_{\pi\in\Pi(\nu_{0},\nu_{1})}\int\mathcal{W}_{\eta}^{2}(\delta_{x},\delta_{y})d\pi(x,y)$.
Finally, the fact that the infimum is actually attained follows from
the fact that $c(x,y):=\mathcal{W}_{\eta}^{2}(\delta_{x},\delta_{y})$
is a nonnegative l.s.c. cost function, so standard Monge-Kantorovich
theory applies, for instance \cite[Theorem 1.7]{santambrogio2015optimal}. 
\end{proof}

We now use the disintegration inequality and the estimates on the nonloal transport between delta masses to obtain the initial, crude, upper bound on nonlocal transport between general probability measures in terms of the Wasserstein distance.
\begin{lemma}
\label{lem:crude W2 upper bound} Let $\nu_{0},\nu_{1}$ be any measures
in $\mathcal{P}_{2}(\mathbb{R}^{d})$. Suppose that $\eta(x,y)\geq c_{s}|x-y|^{-d-s}$
when $|x-y|\leq\frac{1}{6}$, or $\theta(1,0)=\kappa_{\theta}>0$
(or both). Then, 
\[
\mathcal{W}_{\eta,\varepsilon}^{2}(\nu_{0},\nu_{1})\leq2\frac{C_{d,\theta}^{2}}{\eta\left(\frac{1}{2}\right)}\frac{1}{\varepsilon^{2}}W_{2}^{2}(\nu_{0},\nu_{1})+2C_{d,\theta,\eta}^{2}
\]
where $C_{d,\theta}=\frac{C_{\theta}}{\frac{2}{3}\sqrt{\alpha_{d}}\left(\frac{1}{6}\right)^{d/2}}$,
$C_{\theta}=\int_{0}^{1}\frac{1}{\sqrt{\theta(1-r,1+r)}}dr$, and
$C_{d,\theta,\eta}$ is the constant from Lemma \ref{lem:nonlocal dirac crude upper bound}.

Furthermore, in the case where $\eta(x,y)\geq c_{s}|x-y|^{-d-s}$
when $|x-y|\leq\frac{1}{6}$, one has the alternative upper bound
\[
\mathcal{W}_{\eta,\varepsilon}^{2}(\nu_{0},\nu_{1})\leq\min_{\pi\in\Pi(\nu_{0},\nu_{1})}\int_{\left(\mathbb{R}^{d}\right)^{2}}\Phi\left(\frac{|x-y|}{\varepsilon}\right)^{2}d\pi(x,y)
\]
where $\Phi$ is the function from the statement of Lemma \ref{lem:topological upper bound non-integrable}.
\end{lemma}

\begin{rem}
This lemma actually helps to address a question posed by Erbar.
In \cite{erbar2014gradient}, it is mentioned that it is unclear for
which probability measures $\nu_{0}$ and $\nu_{1}$ on $\mathbb{R}^{d}$
we have that $\mathcal{W}(\nu_{0},\nu_{1})<\infty$ (in the case of
the Wasserstein distances $W_{p}$, these are precisely the measures
which have finite $p$th moments). The proposition we are about to
prove gives a sufficient condition on $\eta$ and $\theta$ ensuring
that $\mathcal{W}(\nu_{0},\nu_{1})<\infty$ for all $\nu_{0},\nu_{1}\in\mathcal{P}_{2}(\mathbb{R}^{d})$.
\end{rem}

\begin{proof}
In Proposition \ref{thm:disintegration-inequality}, we proved the
disintegration inequality
\[
\mathcal{W}_{\eta,\varepsilon}^{2}(\nu_{0},\nu_{1})\leq\min_{\pi\in\Pi(\nu_{0},\nu_{1})}\int_{\left(\mathbb{R}^{d}\right)^{2}}\mathcal{W}_{\eta,\varepsilon}^{2}(\delta_{x},\delta_{y})d\pi(x,y).
\]
On the other hand, 
\[
W_{2}^{2}(\nu_{0},\nu_{1}):=\min_{\pi\in\Pi(\nu_{0},\nu_{1})}\int_{\left(\mathbb{R}^{d}\right)^{2}}|x-y|^{2}d\pi(x,y).
\]
Therefore, let $\overline{\pi}$ be $W_{2}$-optimal plan for $(\nu_{0},\nu_{1})$.
Using the assumption that $\eta(x,y)\geq c_{s}|x-y|^{-d-s}$ when
$|x-y|\leq\frac{1}{6}$, or $\theta(1,0)=\kappa_{\theta}>0$ (or both),
it follows that 
\begin{align*}
\mathcal{W}_{\eta,\varepsilon}^{2}(\nu_{0},\nu_{1}) & \leq\int_{\left(\mathbb{R}^{d}\right)^{2}}\mathcal{W}_{\eta,\varepsilon}^{2}(\delta_{x},\delta_{y})d\overline{\pi}(x,y)\\
(\text{Lemma \ref{lem:nonlocal dirac crude upper bound}}) & \leq\int_{\left(\mathbb{R}^{d}\right)^{2}}\left[\frac{C_{d,\theta}}{\sqrt{\eta\left(\frac{1}{2}\right)}}\frac{1}{\varepsilon}|x-y|+C_{d,\theta,\eta}\right]^{2}d\overline{\pi}(x,y)\\
 & \leq2\left(\frac{C_{d,\theta}}{\sqrt{\eta\left(\frac{1}{2}\right)}}\frac{1}{\varepsilon}\right)^{2}\int_{\left(\mathbb{R}^{d}\right)^{2}}|x-y|^{2}d\overline{\pi}(x,y)+2C_{d,\theta,\eta}^{2}\\
 & =2\left(\frac{C_{d,\theta}}{\sqrt{\eta\left(\frac{1}{2}\right)}}\frac{1}{\varepsilon}\right)^{2}W_{2}^{2}(\nu_{0},\nu_{1})+2C_{d,\theta,\eta}^{2}.
\end{align*}
Alternatively, in the case where $\eta(|x-y|)\geq c_{s}|x-y|^{-d-s}$
when $|x-y|\leq\frac{1}{6}$, we can apply Lemma \ref{lem:topological upper bound non-integrable},
and deduce that 
\begin{align*}
\mathcal{W}_{\eta,\varepsilon}^{2}(\nu_{0},\nu_{1}) & \leq\min_{\pi\in\Pi(\nu_{0},\nu_{1})}\int_{\left(\mathbb{R}^{d}\right)^{2}}\mathcal{W}_{\eta,\varepsilon}^{2}(\delta_{x},\delta_{y})d\pi(x,y)\leq\min_{\pi\in\Pi(\nu_{0},\nu_{1})}\int_{\left(\mathbb{R}^{d}\right)^{2}}\Phi\left(\frac{|x-y|}{\varepsilon}\right)^{2}d\pi(x,y).
\end{align*}
\end{proof}

Let us mention the following consequence of Lemma \ref{lem:crude W2 upper bound}.
Together with Proposition \ref{prop:TV lower bound}, this shows that
when $\eta$ is integrable, on a bounded domain the topology induced
by $\mathcal{W}_{\eta}$ is equivalent to the strong topology on probability
measures.
\begin{prop}[$TV$ upper bound on a bounded set]
\label{prop:W+TV upper bound} Suppose that $\eta(x,y)\geq c_{s}|x-y|^{-d-s}$
when $|x-y|\leq\frac{1}{6}$, or $\theta(1,0)=\kappa_{\theta}>0$
(or both), 
and that $\nu_{0},\nu_{1}\in\mathcal{P}(\mathbb{R}^{d})$ are both
supported inside some bounded set $K\subset\mathbb{R}^{d}$. Then,
there exists some constant $C$ (independent of $\nu_{0}$ and $\nu_{1}$,
but allowed to depend on $d,\theta,\eta$, and the diameter of $K$)
such that 
\[
\mathcal{\mathcal{W}}_{\eta}^{2}(\nu_{0},\nu_{1})\leq C\cdot TV(\nu_{0},\nu_{1}).
\]
\end{prop}

\begin{proof}
The idea is that we simply rerun the argument for Lemma \ref{lem:crude W2 upper bound},
but allow the mass in the ``overlap'' between $\nu_{0}$ and $\nu_{1}$
to stay put. 

To wit, define the measure $\Theta:=\min\{\nu_{0},\nu_{1}\}$. We
suppose that $\Theta$ is not identically zero, since otherwise the
desired inequality holds trivially. Observe that 
\[
\Vert\nu_{0}-\Theta\Vert_{TV}=\Vert\nu_{1}-\Theta\Vert_{TV}=2TV(\nu_{0},\nu_{1}),
\]
so in particular 
\[
\mathcal{W}_{\eta}\left(\frac{\nu_{0}-\Theta}{2TV(\nu_{0},\nu_{1})},\frac{\nu_{1}-\Theta}{2TV(\nu_{0},\nu_{1})}\right)
\]
is well-defined. Moreover, by Lemma \ref{lem:crude W2 upper bound}
(with $\varepsilon=1$) it holds that 
\begin{align*}
\mathcal{W}_{\eta}^{2}\left(\frac{\nu_{0}-\Theta}{2TV(\nu_{0},\nu_{1})},\frac{\nu_{1}-\Theta}{2TV(\nu_{0},\nu_{1})}\right) & \leq2\frac{C_{d,\theta}^{2}}{\eta\left(\frac{1}{2}\right)}W_{2}^{2}\left(\frac{\nu_{0}-\Theta}{2TV(\nu_{0},\nu_{1})},\frac{\nu_{1}-\Theta}{2TV(\nu_{0},\nu_{1})}\right)+2C_{d,\theta,\eta}^{2}.
\end{align*}
By the 1-homogeneity of the action $\mathcal{A}_{\eta,\theta}$, and
also the 1-homogeneity of $W_{2}^{2}$, this implies that 
\begin{align*}
\mathcal{W}_{\eta}^{2}\left(\nu_{0}-\Theta,\nu_{1}-\Theta\right) & \leq2\frac{C_{d,\theta}^{2}}{\eta\left(\frac{1}{2}\right)}W_{2}^{2}\left(\nu_{0}-\Theta,\nu_{1}-\Theta\right)\\
 & \qquad+4C_{d,\theta,\eta}^{2}TV(\nu_{0},\nu_{1}).
\end{align*}

At the same time, any solution to the nonlocal continuity equation
with endpoints $\nu_{0}-\Theta$ and $\nu_{1}-\Theta$ extends trivially
to a solution to the nonlocal continuity equation with endpoints $\nu_{0}$
and $\nu_{1}$: this is because the nonlocal continuity equation is
additive, and so we can just add on the constant solution $(\Theta,0)_{t\in[0,1]}$
to the NCE. By the convexity of the action, this implies that 
\[
\mathcal{W}_{\eta}^{2}(\nu_{0},\nu_{1})\leq2TV(\nu_{0},\nu_{1})\mathcal{W}_{\eta}^{2}\left(\frac{\nu_{0}-\Theta}{2TV(\nu_{0},\nu_{1})},\frac{\nu_{1}-\Theta}{2TV(\nu_{0},\nu_{1})}\right)=\mathcal{W}_{\eta}^{2}\left(\nu_{0}-\Theta,\nu_{1}-\Theta\right).
\]
Therefore, 
\begin{align*}
\mathcal{W}_{\eta}^{2}\left(\nu_{0},\nu_{1}\right) & \leq2\frac{C_{d,\theta}^{2}}{\eta\left(\frac{1}{2}\right)}W_{2}^{2}(\nu_{0},\nu_{1})+4C_{d,\theta,\eta}^{2}TV(\nu_{0},\nu_{1}).
\end{align*}
as desired. 

Lastly, by combining \cite[Equation (5.1)]{santambrogio2015optimal}
with \cite[Theorem 4]{gibbs2002choosing}, we see that 
\[
W_{2}^{2}(\nu_{0},\nu_{1})\leq\diam(K)\cdot W_{1}(\nu_{0},\nu_{1})\leq\diam(K)^{2}\cdot TV(\nu_{0},\nu_{1})
\]
which allows us to deduce that
\[
\mathcal{W}_{\eta}^{2}(\nu_{0},\nu_{1})\leq\left(2\frac{C_{d,\theta}^{2}\cdot\diam(K)^{2}}{\eta\left(\frac{1}{2}\right)}+4C_{d,\theta,\eta}^{2}\right)TV(\nu_{0},\nu_{1}).
\]
\end{proof}

\section{Exact Nonlocalization}
\label{exact nonlocalization}


\subsection{Exact solution to nonlocal continuity equation}
We start by introducing a way to use solutions of the continuity equation to create solutions of the nonlocal continuity equation with a kernel $\eta$, \eqref{eq:nce}. 
 Namely we discovered that given a solution of the continuity equation in the flux form one can convolve it by a specific, $\eta$-dependent kernel $\zeta$ so that the convolved flow is an exact solution of the nonlocal continuity equation. We first present the solution in a formal way and then justify it for weak solutions below.

Let $\eta (s)$ denote the radial profile of a kernel $\eta (x,y)$ satisfying Assumption \ref{assu:eta properties}. Define 
\[ \zeta(r) = \int_r^\infty  s \eta(s) ds. \]
One can check that under this assumption, $\zeta$ is integrable in $\R^d$ even when $\eta$ is not.

Consider a solution of the continuity equation
\[ \partial_t \rho + \divv(J) = 0. \]
Let $\rho_\zeta = \rho * \zeta$ and 
$J_\zeta = J * \zeta$.
Then 
\[ \partial_t \rho_\zeta + \divv(J_\zeta) = 0. \]
Let $j(x,y) = (y-x) \cdot (J(y) +  J(x) ) $. We claim that 
\[ \partial_t \rho_\zeta + \int j(x,y) \eta(x-y) dy = 0. \]
Namely note that 
\[ \nabla \zeta(|x-y|) = \eta(|x-y|) (x-y). \]
Thus, using symmetry of $\eta$, 
\begin{align*}
\int j(x,y) \eta(|x-y|) dy & = - \int  \eta(|x-y|) (x-y) \cdot (J(y) + J(x)) dy \\
& = \int \nabla \zeta(|x-y|)  \cdot  J(y) dy  - J(x) \int  \eta(|x-y|) (x-y)  dy \\ 
& = \divv \left( \int \zeta(|x-y|)  J(y) dy  \right)+ 0 = \divv (J_\zeta). 
\end{align*}

While the preceding argument is formal, and written for strong solutions, making the argument rigorous and extending to  weak solutions is straightforward, and is done in the lemma below.

In what follows, we use slightly more burdensome notation: given a
specific kernel $\eta$, we write $\zeta_{\eta}(r):=\int_{r}^{\infty}s\eta(s)ds$;
so in particular $\zeta_{(\eta_{\varepsilon})}(r):=\int_r^\infty s\eta_{\varepsilon}(s)ds$.
Note however that it need not hold that $\zeta_{\eta}(|x|)$ is a
\emph{convolution} kernel, i.e. in may not be normalized when integrating
on $\mathbb{R}^{d}$. We therefore introduce the (normalized) convolution
kernels (that these are the correct normalization constants is shown
in Lemma \ref{lem:zeta kernel mass}):
\[
\bar{\zeta}_{\eta}:=\frac{d}{M_{2}(\eta)}\zeta_{\eta};\qquad\bar{\zeta}_{(\eta_{\varepsilon})}:=\frac{d}{\varepsilon^{2}M_{2}(\eta)}\zeta_{(\eta_{\varepsilon})}.
\]
Note that $\bar{\zeta}_{\eta}$ is a convolution kernel supported
on the unit ball, while $\bar{\zeta}_{(\eta_{\varepsilon})}$ is a
convolution kernel supported on the ball of radius $\varepsilon$.

\begin{lemma}[exact nonlocalization]
\label{lem:exact-nonlocalization} Assume that $\eta$ 
satisfies Assumption \ref{assu:eta properties}. Suppose that for all test functions $\varphi_{t}\in C_{c}^{\infty}((0,1)\times\mathbb{R}^{d})$,
$(\rho_{t},\vec{\mathbf{j}}_{t})_{t\in[0,1]}\in[0,1]\rightarrow\mathcal{P}(\mathbb{R}^{d})\times\mathcal{M}_{loc}(\mathbb{R}^{d};\mathbb{R}^{d})$
satisfies
\[
\int_{0}^{1}\int_{\mathbb{R}^{d}}\partial_{t}\varphi_{t}(x)d\rho_{t}(x)dt+\int_{0}^{1}\int_{\mathbb{R}^{d}}\nabla\varphi_{t}(x)\cdot d\vec{\mathbf{j}}_{t}(x)dt=0.
\]
Then, it holds that $\left(\boldsymbol{\bar{\zeta}_{(\eta_{\varepsilon})}}*\rho_{t},\mathbf{j}_{t}\right)_{t\in[0,1]}$
solves the nonlocal continuity equation, where the measure $\mathbf{j}_{t}:[0,1]\rightarrow\mathcal{M}_{loc}(G)$
is defined in the following way: for all test functions $\Phi(x,y)\in C_{C}^{\infty}(G)$,
we define
\[
\iint_{G}\Phi(x,y)d\mathbf{j}_{t}(x,y):=\frac{d}{\varepsilon^{2}M_{2}(\eta)}\iint_{G}\Phi(x,y)(y-x)\cdot\left(d\vec{\mathbf{j}}_{t}(x)dy+d\vec{\mathbf{j}}_{t}(y)dx\right).
\]
\end{lemma}

\begin{proof}
Suppose that for all test functions $\varphi_{t}\in C_{c}((0,1)\times\mathbb{R}^{d})$,
$(\rho_{t},\vec{\mathbf{j}}_{t})_{t\in[0,1]}\in[0,1]\rightarrow\mathcal{P}(\mathbb{R}^{d})\times\mathcal{M}_{loc}(\mathbb{R}^{d};\mathbb{R}^{d})$
satisfies
\[
\int_{0}^{1}\int_{\mathbb{R}^{d}}\partial_{t}\varphi_{t}(x)d\rho_{t}(x)dt+\int_{0}^{1}\int_{\mathbb{R}^{d}}\nabla\varphi_{t}(x)\cdot d\vec{\mathbf{j}}_{t}(x)dt=0.
\]
Then, it also holds that $\left(\boldsymbol{\bar{\zeta}_{(\eta_{\varepsilon})}}*\rho_{t},\boldsymbol{\bar{\zeta}_{(\eta_{\varepsilon})}}*\vec{\mathbf{j}}_{t}\right)_{t\in[0,1]}$
is also a solution to this form of the continuity equation:
\[
\int_{0}^{1}\int_{\mathbb{R}^{d}}\partial_{t}\varphi_{t}(x)d\left(\boldsymbol{\bar{\zeta}_{(\eta_{\varepsilon})}}*\rho_{t}\right)(x)dt+\int_{0}^{1}\int_{\mathbb{R}^{d}}\nabla\varphi_{t}(x)\cdot d\left(\boldsymbol{\bar{\zeta}_{(\eta_{\varepsilon})}}*\vec{\mathbf{j}}_{t}\right)(x)dt=0.
\]
Our goal is to show that 
\[
\int_{0}^{1}\int_{\mathbb{R}^{d}}\partial_{t}\varphi_{t}(x)d\left(\boldsymbol{\bar{\zeta}_{(\eta_{\varepsilon})}}*\rho_{t}\right)(x)dt+\frac{1}{2}\int_{0}^{1}\iint_{G}\bar{\nabla}\varphi_{t}(x,y)\eta_{\varepsilon}(x,y)d\mathbf{j}_{t}(x,y)dt=0,
\]
so we claim that 
\[
\frac{1}{2}\int_{0}^{1}\iint_{G}\bar{\nabla}\varphi_{t}(x,y)\eta_{\varepsilon}(x,y)d\mathbf{j}_{t}(x,y)dt=\int_{0}^{1}\int_{\mathbb{R}^{d}}\nabla\varphi_{t}(x)\cdot d\left(\boldsymbol{\bar{\zeta}_{(\eta_{\varepsilon})}}*\vec{\mathbf{j}}_{t}\right)(x)dt,
\]
which establishes the theorem. Recalling that $\bar{\nabla}\varphi_{t}(x,y):=\varphi_{t}(y)-\varphi_{t}(x)$,
we see that for each $t\in[0,1]$, 
\begin{align*}
\iint_{G}\bar{\nabla}\varphi_{t}(x,y)\eta_{\varepsilon}(x,y)d\mathbf{j}_{t}(x,y) & =\frac{1}{\varepsilon^{2}\alpha_{d}\sigma_{\eta}}\iint_{G}\varphi_{t}(y)\eta_{\varepsilon}(x,y)(y-x)\cdot\left(d\vec{\mathbf{j}}_{t}(x)dy+d\vec{\mathbf{j}}_{t}(y)dx\right)\\
 & \quad-\frac{1}{\varepsilon^{2}\alpha_{d}\sigma_{\eta}}\iint_{G}\varphi_{t}(x)\eta_{\varepsilon}(x,y)(y-x)\cdot\left(d\vec{\mathbf{j}}_{t}(x)dy+d\vec{\mathbf{j}}_{t}(y)dx\right).
\end{align*}
(Note that the two integrals on the right hand side are well-defined,
since $\varphi_{t}(x)$ is smooth and compactly supported in $\mathbb{R}^{d}$,
$\eta_{\varepsilon}(x,y)(y-x)$ is integrable and compactly supported
in $x$ for each $y$, and $\vec{\mathbf{j}}_{t}\in\mathcal{M}_{loc}(\mathbb{R}^{d};\mathbb{R}^{d})$.)
So first, compute (using the fact that $\nabla_{y}\zeta_{(\eta_{\varepsilon})}(|x-y|)=(x-y)\eta_{\varepsilon}(|x-y|)$)
that 
\begin{align*}
\frac{d}{\varepsilon^{2}M_{2}(\eta)}\iint_{G}\varphi_{t}(y)\eta_{\varepsilon}(x,y)(y-x)\cdot d\vec{\mathbf{j}}_{t}(x)dy & =-\frac{d}{\varepsilon^{2}M_{2}(\eta)}\iint_{G}\varphi_{t}(y)\nabla_{y}\zeta_{(\eta_{\varepsilon})}(|x-y|)\cdot d\vec{\mathbf{j}}_{t}(x)dy\\
 & =\frac{d}{\varepsilon^{2}M_{2}(\eta)}\iint_{G}\zeta_{(\eta_{\varepsilon})}(|x-y|)\nabla\varphi_{t}(y)\cdot d\vec{\mathbf{j}}_{t}(x)dy\\
 & =\iint_{G}\bar{\zeta}_{(\eta_{\varepsilon})}(|x-y|)\nabla\varphi_{t}(y)\cdot d\vec{\mathbf{j}}_{t}(x)dy\\
 & =\int_{\mathbb{R}^{d}}\nabla\varphi_{t}(y)\cdot\left(\int_{\mathbb{R}^{d}}\bar{\zeta}_{(\eta_{\varepsilon})}(|x-y|)d\vec{\mathbf{j}}_{t}(x)\right)dy\\
 & =\int_{\mathbb{R}^{d}}\nabla\varphi_{t}(y)\cdot d\left(\boldsymbol{\bar{\zeta}_{(\eta_{\varepsilon})}}*\vec{\mathbf{j}}_{t}\right)(y).
\end{align*}
By identical reasoning, 
\[
\frac{d}{\varepsilon^{2}M_{2}(\eta)}\iint_{G}\varphi_{t}(x)\eta_{\varepsilon}(x,y)(y-x)\cdot d\vec{\mathbf{j}}_{t}(y)dx=-\int_{\mathbb{R}^{d}}\nabla\varphi_{t}(x)\cdot d\left(\boldsymbol{\bar{\zeta}_{(\eta_{\varepsilon})}}*\vec{\mathbf{j}}_{t}\right)(x).
\]
Next, compute that 
\begin{align*}
\iint_{G}\varphi_{t}(y)\eta_{\varepsilon}(x,y)(y-x)\cdot d\vec{\mathbf{j}}_{t}(y)dx & =\int_{\mathbb{R}^{d}}\varphi_{t}(y)\left(\int_{\mathbb{R}^{d}}\eta_{\varepsilon}(x,y)(y-x)dx\right)\cdot d\vec{\mathbf{j}}_{t}(y)\\
 & =0
\end{align*}
since the function $\eta_{\varepsilon}(x,y)(y-x)$ is radially anti-symmetric
around $y$, implying that $\int_{\mathbb{R}^{d}}\eta_{\varepsilon}(x,y)(y-x)dx=0$.
 By identical reasoning, it also holds that 
\[
\iint_{G}\varphi_{t}(x)\eta_{\varepsilon}(x,y)d\left(\vec{\mathbf{j}}_{t}\cdot(y-x)\right)(x)dy=0.
\]

Therefore, for all $t\in[0,1]$,
\[
\iint_{G}\bar{\nabla}\varphi_{t}(x,y)\eta_{\varepsilon}(x,y)d\mathbf{j}_{t}(x,y)=2\int_{\mathbb{R}^{d}}\nabla\varphi_{t}(x)\cdot\left(\boldsymbol{\bar{\zeta}_{(\eta_{\varepsilon})}}*\vec{\mathbf{j}}_{t}(y)\right)dx,
\]
which establishes the claim.
\end{proof}

\subsection{A quantitative upper bound on the nonlocal Wasserstein distance} 
\label{sec:UB-error}

In this section we establish a bound on the nonlocal Wasserstein distance of the form $\mathcal{W}_{\eta, \veps}  \leq  C W_{2} + O(\sqrt \veps)$, where $C$ is the exact proportionality constant that also appears in the matching lower bound on $\mathcal{W}_{\eta, \veps} $ presented in  Corollary \ref{lower-bound}.

\begin{prop}[bounding $\mathcal{W}$ by $W_{2}$]
 Assume that $\eta$ and $\theta$ satisfy Assumptions \ref{assu:eta properties}
and \ref{assu:theta properties} respectively. Let $K$ denote the
convolution kernel $K(x)=c_{K}e^{-|x|}$, where $c_{K}$ is a normalizing
constant, and let $\bar{\zeta}_{(\eta_{\varepsilon})}$ denote the
convolution kernel $\bar{\zeta}_{(\eta_{\varepsilon})}(x)=\frac{d}{\varepsilon^{2}M_{2}(\eta)}\int_{|x|}^{\infty}t\eta_{\varepsilon}(t)dt$.
Let $(\nu_{t},\vec{\mathbf{j}}_{t})_{t\in[0,1]}$ be a solution to
the (local) continuity equation in flux form. Furthermore, define
$\mathbf{j}_{t}:[0,1]\rightarrow\mathcal{M}_{loc}(G)$ as follows:
\[
d\mathbf{j}_{t}(x,y)=\frac{d}{\varepsilon^{2}M_{2}(\eta)}(y-x)\cdot\left(d(\boldsymbol{K_{s}}*\vec{\mathbf{j}}_{t})(x)dy+d(\boldsymbol{K_{s}}*\vec{\mathbf{j}}_{t}^{s})(y)dx\right).
\] 

Then, $(\boldsymbol{\bar{\zeta}_{(\eta_{\varepsilon})}}*\boldsymbol{K_{s}}*\rho_{t},\mathbf{j}_{t})_{t\in[0,1]}$
solves the nonlocal continuity equation, and for all $0<\varepsilon<s$,
and all $t\in[0,1]$, 
\[
\mathcal{A}_{\eta,\varepsilon}\left(\boldsymbol{\bar{\zeta}_{(\eta_{\varepsilon})}}*\boldsymbol{K_{s}}*\rho_{t},\mathbf{j}_{t}\right)\leq\frac{2d}{\varepsilon^{2}M_{2}(\eta)}\left(1+\frac{3}{s}\varepsilon\right)^{4}\mathcal{A}(\boldsymbol{K_{s}}*\rho_{t},\boldsymbol{K_{s}}*\vec{\mathbf{j}}_{t}).
\]
In particular, for all $\rho_{0},\rho_{1}\in\mathcal{P}_{2}(\mathbb{R}^{d})$,
\begin{align*}
\mathcal{W}_{\eta,\varepsilon}\left(\boldsymbol{\bar{\zeta}_{(\eta_{\varepsilon})}}*\boldsymbol{K_{s}}*\rho_{0},\boldsymbol{\bar{\zeta}_{(\eta_{\varepsilon})}}*\boldsymbol{K_{s}}*\rho_{1})\right) & \leq\frac{1}{\varepsilon}\left(\frac{2d}{M_{2}(\eta)}\right)^{1/2}\left(1+\frac{3}{s}\varepsilon\right)^{2}W_{2}\left(\rho_{0},\rho_{1}\right).
\end{align*}
\end{prop}
\begin{proof}
Let $(\rho_{t},\vec{\mathbf{j}}_{t})_{t\in[0,1]}$ be a solution to
the (local) continuity equation in flux form. Let $s>0$ be some fixed
quantity chosen later on; we define $\rho_{t}^{s}:=\boldsymbol{K_{s}}*\nu_{t}$
and $\vec{\mathbf{j}}_{t}^{s}:=\boldsymbol{K_{s}}*\vec{\mathbf{j}}_{t}$.
Note that these objects are measures; the corresponding Lebesgue densities
are $K_{s}*\rho_{t}$ and $\vec{j}_{t}^{s}:=K_{s}*\vec{\mathbf{j}}_{t}$
respectively. By \cite[Lemma 8.1.9]{ambrosio2008gradient}, $(\rho_{t}^{s},\vec{\mathbf{j}}_{t}^{s})$
also solves the continuity equation in flux form. We then smooth $\rho_{t}^{s}$
and $\vec{\mathbf{j}}_{t}^{s}$ \emph{again}, using the kernel $\bar{\zeta}_{(\eta_{\varepsilon})}$,
so that $\left(\boldsymbol{\bar{\zeta}_{(\eta_{\varepsilon})}}*\rho_{t}^{s},\boldsymbol{\bar{\zeta}_{(\eta_{\varepsilon})}}*\vec{\mathbf{j}}_{t}^{s}\right)$
again solves the continuity equation in flux form. By Lemma \ref{lem:zeta-rel-lipschitz-regularity},
we know (since $\eta_{\varepsilon}(|x-y|)$ is supported on $B(0,\varepsilon)$)
that if $\varepsilon<s$, then the corresponding Lebesgue density
$\bar{\zeta}_{(\eta_{\varepsilon})}*\rho_{t}^{s}$ has local relative
Lipschitz regularity of the form 
\[
\frac{\bar{\zeta}_{(\eta_{\varepsilon})}*\rho_{t}^{s}(y)}{\bar{\zeta}_{(\eta_{\varepsilon})}*\rho_{t}^{s}(x)}\leq\left(1+\frac{3}{s}\varepsilon\right)^{2}\left(1+\frac{3}{s}|x-y|\right)\text{ when }|x-y|<s.
\]

Define $\mathbf{j}_{t}:[0,1]\rightarrow\mathcal{M}_{loc}(G)$ as in
Lemma \ref{lem:exact-nonlocalization} with respect to $\vec{\mathbf{j}}_{t}^{s}$,
namely 
\[
d\mathbf{j}_{t}(x,y)=\frac{d}{\varepsilon^{2}M_{2}(\eta)}(y-x)\cdot\left(d\vec{\mathbf{j}}_{t}^{s}(x)dy+d\vec{\mathbf{j}}_{t}^{s}(y)dx\right).
\]
In this case, since $\vec{\mathbf{j}}_{t}^{s}$ has a density with
respect to the Lebesgue measure given by $\vec{j}_{t}^{s}$, it follows
that $\mathbf{j}_{t}$ has density with respect to the product Lebesgue
measure restricted to $G$, given by
\[
j_{t}(x,y):=\frac{d\mathbf{j}_{t}}{dxdy}=\frac{d}{\varepsilon^{2}M_{2}(\eta)}(y-x)\cdot\left[\vec{j}_{t}^{s}(x)+\vec{j}_{t}^{s}(y)\right].
\]
Furthermore, by Lemma \ref{lem:exact-nonlocalization}, we know that
$\left(\boldsymbol{\bar{\zeta}_{(\eta_{\varepsilon})}}*\rho_{t}^{s},\mathbf{j}_{t}\right)$
solves the nonlocal continuity equation. 

Now, let us compare the nonlocal action $\mathcal{A}_{\eta,\varepsilon}\left(\boldsymbol{\bar{\zeta}_{(\eta_{\varepsilon})}}*\rho_{t}^{s},\mathbf{j}_{t}\right)$
with the local action $\mathcal{A}(\rho_{t}^{s},\vec{\mathbf{j}}_{t}^{s})$. 
Relying on the homogeneity of the interpolation $\theta$, we observe (using
Lemma \ref{lem:zeta-rel-lipschitz-regularity}) that 
\begin{align*}
\theta(\bar{\zeta}_{(\eta_{\varepsilon})}*\rho_{t}^{s}(x),\bar{\zeta}_{(\eta_{\varepsilon})}*\rho_{t}^{s}(y)) & =\bar{\zeta}_{(\eta_{\varepsilon})}*\rho_{t}^{s}(y)\theta\left(\frac{\bar{\zeta}_{(\eta_{\varepsilon})}*\rho_{t}^{s}(x)}{\bar{\zeta}_{(\eta_{\varepsilon})}*\rho_{t}^{s}(y)},1\right)\\
 & \geq\bar{\zeta}_{(\eta_{\varepsilon})}*\rho_{t}^{s}(y)\theta\left(\frac{1}{\left(1+\frac{3}{s}\varepsilon\right)^{2}}\frac{1}{1+\frac{3}{s}|x-y|},1\right)\\
 & \geq\bar{\zeta}_{(\eta_{\varepsilon})}*\rho_{t}^{s}(y)\cdot\frac{1}{\left(1+\frac{3}{s}\varepsilon\right)^{2}}\frac{1}{1+\frac{3}{s}|x-y|}.
\end{align*}
Therefore, 
\begin{align*}
\mathcal{A}_{\eta,\varepsilon} & \left(\boldsymbol{\bar{\zeta}_{(\eta_{\varepsilon})}}*\rho_{t}^{s},\mathbf{j}_{t}\right)  =\int_{\mathbb{R}^{d}}\int_{\mathbb{R}^{d}}\frac{j_{t}(x,y)^{2}}{2\theta(\bar{\zeta}_{(\eta_{\varepsilon})}*\rho_{t}^{s}(x),\bar{\zeta}_{(\eta_{\varepsilon})}*\rho_{t}^{s}(y))}\eta_{\varepsilon}(x,y)dxdy\\
 & =\int_{\mathbb{R}^{d}}\int_{\mathbb{R}^{d}}\frac{\left(\frac{d}{\varepsilon^{2}M_{2}(\eta)}\left[\vec{j}_{t}^{s}(x)+\vec{j}_{t}^{s}(y)\right]\cdot(y-x)\right)^{2}}{2\theta(\bar{\zeta}_{(\eta_{\varepsilon})}*\rho_{t}^{s}(x),\bar{\zeta}_{(\eta_{\varepsilon})}*\rho_{t}^{s}(y))}\eta_{\varepsilon}(x,y)dxdy\\
 & \leq\left(1+\frac{3}{s}\varepsilon\right)^{2}\int_{\mathbb{R}^{d}}\int_{\mathbb{R}^{d}}\left(1+\frac{3}{s}|x-y|\right)\frac{\left(\frac{d}{\varepsilon^{2}M_{2}(\eta)}\left[\vec{j}_{t}^{s}(x)+\vec{j}_{t}^{s}(y)\right]\cdot(y-x)\right)^{2}}{2\bar{\zeta}_{(\eta_{\varepsilon})}*\rho_{t}^{s}(y)}\eta_{\varepsilon}(x,y)dxdy.
\end{align*}
 By \cite[Corollary 2.16]{garciatrillos2020gromov},
together with Lemma \ref{lem:moment versus profile}, we have that
\[
\int_{\mathbb{R}^{d}}\left(\vec{j}_{t}^{s}(y)\cdot(y-x)\right)^{2}\eta_{\varepsilon}(x,y)dx=\varepsilon^{2}\frac{M_{2}(\eta)}{d}|\vec{j}_{t}^{s}(y)|^{2}.
\]
Consequently, also using the fact that the support of $\eta_{\varepsilon}(x,y)$ has diameter $\varepsilon$,
\begin{multline*}
\left(1+\frac{3}{s}\varepsilon\right)^{2}\int_{\mathbb{R}^{d}}\int_{\mathbb{R}^{d}}\left(1+\frac{3}{s}|x-y|\right)\frac{\left(\frac{d}{\varepsilon^{2}M_{2}(\eta)}\vec{j}_{t}^{s}(y)\cdot(y-x)\right)^{2}}{\bar{\zeta}_{(\eta_{\varepsilon})}*\rho_{t}^{s}(y)}\eta_{\varepsilon}(x,y)dxdy\\
\begin{aligned} & \leq\left(\frac{d}{\varepsilon^{2}M_{2}(\eta)}\right)^{2}\left(1+\frac{3}{s}\varepsilon\right)^{3}\int_{\mathbb{R}^{d}}\int_{\mathbb{R}^{d}}\frac{\left(\vec{j}_{t}^{s}(y)\cdot(y-x)\right)^{2}}{\bar{\zeta}_{(\eta_{\varepsilon})}*\rho_{t}^{s}(y)}\eta_{\varepsilon}(x,y)dxdy\\
 & \leq\frac{d}{\varepsilon^{2}M_{2}(\eta)}\left(1+\frac{3}{s}\varepsilon\right)^{3}\int_{\mathbb{R}^{d}}\frac{(\vec{j}_{t}^{s}(y))^{2}}{\bar{\zeta}_{(\eta_{\varepsilon})}*\rho_{t}^{s}(y)}dy.
\end{aligned}
\end{multline*}
Likewise, 
\begin{multline*}
\left(1+\frac{3}{s}\varepsilon\right)^{2}\int_{\mathbb{R}^{d}}\int_{\mathbb{R}^{d}}\left(1+\frac{3}{s}|x-y|\right)\frac{\left(\frac{d}{\varepsilon^{2}M_{2}(\eta)}\vec{j}_{t}^{s}(x)\cdot(y-x)\right)^{2}}{\bar{\zeta}_{(\eta_{\varepsilon})}*\rho_{t}^{s}(y)}\eta_{\varepsilon}(x,y)dxdy\\
\begin{aligned} & \leq\frac{d}{\varepsilon^{2}M_{2}(\eta)}\left(1+\frac{3}{s}\varepsilon\right)^{3}\int_{\mathbb{R}^{d}}\frac{(\vec{j}_{t}^{s}(x))^{2}}{\bar{\zeta}_{(\eta_{\varepsilon})}*\rho_{t}^{s}(x)}dx\end{aligned}
\end{multline*}
and so we deduce that 
\[
\mathcal{A}_{\eta,\varepsilon}\left(\boldsymbol{\bar{\zeta}_{(\eta_{\varepsilon})}}*\rho_{t}^{s},\mathbf{j}_{t}\right)\leq\frac{2d}{\varepsilon^{2}M_{2}(\eta)}\left(1+\frac{3}{s}\varepsilon\right)^{3}\int_{\mathbb{R}^{d}}\frac{(\vec{j}_{t}^{s}(y))^{2}}{\bar{\zeta}_{(\eta_{\varepsilon})}*\rho_{t}^{s}(y)}dy.
\]
From Lemma \ref{lem:zeta-rel-lipschitz-regularity}, we know that
$\rho_{t}^{s}(x)\left(\frac{1}{1+\frac{3}{s}\varepsilon}\right)\leq(\bar{\zeta}_{(\eta_{\varepsilon})}*\rho_{t}^{s})(x)$.
Therefore,
\begin{align*}
\int_{\mathbb{R}^{d}}\frac{(\vec{j}_{t}^{s}(y))^{2}}{\bar{\zeta}_{(\eta_{\varepsilon})}*\rho_{t}^{s}(y)}dy & \leq\left(1+\frac{3}{s}\varepsilon\right)\int_{\mathbb{R}^{d}}\frac{(\vec{j}_{t}^{s}(y))^{2}}{\rho_{t}^{s}(y)}dy
\end{align*}
and hence 
\[
\mathcal{A}_{\eta,\varepsilon}\left(\boldsymbol{\bar{\zeta}_{(\eta_{\varepsilon})}}*\rho_{t}^{s},\mathbf{j}_{t}\right)\leq\frac{2d}{\varepsilon^{2}M_{2}(\eta)}\left(1+\frac{3}{s}\varepsilon\right)^{4}\mathcal{A}(\rho_{t}^{s},\vec{\mathbf{j}}_{t}^{s}).
\]

We now take $(\rho_{t},\vec{\mathbf{j}}_{t})_{t\in[0,1]}$
to be a curve of least action for $W_{2}$, connecting $\rho_{0}$
and $\rho_{1}$. Then, by using the fact that $\mathcal{A}(\rho_{t}^{s},\vec{\mathbf{j}}_{t}^{s})\leq\mathcal{A}(\rho_{t},\mathbf{j}_{t})$,
and integrating in $t$, we have that 
\[
\mathcal{W}\left(\boldsymbol{\bar{\zeta}_{(\eta_{\varepsilon})}}*\rho_{0}^{s},\boldsymbol{\bar{\zeta}_{(\eta_{\varepsilon})}}*\rho_{1}^{s}\right)\leq\frac{1}{\varepsilon}\left(\frac{2d}{M_{2}(\eta)}\right)^{1/2}\left(1+\frac{3}{s}\varepsilon\right)^{2}W_{2}(\rho_{0},\rho_{1}),
\]
as desired.
\end{proof}

\begin{cor}
\label{upper-bound}Let $\mu_{0},\mu_{1}\in\mathcal{P}_{2}(\mathbb{R}^{d})$.
Assume that $\eta$ and $\theta$ satisfy Assumptions \ref{assu:eta properties}
and \ref{assu:theta properties} respectively. Suppose that
$\eta(|x-y|)\geq c_{s}|x-y|^{-d-s}$ whenever $|x-y|\leq\frac{1}{6}$,
or $\theta(1,0)=\kappa_{\theta}>0$. Then, 
\[
\varepsilon\mathcal{W}_{\eta,\varepsilon}(\mu_{0},\mu_{1})\leq\left(\frac{2d}{M_{2}(\eta)}\right)^{1/2}W_{2}(\mu_{0},\mu_{1})+O(\sqrt{\varepsilon}).
\]
Explicitly, 
\begin{align*}
\mathcal{W}_{\varepsilon,\eta}(\mu_{0},\mu_{1})\leq & \frac{1}{\varepsilon}\left(\frac{2d}{M_{2}(\eta)}\right)^{1/2}\left(1+\sqrt{\varepsilon}\right)^{2}W_{2}(\mu_{0},\mu_{1})\\
 & +2\sqrt{2}\left(\frac{C_{d,\theta}}{\sqrt{\eta\left(\frac{1}{2}\right)}}\frac{1}{\varepsilon}\right)\left(\left(d^{2}+d\right)^{1/2}\sqrt{\varepsilon}+\left(\frac{d}{d+2}\frac{M_{4}(\eta)}{M_{2}(\eta)}\right)^{1/2}\varepsilon\right)+2\sqrt{2}C_{d,\theta,\eta}
\end{align*}
where $C_{d,\theta}=\frac{C_{\theta}}{\frac{2}{3}\sqrt{\alpha_{d}}\left(\frac{1}{6}\right)^{d/2}}$,
$C_{\theta}=\int_{0}^{1}\frac{1}{\sqrt{\theta(1-r,1+r)}}dr$, and
$C_{d,\theta,\eta}$ is the constant from Lemma \ref{lem:nonlocal dirac crude upper bound}.
\end{cor}

\begin{proof}
We saw in Lemma \ref{lem:crude W2 upper bound} that for arbitrary
$\nu_{0},\nu_{1}\in\mathcal{P}_{2}(\mathbb{R}^{d})$, 
\[
\mathcal{W}_{\eta,\varepsilon}^{2}(\nu_{0},\nu_{1})\leq2\frac{C_{d,\theta}^{2}}{\eta\left(\frac{1}{2}\right)}\frac{1}{\varepsilon^{2}}W_{2}^{2}(\nu_{0},\nu_{1})+2C_{d,\theta,\eta}^{2}
\]
where $C_{d,\theta}=\frac{C_{\theta}}{\frac{2}{3}\sqrt{\alpha_{d}}\left(\frac{1}{6}\right)^{d/2}}$,
$C_{\theta}=\int_{0}^{1}\frac{1}{\sqrt{\theta(1-r,1+r)}}dr$, and
$C_{d,\theta,\eta}$ is the constant from Lemma \ref{lem:nonlocal dirac crude upper bound}.
In particular, take $\nu_{0}=\mu_{0}$ and $\nu_{1}=\boldsymbol{\bar{\zeta}_{(\eta_{\varepsilon})}}*\boldsymbol{K_{s}}*\mu_{0}$.
It follows that 
\[
\mathcal{W}_{\eta,\varepsilon}^{2}(\mu_{0},\boldsymbol{\bar{\zeta}_{(\eta_{\varepsilon})}}*\boldsymbol{K_{s}}*\mu_{0})\leq2\frac{C_{d,\theta}^{2}}{\eta\left(\frac{1}{2}\right)}\frac{1}{\varepsilon^{2}}W_{2}^{2}(\mu_{0},\boldsymbol{\bar{\zeta}_{(\eta_{\varepsilon})}}*\boldsymbol{K_{s}}*\mu_{0})+2C_{d,\theta,\eta}^{2}.
\]
By identical reasoning, 
\[
\mathcal{W}_{\eta,\varepsilon}^{2}(\mu_{1},\boldsymbol{\bar{\zeta}_{(\eta_{\varepsilon})}}*\boldsymbol{K_{s}}*\mu_{1})\leq2\frac{C_{d,\theta}^{2}}{\eta\left(\frac{1}{2}\right)}\frac{1}{\varepsilon^{2}}W_{2}^{2}(\mu_{1},\boldsymbol{\bar{\zeta}_{(\eta_{\varepsilon})}}*\boldsymbol{K_{s}}*\mu_{1})+2C_{d,\theta,\eta}^{2}.
\]
Now since, by the triangle inequality, 
\[
\mathcal{W}_{\eta,\varepsilon}(\mu_{0},\mu_{1})\leq\mathcal{W}_{\eta,\varepsilon}(\mu_{0},\boldsymbol{\bar{\zeta}_{(\eta_{\varepsilon})}}*\boldsymbol{K_{s}}*\mu_{0})+\mathcal{W}_{\eta,\varepsilon}(\boldsymbol{\bar{\zeta}_{(\eta_{\varepsilon})}}*\boldsymbol{K_{s}}*\mu_{0},\boldsymbol{\bar{\zeta}_{(\eta_{\varepsilon})}}*\boldsymbol{K_{s}}*\mu_{1})+\mathcal{W}_{\eta,\varepsilon}(\mu_{1},\boldsymbol{\bar{\zeta}_{(\eta_{\varepsilon})}}*\boldsymbol{K_{s}}*\mu_{1})
\]
we can use the previous proposition to see that 
\begin{align*}
\mathcal{W}_{\eta,\varepsilon}(\mu_{0},\mu_{1})\leq & +\frac{1}{\varepsilon}\left(\frac{2d}{M_{2}(\eta)}\right)^{1/2}\left(1+\frac{3}{s}\varepsilon\right)^{2}W_{2}(\mu_{0},\mu_{1})\\
 & +\sqrt{2\frac{C_{d,\theta}^{2}}{\eta\left(\frac{1}{2}\right)}\frac{1}{\varepsilon^{2}}W_{2}^{2}(\mu_{0},\boldsymbol{\bar{\zeta}_{(\eta_{\varepsilon})}}*\boldsymbol{K_{s}}*\mu_{0})+2C_{d,\theta,\eta}^{2}}\\
 & +\sqrt{2\frac{C_{d,\theta}^{2}}{\eta\left(\frac{1}{2}\right)}\frac{1}{\varepsilon^{2}}W_{2}^{2}(\mu_{1},\boldsymbol{\bar{\zeta}_{(\eta_{\varepsilon})}}*\boldsymbol{K_{s}}*\mu_{1})+2C_{d,\theta,\eta}^{2}}.
\end{align*}
 It remains to estimate $W_{2}(\mu_{0},\boldsymbol{\bar{\zeta}_{(\eta_{\varepsilon})}}*\boldsymbol{K_{s}}*\mu_{0})$
and $W_{2}(\mu_{1},\boldsymbol{\bar{\zeta}_{(\eta_{\varepsilon})}}*\boldsymbol{K_{s}}*\mu_{1})$.
To do so, we can use two successive convolution estimates, since 
\[
W_{2}(\mu_{0},\boldsymbol{\bar{\zeta}_{(\eta_{\varepsilon})}}*\boldsymbol{K_{s}}*\mu_{0})\leq W_{2}(\mu_{0},\boldsymbol{K_{s}}*\mu_{0})+W_{2}(\boldsymbol{K_{s}}*\mu_{0},\boldsymbol{\bar{\zeta}_{(\eta_{\varepsilon})}}*\boldsymbol{K_{s}}*\mu_{0})
\]
and similarly for $\mu_{1}$. Thanks to the estimates from Lemmas
\ref{lem:more convolution estimates} and \ref{lem:Laplace kernel moments},
we know that 
\[
W_{2}(\mu_{0},\boldsymbol{K_{s}}*\mu_{0})\leq\left(\int|y|^{2}c_{K}e^{-|y|}dy\right)^{1/2}s=(d^{2}+d)^{1/2}s
\]
and 
\[
W_{2}(\boldsymbol{K_{s}}*\mu_{0},\boldsymbol{\bar{\zeta}_{(\eta_{\varepsilon})}}*\boldsymbol{K_{s}}*\mu_{0})\leq\left(\frac{d}{d+2}\frac{M_{4}(\eta)}{M_{2}(\eta)}\right)^{1/2}\varepsilon
\]
and of course the same holds for $\mu_{1}$. Therefore,  putting
$s=\sqrt{\varepsilon}$, we deduce the estimate 
\begin{align*}
\mathcal{W}_{\varepsilon,\eta}(\mu_{0},\mu_{1})\leq & \frac{1}{\varepsilon}\left(\frac{2d}{M_{2}(\eta)}\right)^{1/2}\left(1+\sqrt{\varepsilon}\right)^{2}W_{2}(\mu_{0},\mu_{1})\\
 & +2\sqrt{2}\left(\frac{C_{d,\theta}}{\sqrt{\eta\left(\frac{1}{2}\right)}}\frac{1}{\varepsilon}\right)\left(\left(d^{2}+d\right)^{1/2}\sqrt{\varepsilon}+\left(\frac{d}{d+2}\frac{M_{4}(\eta)}{M_{2}(\eta)}\right)^{1/2}\varepsilon\right)+2\sqrt{2}C_{d,\theta,\eta}
\end{align*}
as desired.
\end{proof}

\section{Nonlocal Hamilton-Jacobi Subsolution}
\label{sec:nonlocal hj}

We first first sketch the argument of this section. Our aim is to prove
a bound of the following form: 
\[
W_{2}(\mu_{0},\mu_{1})\leq\varepsilon\sqrt{\frac{M_{2}(\eta)}{2d}}\mathcal{W}_{\eta,\varepsilon}(\mu_{0},\mu_{1})+\text{error terms}.
\]
Ultimately, we will show that the error terms are of order $\sqrt{\varepsilon}$.
Our strategy is to use  \emph{Hamilton-Jacobi duality
}for $W_{2}$ (and also $\mathcal{W}_{\eta,\varepsilon}$). Simplifying
somewhat:
\begin{listi}
\item Let $(\phi_{t})_{t\in[0,1]}$ be a (viscosity) solution to the Hamilton-Jacobi
equation $\partial_{t}\phi_{t}+|\nabla\phi_{t}|^{2}=0$. Duality theory
for $W_{2}$ (specifically \cite[Proposition 5.48]{villani2003topics})
tells us that 
\[
\frac{1}{2}W_{2}^{2}(\mu_{0},\mu_{1})=\max_{\phi_{0}\in C_{b}}\left\{ \int\phi_{1}d\mu_{1}-\int\phi_{0}d\mu_{0}\::\:\partial_{t}\phi_{t}+\frac{1}{2}|\nabla\phi_{t}|^{2}=0\right\} .
\]
\item We expect that a similar duality theorem holds for $\mathcal{W}_{\eta,\varepsilon}$
(but there is a different notion of ``nonlocal Hamilton-Jacobi equation''):
\[
\frac{1}{2}\mathcal{W}_{\eta,\varepsilon}^{2}(\mu_{0},\mu_{1})=\max \left\{ \int\phi_{1}^\varepsilon d\mu_{1}-\int\phi_{0}^\varepsilon d\mu_{0}\::\:(\phi_{t}^\varepsilon)_{t\in[0,1]} \text{ is a n.l. HJ subsolution}\right\} .
\]
\item We will use solutions of the Hamilton-Jacobi equation to construct subsolutions to the nonlocal Hamilton-Jacobi equations thus obtaining a lower bound on $\mathcal{W}_{\eta,\varepsilon}^{2}$.
\end{listi}
Unfortunately, things are not so simple an we will need to introduce a layer of  approximations.
Because of the (conjectured, at this point, based
on the outcome of Section 4) asymptotic proportionality constant between
$W_{2}$ and $\mathcal{W}_{\eta,\varepsilon}$, the constant prefactor
in any mapping which takes a (local) HJ solution and gives us a
nonlocal HJ subsolution \emph{must} be $\frac{2d}{\varepsilon^{2}M_{2}(\eta)}$. 

Before proceeding, we present several preparatory lemmas.
\begin{defn}
(space of space-time bounded Lipschitz functions) We define 
\[
BL([0,1]\times\mathbb{R}^{d}):=\{\phi_{t}(x):[0,1]\times\mathbb{R}^{d}\rightarrow\mathbb{R} \: \mid \: \text{Lip}_{[0,1]\times\mathbb{R}^{d}}(\phi)+\Vert\phi\Vert_{L^{\infty}([0,1]\times\mathbb{R}^{d})}<\infty\}.
\]
\end{defn}

\begin{lemma}
\label{lem:Kantorovich potential Lipschitz bound}Suppose that $\mu_{0}$
and $\mu_{1}$ are probability measures which are both supported within
$B(0,R)$ inside $\mathbb{R}^{d}$. Then, in the Kantorovich duality
\[
\frac{1}{2}W_{2}^{2}(\mu_{0},\mu_{1})=\sup_{\phi\in BL(\mathbb{R}^{d})}\left\{ \int\phi^{c}(y)d\mu_{1}(y)-\int\phi(x)d\mu_{0}(x)\right\} 
\]
(where $\phi^{c}(y):=\inf_{x}\{\phi(x)+\frac{1}{2}|x-y|^{2}\}$),
the optimal potential, that is, 
\[
\underset{\phi\in BL(\mathbb{R}^{d})}{\text{argmax}}\left\{ \int\phi^{c}(y)d\mu_{1}(y)-\int\phi(x)d\mu_{0}(x)\right\} 
\]
has $\text{Lip}(\phi)\leq R$. 
\end{lemma}

Note that by Rademacher's theorem, if $\phi$ is Lipschitz then $\nabla\phi$
exists Lebesgue-almost everywhere; the lemma therefore also shows
that for the optimal Kantorovich potential, $|\nabla\phi|\leq R$.
\begin{proof}
This is an easy refinement of standard results concerning Kantorovich
duality, such as \cite[Theorem 5.10]{villani2008optimal}. 

Indeed, as discussed on \cite[p. 11]{santambrogio2015optimal}, if
$c(x,y)$ is any continuous cost function and $\psi^{c}$ is any $c$-convex
(resp. $c$-concave) function, it holds automatically that any modulus
of continuity for $c(x,y)$ is also a modulus of continuity for $\psi^{c}$.
In the case of $c(x,y)=\frac{1}{2}|x-y|^{2}$ on a domain of diameter
$R$, we can take $|c(x,y)-c(x^{\prime},y)|\leq R|x-x^{\prime}|$
as a crude global modulus of continuity in the $x$ variable (and
of course the same reasoning applies to the $y$ variable). Since
the optimal Kantorovich potential $\phi$ can always be taken to be
$c$-convex, the claim follows.
\end{proof}

The Kantorovich duality formula for $W_{2}$ also has a ``dynamic''
counterpart in terms of solutions to a Hamilton-Jacobi equation. This
fact was initially observed in \cite{bobkov2001hypercontractivity,otto2000generalization};
here we just give a  proof for convenience.


\begin{cor}
\label{cor:HJ Lipschitz diameter bound}Suppose that $\mu_{0}$ and
$\mu_{1}$ are probability measures which are both supported within
some domain of radius $R$ inside $\mathbb{R}^{d}$. Then, 
\[
\frac{1}{2}W_{2}^{2}(\mu_{0},\mu_{1})=\sup_{\phi_{t}\in BL([0,1]\times\mathbb{R}^{d})}\left\{ \int\phi_{1}d\mu_{1}-\int\phi_{0}d\mu_{0}:\partial_{t}\phi_{t}+\frac{1}{2}|\nabla\phi_{t}|^{2}=0\text{ in viscosity sense}\right\} 
\]
and it holds that the optimal Hamilton-Jacobi subsolution, that is,
\[
\underset{\phi\in BL([0,1]\times\mathbb{R}^{d})}{\mathrm{argmax}}\left\{ \int\phi_{1}d\mu_{1}-\int\phi_{0}d\mu_{0}:\partial_{t}\phi_{t}+\frac{1}{2}|\nabla\phi_{t}|^{2}=0\text{ in viscosity sense}\right\} 
\]
has the property that $\text{Lip}(\phi_{t})\leq R$, for all $t\in[0,1]$. 
\end{cor}

\begin{proof}
Let $\phi_{0}\in BL(\mathbb{R}^{d})$. By \cite[Theorem 10.3.3]{evans2022partial},
the unique viscosity solution of $\partial_{t}\phi_{t}+\frac{1}{2}|\nabla\phi_{t}|^{2}=0$
with initial condition $\phi_{0}$ is given by the Hopf-Lax formula
$\phi_{t}(x)=\inf_{y}\{\phi_{0}(y)+\frac{t}{2}|x-y|^{2}\}$ (which
at time 1 just returns $\phi_{0}^{c}$ for $c=\frac{1}{2}|x-y|^{2}$,
and so by comparing with \cite[Theorem 5.10]{villani2008optimal},
we see that 
\[
\sup_{\phi\in BL(\mathbb{R}^{d})}  \int\phi^{c}(y)d\mu_{1}(y)-\int\phi(x)d\mu_{0}(x) =\sup_{\phi_{0}\in BL(\mathbb{R}^{d})}\left\{ \int\phi_{1}d\mu_{1}- \int\phi_{0}d\mu_{0}:\partial_{t}\phi_{t}+\frac{1}{2}|\nabla\phi_{t}|^{2}=0\right\} 
\]
where on the right hand side $\phi_{t}$ solves $\partial_{t}\phi_{t}+\frac{1}{2}|\nabla\phi_{t}|^{2}=0$
in the viscosity sense. By \cite[Lemma 3.3.2]{evans2022partial},
we know that $\phi_{t}\in BL([0,1]\times\mathbb{R}^{d})$. 

Furthermore, the proof of \cite[Lemma 3.3.2]{evans2022partial} indicates
that $\text{Lip}(\phi_{t})\leq\text{Lip}(\phi_{0})$ for all $t\in[0,1].$
Therefore, by Lemma \ref{lem:Kantorovich potential Lipschitz bound}
it holds that $\text{Lip}(\phi_{t})\leq R$, for all $t\in[0,1]$.
\end{proof}
\begin{defn}
Following \cite[Definition 3.1]{erbar2019geometry}, we say that $\phi_{t}(x):[0,1]\times\mathbb{R}^{d}\rightarrow\mathbb{R}$
is a \emph{nonlocal Hamilton-Jacobi (HJ) subsolution}, and write $\phi_{t}(x)\in\text{HJ}_{\text{NL}}^{1}$
if, for a.e. $t\in(0,1)$,
if, for a.e. $t\in(0,1)$, the partial derivative $\partial_{t}\phi_{t}(x)$
exists for every $x\in\mathbb{R}^{d}$, and $\sup_{x\in\mathbb{R}^{d}}|\partial_{t}\phi_{t}(x)|<\infty$;
and we have, for all probability measures $\mu\in\mathcal{P}(\mathbb{R}^{d})$,
and for any (hence all) $\lambda$ such that $\mu\ll\lambda$, 
\[
\int\partial_{t}\phi_{t}(x)d\mu(x)+\frac{1}{4}\int(\phi_{t}(y)-\phi_{t}(x))^{2}\theta\left(\frac{d\mu}{d\lambda}(x),\frac{d\mu}{d\lambda}(y)\right)\eta_{\varepsilon}(x,y)d\lambda(x)d\lambda(y)\leq0.
\]
\end{defn}

\begin{rem}[Nonlocal Hamilton-Jacobi solutions]
 In the present work, it is the notion of nonlocal HJ subsolution
which is relevant. Nonetheless, we mention here an associated notion
of \emph{solution}. Following \cite{gangbo2019geodesics}, we say
that $\phi_{t}(x)$ is a \emph{nonlocal Hamilton-Jacobi solution}
if it is a nonlocal Hamilton-Jacobi subsolution, and moreover and
we have that
\[
\sup_{\mu\in\mathcal{P}(\mathbb{R}^{d});\lambda\gg\mu}\left\{ \int\partial_{t}\phi_{t}(x)d\mu(x)+\frac{1}{4}\int(\phi_{t}(y)-\phi_{t}(x))^{2}\theta\left(\frac{d\mu}{d\lambda}(x),\frac{d\mu}{d\lambda}(y)\right)\eta_{\varepsilon}(x,y)d\lambda(x)d\lambda(y)\right\} =0.
\]
We leave investigation of this rather atypical PDE to future work.
\end{rem}

The duality formula we expect to hold for $\mathcal{W}_{\eta,\varepsilon}$
is 
\[
\frac{1}{2}\mathcal{W}_{\eta,\varepsilon}^{2}(\mu_{0},\mu_{1})=\sup\left\{ \int\phi_{1}(x)d\mu_{1}(x)-\int\phi_{0}(x)\mu_{0}(x):\phi_{t}\in\text{HJ}_{\text{NL}}^{1}\right\} .
\]
However, in this work, we do not attempt to prove this duality formula
directly. Rather, for technical reasons, we introduce a ``smoothed
version'' of the nonlocal Wasserstein distance (for which we do prove
a partial duality result), as follows. 

\begin{defn}
Let $K$ denote the convolution kernel $c_{K}e^{-|x|}$. The ``$s$-smoothed
nonlocal Wasserstein distance'', denoted $\mathcal{W}_{\eta,\varepsilon,s}$,
is defined as follows: given $\mu_{0},\mu_{1}\in\mathcal{P}(\mathbb{R}^{d})$,
and denoting $\mu_{t}^{s}:=\boldsymbol{K_{s}}*\mu_{t}$ and $\mathbf{j}_{t}^{s}:=\boldsymbol{K_{s}}*\mathbf{j}_{t}$,
\begin{align*}
\mathcal{W}_{\eta,\varepsilon,s}^{2}(\mu_{0}^{s},\mu_{1}^{s}) & :=\inf_{(\mu_{t},\mathbf{j}_{t})_{t\in[0,1]}}\left\{ \int_{0}^{1}\mathcal{A}(\mu_{t}^{s},\mathbf{j}_{t}^{s})dt:(\mu_{t}^{s},\mathbf{j}_{t}^{s})_{t\in[0,1]}\in\mathcal{CE}(\mu_{0}^{s},\mu_{1}^{s})\right\} .
\end{align*}
\end{defn}

In other words, $\mathcal{W}_{\eta,\varepsilon,s}$ is defined in
the same variational fashion as $\mathcal{W}_{\eta,\varepsilon}$,
except we restrict to the class of a.c. curves which have been smoothed
using the mollification kernel $K_{s}$. Similarly, we also have a
notion of ``smoothed nonlocal HJ subsolution:'' 
\begin{defn}
We say that $\phi_{t}(x):[0,1]\times\mathbb{R}^{d}\rightarrow\mathbb{R}$
is an \emph{$s$-smoothed nonlocal Hamilton-Jacobi subsolution}, and
write $\phi_{t}(x)\in\text{HJ}_{\text{NL}}^{1,s}$ if 
the weak partial
derivative $\partial_{t}\phi_{t}(x):\mathbb{R}^{d}\rightarrow\mathbb{R}$
exists, and there exists some $p\in[1,\infty]$ such that $\Vert\partial_{t}\phi_{t}\Vert_{L^{p}(\mathbb{R}^{d})}<\infty$,
for Lebesgue almost all $t\in(0,1)$; and we have, for all probability
measures $\mu\in\mathcal{P}(\mathbb{R}^{d})$, for almost all $t\in(0,1)$,
\[
\int\partial_{t}\phi_{t}(x)d(\boldsymbol{K_{s}}*\mu)(x)+\frac{1}{4}\iint(\phi_{t}(y)-\phi_{t}(x))^{2}\theta\left(\frac{d(\boldsymbol{K_{s}}*\mu)}{dLeb}(x),\frac{d(\boldsymbol{K_{s}}*\mu)}{dLeb}(y)\right)\eta_{\varepsilon}(x,y)dxdy\leq0.
\]
\end{defn}

\begin{rem}
The regularity assumption we have imposed on $\partial_{t}\phi_{t}(x)$
is chosen so that $|\int\partial_{t}\phi_{t}d\boldsymbol{K_{s}}*\mu|<\infty$
for almost all $t\in(0,1)$. Indeed, $\boldsymbol{K_{s}}*\mu$ has
density $K_{s}*\mu$ with respect to the Lebesgue measure, and since
$\boldsymbol{K_{s}}*\mu$ is a probability measure, $K_{s}*\mu\in L^{1}(\mathbb{R}^{d})$.
At the same time, Young's convolution inequality implies that $K_{s}*\mu\in L^{\infty}(\mathbb{R}^{d})$
since $e^{-|x|}\in L^{\infty}$. Hence by interpolation, $K_{s}*\mu\in L^{q}$
for all $q\in[1,\infty]$, so it suffices that $\partial_{t}\phi_{t}(x)\in L^{p}(\mathbb{R}^{d})$
for some $p\in[1,\infty]$, for almost all $t\in(0,1)$. In our present
situation, it will always be the case that $\phi_{t}(x)\in BL([0,1]\times\mathbb{R}^{d})$,
so in particular we can take $p=\infty$.
\end{rem}

\begin{lemma}
\label{lem:s-smoothed nonlocal comparison}For all $\mu_{0},\mu_{1}\in\mathcal{P}_{2}(\mathbb{R}^{d})$,
it holds that
\[
\mathcal{W}_{\eta,\varepsilon}(\boldsymbol{K_{s}}*\mu_{0},\boldsymbol{K_{s}}*\mu_{1})\leq\mathcal{W}_{\eta,\varepsilon,s}(\boldsymbol{K_{s}}*\mu_{0},\boldsymbol{K_{s}}*\mu_{1})\leq\mathcal{W}_{\eta,\varepsilon}(\mu_{0},\mu_{1}).
\]
\end{lemma}

\begin{proof}
That $\mathcal{W}_{\eta,\varepsilon}(\boldsymbol{K_{s}}*\mu_{0},\boldsymbol{K_{s}}*\mu_{1})\leq\mathcal{W}_{\eta,\varepsilon,s}(\boldsymbol{K_{s}}*\mu_{0},\boldsymbol{K_{s}}*\mu_{1})$
holds is immediate from the fact that the infimum in the definition
of $\mathcal{W}_{\eta,\varepsilon,s}$ runs only over a.c. curves
connecting $\boldsymbol{K_{s}}*\mu_{0}$ and $\boldsymbol{K_{s}}*\mu_{1}$
which happen to be $K_{s}$-smoothings of (not necessarily a.c.) curves
between $\mu_{0}$ and $\mu_{1}$; whereas, for $\mathcal{W}_{\eta,\varepsilon}(\boldsymbol{K_{s}}*\mu_{0},\boldsymbol{K_{s}}*\mu_{1})$,
the infimum runs over all a.c. curves connecting $\boldsymbol{K_{s}}*\mu_{0}$
and $\boldsymbol{K_{s}}*\mu_{1}$.

That $\mathcal{W}_{\eta,\varepsilon,s}(\boldsymbol{K_{s}}*\mu_{0},\boldsymbol{K_{s}}*\mu_{1})\leq\mathcal{W}_{\eta,\varepsilon}(\mu_{0},\mu_{1})$
follows from the fact that mollifying a.c. curves reduces their total
action, see Proposition \ref{prop:convolution action stability}; and
that the class of time-dependent mass-flux pairs $(\mu_{t},\mathbf{j}_{t})_{t\in[0,1]}$
considered in the infimum defining $\mathcal{W}_{\eta,\varepsilon,s}$
includes all a.c. curves connecting $\mu_{0}$ and $\mu_{1}$.
\end{proof}
\begin{prop}
\label{prop:s-smooth duality inequality} Assume that $\eta$ and $\theta$ satisfy Assumptions \ref{assu:eta properties} (i-iv)
and \ref{assu:theta properties} respectively. Let $s>0$. Denote $\mu^{s}:=\boldsymbol{K_{s}}*\mu$
for any $\mu\in\mathcal{P}(\mathbb{R}^{d})$. The following duality
inequality for $\mathcal{W}_{\eta,\varepsilon,s}$ holds: 
\begin{align*}
\frac{1}{2}\mathcal{W}_{\eta,\varepsilon,s}^{2}(\mu_{0}^{s},\mu_{1}^{s}) & \geq\sup\left\{ \int\phi_{1}(x)d\mu_{1}^{s}(x)-\int\phi_{0}(x)d\mu_{0}^{s}(x):\phi_{t}\in\text{HJ}_{\text{NL}}^{1,s}\cap BL([0,1]\times\mathbb{R}^{d})\right\} .
\end{align*}
\end{prop}

\begin{proof}
We follow rather closely the argument given in \cite[Section 3]{erbar2019geometry}.
Following the exposition there, we introduce the shorthand notation
\[
\langle\phi,\mu\rangle:=\int_{\mathbb{R}^{d}}\phi(x)d\mu(x);\qquad\langle\langle\Phi,\mathbf{j}\rangle\rangle:=\iint_{G}\Phi(x,y)\eta(x,y)d\mathbf{j}(x,y).
\]
We also use the notation $\mu^{s}:=\boldsymbol{K_{s}}*\mu$ and $\mathbf{j}^{s}:=\boldsymbol{K_{s}}*\mathbf{j}$. 
In this argument, we allow $\mu_{t}$ to take values in $\mathcal{M}^{+}(\mathbb{R}^{d})$
rather than just $\mathcal{P}(\mathbb{R}^{d})$. Note that in this
situation, the action $\mathcal{A}(\mu,\mathbf{j})$ and the nonlocal continuity equation are still well-defined. 
Additionally,
there will be no loss of generality in assuming that, for paths we consider, $\mathcal{A}(\mu_{t}^{s},\mathbf{j}_{t}^{s})<\infty$
for almost all $t\in[0,1]$; since otherwise, we will find that $\mathcal{W}_{\eta,\varepsilon,s}^{2}(\boldsymbol{K_{s}}*\mu_{0},\boldsymbol{K_{s}}*\mu_{1})=\infty$,
and so the proposition holds trivially. In particular, quantification
over ``all'' $(\mu_{t},\mathbf{j}_{t})_{t\in[0,1]}$ shall be understood
to mean that:
\begin{itemize}
\item both $\mu_{t}:[0,1]\rightarrow\mathcal{M}^{+}(\mathbb{R}^{d})$ is
continuous, and $\mathbf{j}_{t}:[0,1]\rightarrow\mathcal{M}_{loc}(G)$
is Borel measurable, w.r.t. the respective weak{*} topologies on $\mathcal{M}^{+}(\mathbb{R}^{d})$
and $\mathcal{M}_{loc}(G)$; and
\item $\mathcal{A}(\mu_{t}^{s},\mathbf{j}_{t}^{s})<\infty$ for almost all
$t\in[0,1]$. 
\end{itemize}

By \cite[Lemma 3.1]{erbar2014gradient}, we know that, if $(\mu_{t}^{s},\mathbf{j}_{t}^{s})_{t\in[0,1]}$
satisfies the nonlocal continuity equation in the sense of distributions,
then for all $\phi_{t}(x)\in C_{c}^{\infty}([0,1]\times\mathbb{R}^{d})$,
\[
\langle\phi_{1},\mu_{1}^{s}\rangle-\langle\phi_{0},\mu_{0}^{s}\rangle-\int_{0}^{1}\left(\langle\partial_{t}\phi_{t},\mu_{t}^{s}\rangle+\frac{1}{2}\langle\langle\bar{\nabla}\phi_{t},\mathbf{j}_{t}^{s}\rangle\rangle\right)dt=0.
\]
Now consider, more generally, the case where $\tilde{\phi}_{t}(x)\in BL([0,1]\times\mathbb{R}^{d})$.
Reasoning as in \cite[Remark 3.3]{erbar2014gradient}, we approximate
$\tilde{\phi}_{t}(x)$ by functions $\phi_{t}\in C_{c}^{\infty}([0,1]\times\mathbb{R}^{d})$
which are uniformly bounded in $C^{1}([0,1]\times\mathbb{R}^{d})$
norm, and use the fact that 
\[
|\langle\langle\bar{\nabla}\phi_{t},\mathbf{j}_{t}^{s}\rangle\rangle|\leq\left(\sup_{t\in[0,1]}\Vert\phi_{t}\Vert_{C^{1}(\mathbb{R}^{d})}\right)\iint(1\wedge|x-y|)\eta_{\varepsilon}(x,y)d|\mathbf{j}_{t}^{s}|(x,y)\in L^{1}([0,1])
\]
to pass to the limit in the nonlocal continuity equation, and deduce
that more generally,
\begin{equation}
\forall\phi_{t}(x)\in BL([0,1]\times\mathbb{R}^{d})\quad\langle\phi_{1},\mu_{1}^{s}\rangle-\langle\phi_{0},\mu_{0}^{s}\rangle-\int_{0}^{1}\left(\langle\partial_{t}\phi_{t},\mu_{t}^{s}\rangle+\frac{1}{2}\langle\langle\bar{\nabla}\phi_{t},\mathbf{j}_{t}^{s}\rangle\rangle\right)dt=0.\label{eq:nce with bl}
\end{equation}
Therefore, for any two $s$-smoothed probability measures $\bar{\mu}_{0}^{s}$
and $\bar{\mu}_{1}^{s}$ with $\mathcal{W}_{\eta,\varepsilon,s}^{2}(\bar{\mu}_{0}^{s},\bar{\mu}_{1}^{s})<\infty$,
we have that
\[
\frac{1}{2}\mathcal{W}_{\eta,\varepsilon,s}^{2}(\bar{\mu}_{0}^{s},\bar{\mu}_{1}^{s})=\inf_{(\mu_{t},\mathbf{j}_{t})_{t\in[0,1]}:\mu_{0}=\bar{\mu}_{0},\mu_{1}=\bar{\mu}_{1}}\left\{ \int_{0}^{1}\frac{1}{2}\mathcal{A}(\mu_{t}^{s},\mathbf{j}_{t}^{s})dt:(\mu_{t}^{s},\mathbf{j}_{t}^{s})_{t\in[0,1]}\text{ satisfies }(\ref{eq:nce with bl})\right\} .
\]

Introducing a Lagrange multiplier for the constraint ``$(\mu_{t}^{s},\mathbf{j}_{t}^{s})_{t\in[0,1]}\text{ satisfies }(\ref{eq:nce with bl})$''
we see that 
\begin{multline*}
\frac{1}{2}\mathcal{W}_{\eta,\varepsilon,s}^{2}(\bar{\mu}_{0}^{s},\bar{\mu}_{1}^{s})=\inf_{(\mu_{t},\mathbf{j}_{t})_{t\in[0,1]}:\mu_{0}=\bar{\mu}_{0},\mu_{1}=\bar{\mu}_{1}}\sup_{\phi_{t}\in BL([0,1]\times\mathbb{R}^{d})}\bigg\{\int_{0}^{1}\frac{1}{2}\mathcal{A}(\mu_{t}^{s},\mathbf{j}_{t}^{s})dt\\
+\langle\phi_{1},\mu_{1}^{s}\rangle-\langle\phi_{0},\mu_{0}^{s}\rangle-\int_{0}^{1}\left(\langle\partial_{t}\phi_{t},\mu_{t}^{s}\rangle+\frac{1}{2}\langle\langle\bar{\nabla}\phi_{t},\mathbf{j}_{t}^{s}\rangle\rangle\right)dt\bigg\}.
\end{multline*}
(Note that the the inner supremum takes the value $+\infty$ unless
(\ref{eq:nce with bl}) holds.)

Using the general fact that $\sup\inf\leq\inf\sup$, we see that
\begin{multline*}
\frac{1}{2}\mathcal{W}_{\eta,\varepsilon,s}^{2}(\bar{\mu}_{0}^{s},\bar{\mu}_{1}^{s})\geq\sup_{\phi_{t}\in BL([0,1]\times\mathbb{R}^{d})}\inf_{(\mu_{t},\mathbf{j}_{t})_{t\in[0,1]}:\mu_{0}=\bar{\mu}_{0},\mu_{1}=\bar{\mu}_{1}}\bigg\{\langle\phi_{1},\mu_{1}^{s}\rangle-\langle\phi_{0},\mu_{0}^{s}\rangle\\
+\int_{0}^{1}\left(\frac{1}{2}\mathcal{A}(\mu_{t}^{s},\mathbf{j}_{t}^{s})-\langle\partial_{t}\phi_{t},\mu_{t}^{s}\rangle-\frac{1}{2}\langle\langle\bar{\nabla}\phi_{t},\mathbf{j}_{t}^{s}\rangle\rangle\right)dt\bigg\}
\end{multline*}
which in turn implies that (now letting the infimum quantify over
a larger set, without fixed endpoints $\bar{\mu}_{0}$ and $\bar{\mu}_{1}$)
\begin{multline*}
\frac{1}{2}\mathcal{W}_{\eta,\varepsilon,s}^{2}(\bar{\mu}_{0}^{s},\bar{\mu}_{1}^{s})\geq\sup_{\phi_{t}\in BL([0,1]\times\mathbb{R}^{d})}\bigg\{\langle\phi_{1},\bar{\mu}_{1}^{s}\rangle-\langle\phi_{0},\bar{\mu}_{0}^{s}\rangle\\
+\inf_{(\mu_{t},\mathbf{j}_{t})_{t\in[0,1]}}\int_{0}^{1}\left(\frac{1}{2}\mathcal{A}(\mu_{t}^{s},\mathbf{j}_{t}^{s})-\langle\partial_{t}\phi_{t},\mu_{t}^{s}\rangle-\frac{1}{2}\langle\langle\bar{\nabla}\phi_{t},\mathbf{j}_{t}^{s}\rangle\rangle\right)dt\bigg\}.
\end{multline*}
Observe that due to the 1-homogeneity in $(\mu_{t}^{s},\mathbf{j}_{t}^{s})_{t\in[0,1]}$
of both the total action and nonlocal continuity equation, the inner
infimum evaluates to $-\infty$ unless $\phi_{t}$ is chosen so that,
for all $(\mu_{t},\mathbf{j}_{t})_{t\in[0,1]}$, 
\[
\int_{0}^{1}\left(\frac{1}{2}\mathcal{A}(\mu_{t}^{s},\mathbf{j}_{t}^{s})-\langle\partial_{t}\phi_{t},\mu_{t}^{s}\rangle-\frac{1}{2}\langle\langle\bar{\nabla}\phi_{t},\mathbf{j}_{t}^{s}\rangle\rangle\right)dt\geq0
\]
since otherwise we can just replace $(\mu_{t}^{s},\mathbf{j}_{t}^{s})$
with $(\lambda\mu_{t}^{s},\lambda\mathbf{j}_{t}^{s})$ and then send
$\lambda\rightarrow\infty$. %

At the same time,  since $\mu_{t}^{s}$ has full support, and $\mathcal{A}(\mu_{t}^{s},\mathbf{j}_{t}^{s})<\infty$
by assumption, we have by \cite[Lemma 2.3]{erbar2014gradient} that
$\mathbf{j}_{t}^{s}\ll dx\otimes dy$. So, we can compute that
\begin{multline*}
\frac{1}{2}\mathcal{A}(\mu_{t}^{s},\mathbf{j}_{t}^{s})-\frac{1}{2}\langle\langle\bar{\nabla}\phi_{t},\mathbf{j}_{t}^{s}\rangle\rangle=\frac{1}{4}\iint\frac{\left(\frac{d\mathbf{j}_{t}^{s}}{dxdy}(x,y)-\bar{\nabla}\phi_{t}(x,y)\theta\left(\frac{d\mu_{t}^{s}}{dx}(x),\frac{d\mu_{t}^{s}}{dy}(y)\right)\right)^{2}}{\theta\left(\frac{d\mu_{t}^{s}}{dx}(x),\frac{d\mu_{t}^{s}}{dy}(y)\right)}\eta_{\varepsilon}(x,y)dxdy\\
-\frac{1}{4}\iint(\bar{\nabla}\phi_{t}(x,y))^{2}\theta\left(\frac{d\mu_{t}^{s}}{dx}(x),\frac{d\mu_{t}^{s}}{dy}(y)\right)(x,y)dxdy.
\end{multline*}
Therefore, we conclude that %
\begin{multline}
\int_{0}^{1} \int\partial_{t}\phi_{t}d\mu_{t}^{s}+\frac{1}{4}\iint(\bar{\nabla}\phi_{t}(x,y))^{2}\theta\left(\frac{d\mu_{t}^{s}}{dx}(x),\frac{d\mu_{t}^{s}}{dy}(y)\right)\eta_{\varepsilon}(x,y)dxdy\,dt\\
\leq\int_{0}^{1} \frac{1}{4}\iint\frac{\left(\frac{d\mathbf{j}_{t}^{s}}{dxdy}(x,y)-\bar{\nabla}\phi_{t}(x,y)\theta\left(\frac{d\mu_{t}^{s}}{dx}(x),\frac{d\mu_{t}^{s}}{dy}(y)\right)\right)^{2}}{\theta\left(\frac{d\mu_{t}^{s}}{dx}(x),\frac{d\mu_{t}^{s}}{dy}(y)\right)}\eta_{\varepsilon}(x,y)dxdy\,dt.\label{eq:s-smoothed HJ inequality}
\end{multline}
Again, this condition holds provided that the inner infimum is $\geq0$
(as opposed to $-\infty$). Of course, if it \emph{is} the case that
the inner infimum is nonnegative (for a given $\phi_{t}(x)\in BL([0,1]\times\mathbb{R}^{d})$),
we certainly have that 
\begin{equation*}
 \langle\phi_{1},\bar{\mu}_{1}^{s}\rangle-\langle\phi_{0},\bar{\mu}_{0}^{s}\rangle+\inf_{(\mu_{t},\mathbf{j}_{t})_{t\in[0,1]}}\int_{0}^{1}\left(\frac{1}{2}\mathcal{A}(\mu_{t}^{s},\mathbf{j}_{t}^{s})-\langle\partial_{t}\phi_{t},\mu_{t}^{s}\rangle-\frac{1}{2}\langle\langle\bar{\nabla}\phi_{t},\mathbf{j}_{t}^{s}\rangle\rangle\right)dt\geq\langle\phi_{1},\bar{\mu}_{1}^{s}\rangle-\langle\phi_{0},\bar{\mu}_{0}^{s}\rangle.
\end{equation*}
Therefore, we deduce the duality relation
\[
\frac{1}{2}\mathcal{W}_{\eta,\varepsilon,s}^{2}(\bar{\mu}_{0}^{s},\bar{\mu}_{1}^{s})\geq\sup_{\phi_{t}\in BL([0,1]\times\mathbb{R}^{d})}\{\langle\phi_{1},\bar{\mu}_{1}^{s}\rangle-\langle\phi_{0},\bar{\mu}_{0}^{s}\rangle:\forall(\mu_{t},\mathbf{j}_{t})_{t\in[0,1]}\text{ (\ref{eq:s-smoothed HJ inequality}) holds}\}.
\]
In turn, since in general, 
\[
\int_{0}^{1} \frac{1}{4}\iint\frac{\left(\frac{d\mathbf{j}_{t}^{s}}{dxdy}(x,y)-\bar{\nabla}\phi_{t}(x,y)\theta\left(\frac{d\mu_{t}^{s}}{dx}(x),\frac{d\mu_{t}^{s}}{dy}(y)\right)\right)^{2}}{\theta\left(\frac{d\mu_{t}^{s}}{dx}(x),\frac{d\mu_{t}^{s}}{dy}(y)\right)}\eta_{\varepsilon}(x,y)dxdy\, dt\geq0,
\]
this implies that 
\begin{multline*}
\frac{1}{2}\mathcal{W}_{\eta,\varepsilon,s}^{2}(\bar{\mu}_{0}^{s},\bar{\mu}_{1}^{s})\geq\sup_{\phi_{t}\in BL([0,1]\times\mathbb{R}^{d})}\bigg\{\langle\phi_{1},\bar{\mu}_{1}^{s}\rangle-\langle\phi_{0},\bar{\mu}_{0}^{s}\rangle:\\
\forall(\mu_{t},\mathbf{j}_{t})_{t\in[0,1]},\int_{0}^{1}\int\partial_{t}\phi_{t}d\mu_{t}^{s}+\frac{1}{4}\iint(\bar{\nabla}\phi_{t}(x,y))^{2}\theta\left(\frac{d\mu_{t}^{s}}{dx}(x),\frac{d\mu_{t}^{s}}{dy}(y)\right)\eta_{\varepsilon}(x,y)dxdy\,dt\leq0\bigg\}.
\end{multline*}
this implies that 
\begin{multline*}
\frac{1}{2}\mathcal{W}_{\eta,\varepsilon,s}^{2}(\bar{\mu}_{0}^{s},\bar{\mu}_{1}^{s})\geq\sup_{\phi_{t}\in BL([0,1]\times\mathbb{R}^{d})}\bigg\{\langle\phi_{1},\bar{\mu}_{1}^{s}\rangle-\langle\phi_{0},\bar{\mu}_{0}^{s}\rangle:\\
\forall(\mu_{t},\mathbf{j}_{t})_{t\in[0,1]},\underbrace{\int_{0}^{1}\int\partial_{t}\phi_{t}d\mu_{t}^{s}+\frac{1}{4}\iint(\bar{\nabla}\phi_{t}(x,y))^{2}\theta\left(\frac{d\mu_{t}^{s}}{dx}(x),\frac{d\mu_{t}^{s}}{dy}(y)\right)\eta_{\varepsilon}(x,y)dxdy\,dt}_{(\dagger)}\leq0\bigg\}.
\end{multline*}
However, the quantity $(\dagger)$ is independent of $\mathbf{j}_{t}$. Therefore the statement reduces
to quantification over $(\mu_{t})_{t\in[0,1]}$ is for all weak{*}ly
continuous ``curves'' in $\mathcal{M}^{+}(\mathbb{R}^{d})$, such
that there exists a $(\mathbf{j}_{t})_{t\in[0,1]}$ for which $\mathcal{A}(\mu_{t}^{s},\mathbf{j}_{t}^{s})<\infty$
for almost all $t\in[0,1]$; but we do not even require satisfaction
of the nonlocal continuity equation, so we can always just take $\mathbf{j}_{t}=0$.
From this, by restricting the quantification only to ``curves''
$(\mu_{t})_{t\in[0,1]}$ which are constant and belong to $\mathcal{P}(\mathbb{R}^{d})$,
we deduce that
\begin{multline*}
\frac{1}{2}\mathcal{W}_{\eta,\varepsilon,s}^{2}(\bar{\mu}_{0}^{s},\bar{\mu}_{1}^{s})\geq\sup_{\phi_{t}\in BL([0,1]\times\mathbb{R}^{d})}\bigg\{\langle\phi_{1},\bar{\mu}_{1}^{s}\rangle-\langle\phi_{0},\bar{\mu}_{0}^{s}\rangle:\\
\forall\mu\in\mathcal{P}(\mathbb{R}^{d}),\int\partial_{t}\phi_{t}d\mu^{s}+\frac{1}{4}\iint(\bar{\nabla}\phi_{t}(x,y))^{2}\theta\left(\frac{d\mu_{t}^{s}}{dx}(x),\frac{d\mu_{t}^{s}}{dy}(y)\right)\eta_{\varepsilon}(x,y)dxdy\leq0\quad t\text{-a.s.}\bigg\}
\end{multline*}
as desired.
\end{proof}


\begin{prop}
\label{prop:local-to-nonlocal subsolution}Let $\varepsilon\in(0,1]$
and $s\geq\varepsilon$. Assume that $\eta$ and $\theta$ satisfy Assumptions \ref{assu:eta properties} (i-iv)
and \ref{assu:theta properties} respectively. Assume also that $M_{5}(\eta)<\infty$. %
{} Suppose that $(\phi_{t})_{t\in[0,1]}\in BL([0,1]\times\mathbb{R}^{d})$
is a Hamilton-Jacobi subsolution %
{} in the sense that %
\[
\partial_{t}\phi_{t}(x)+\frac{1}{2}|\nabla\phi_{t}(x)|^{2}\leq0\text{ a.e.}.
\]
Then, 
\[
\left(\frac{2d}{\varepsilon^{2}M_{2}(\eta)}K_{s}*\phi_{t}-\frac{CA^{2}}{\varepsilon s}t\right)_{t\in[0,1]};
\]
\[
C=\frac{d^{2}}{M_{2}(\eta)^{2}}\left[\frac{3}{8}M_{3}(\eta)+\sqrt{\left(\frac{M_{2}(\eta)}{d}+\frac{3}{2}M_{3}(\eta)\right)\left(M_{4}(\eta)+\frac{3}{2}M_{5}(\eta)\right)}+\frac{1}{4}\left(M_{4}(\eta)+\frac{3}{2}M_{5}(\eta)\right)\right]
\]
\[
A=\sup_{t\in[0,1]}\Vert\nabla\phi_{t}\Vert_{L^{\infty}(\mathbb{R}^{d})}
\]
is an $s$-smooth nonlocal Hamilton-Jacobi subsolution, belonging
to $BL([0,1]\times\mathbb{R}^{d})$.
\end{prop}

\begin{proof}
Let $\phi_{t}\in BL([0,1]\times\mathbb{R}^{d})$ satisfy $\partial_{t}\phi_{t}(x)+\frac{1}{2}|\nabla\phi_{t}(x)|^{2}\leq0$
almost everywhere. Then $\phi_{t}^{s}(x):=(K_{s}*\phi_{t})(x)$
also satisfies 
\[
\partial_{t}\phi_{t}^{s}(x)+\frac{1}{2}|\nabla\phi_{t}^{s}(x)|^{2}\leq0\quad(t,x)\text{-a.e.}
\]
by Lemma \ref{lem:convolution of HJ subsoln}. Additionally, we observe
that $\phi_{t}^{s}\in BL([0,1]\times\mathbb{R}^{d})$, thanks to Lemma
\ref{lem:BL stable convolution in space}.

We claim that some slight modification of 
\[
\tilde{\phi}_{t}^{s,\varepsilon}:=\frac{2d}{\varepsilon^{2}M_{2}(\eta)}\phi_{t}^{s},
\]
more precisely 
\[
\check{\phi}_{t}^{s,\varepsilon}:=\tilde{\phi}_{t}^{s,\varepsilon}-\frac{CA^{2}}{\varepsilon s}t
\]
where $C>0$ and $A>0$ are some constants to be determined later
on, is a subsolution to the \emph{$s$-smoothed nonlocal }HJ equation. 

First, observe that since $\partial_{t}\phi_{t}^{s}(x)+\frac{1}{2}|\nabla\phi_{t}^{s}(x)|^{2}\leq0$,
it follows that 
\[
\frac{2d}{\varepsilon^{2}M_{2}(\eta)}\partial_{t}\phi_{t}^{s}(x)+\frac{2d}{\varepsilon^{2}M_{2}(\eta)}\frac{1}{2}|\nabla\phi_{t}^{s}(x)|^{2}\leq0
\]
and so
\[
\partial_{t}\tilde{\phi}_{t}^{s,\varepsilon}+\frac{1}{4}\varepsilon^{2}\frac{M_{2}(\eta)}{d}|\nabla\tilde{\phi}_{t}^{s,\varepsilon}(x)|^{2}\leq0.
\]
Second, note that replacing $\tilde{\phi}_{t}^{s,\varepsilon}$
with $\check{\phi}_{t}^{s,\varepsilon}:=\tilde{\phi}_{t}^{s,\varepsilon}-\frac{CA^{2}}{\varepsilon s}t$
leaves the gradient unchanged (that is, $\nabla\tilde{\phi}_{t}^{s,\varepsilon}(x)=\nabla\check{\phi}_{t}^{s,\varepsilon}(x)$),
whereas $\partial_{t}(\tilde{\phi}_{t}^{s,\varepsilon}-\frac{CA^{2}}{\varepsilon s}t)=\partial_{t}\tilde{\phi}_{t}^{s,\varepsilon}-\frac{CA^{2}}{\varepsilon s}$.
In particular, given that 
\[
\partial_{t}\tilde{\phi}_{t}^{s,\varepsilon}+\frac{1}{4}\varepsilon^{2}\frac{M_{2}(\eta)}{d}|\nabla\tilde{\phi}_{t}^{s,\varepsilon}(x)|^{2}\leq0
\]
we have that 
\[
\partial_{t}\left(\tilde{\phi}_{t}^{s,\varepsilon}-\frac{CA^{2}}{\varepsilon s}t\right)+\frac{1}{4}\varepsilon^{2}\frac{M_{2}(\eta)}{d}|\nabla\left(\tilde{\phi}_{t}^{s,\varepsilon}-\frac{CA^{2}}{\varepsilon s}t\right)(x)|^{2}\leq-\frac{CA^{2}}{\varepsilon s}
\]
so we know that $\check{\phi}_{t}^{s,\varepsilon}$ is also a (local)
HJ subsolution. At the same time, it is clear that $\check{\phi}_{t}^{s,\varepsilon}\in BL([0,1]\times\mathbb{R}^{d})$,
from the fact that $\phi_{t}^{s}\in BL([0,1]\times\mathbb{R}^{d})$. 

Given some arbitrary $\mu\in\mathcal{P}(\mathbb{R}^{d})$, we denote
$\rho:=K_{s}*\mu$. Note that $\rho$ is a Lebesgue density. Then,
we see that 
\begin{multline*}
\int\partial_{t}\check{\phi}_{t}^{s,\varepsilon}(x)d\rho(x)+\frac{1}{4}\int\int(\check{\phi}_{t}^{s,\varepsilon}(y)-\check{\phi}_{t}^{s,\varepsilon}(x))^{2}\theta(\rho(x),\rho(y))\eta_{\varepsilon}(x,y)dxdy\\
\begin{aligned} & \leq-\frac{1}{4}\varepsilon^{2}\frac{M_{2}(\eta)}{d}\int|\nabla\check{\phi}_{t}^{s,\varepsilon}(x)|^{2}d\rho(x)-\frac{CA^{2}}{\varepsilon s}+\frac{1}{4}\iint(\check{\phi}_{t}^{s,\varepsilon}(y)-\check{\phi}_{t}^{s,\varepsilon}(x))^{2}\theta(\rho(x),\rho(y))\eta_{\varepsilon}(x,y)dxdy\\
 & \leq-\frac{1}{4}\varepsilon^{2}\frac{M_{2}(\eta)}{d}\int|\nabla\check{\phi}_{t}^{s,\varepsilon}(x)|^{2}d\rho(x)-\frac{CA^{2}}{\varepsilon s} \\
 & \phantom{\leq} \;\, +\frac{1}{4}\iint(\check{\phi}_{t}^{s,\varepsilon}(y)-\check{\phi}_{t}^{s,\varepsilon}(x))^{2}\left(1+\frac{3}{2s}|y-x|\right)\rho(x)\eta_{\varepsilon}(x,y)dxdy
\end{aligned}
\end{multline*}
where we have used the fact that $K=c_{K}e^{-|x|}$, which implies (thanks to Lemma \ref{lem:relative Lipschitz})
that
\[
\theta(\rho(x),\rho(y))=\rho(x)\theta\left(1,\frac{\rho(y)}{\rho(x)}\right)\leq\rho(x)\theta\left(1,1+\frac{3}{s}|y-x|\right)\leq\rho(x)\left(1+\frac{3}{2s}|y-x|\right).
\]
Next, we replace $(\check{\phi}_{t}^{s,\varepsilon}(y)-\check{\phi}_{t}^{s,\varepsilon}(x))^{2}$
with an expression involving $|\nabla\check{\phi}_{t}^{s,\varepsilon}|^{2}$,
using a Taylor approximation. Note that such a Taylor series approximation
requires sufficient regularity of $\check{\phi}_{t}^{s,\varepsilon}$
(which is why we are using the smooth potential $\check{\phi}_{t}^{s,\varepsilon}$
as our candidate $s$-smoothed nonlocal HJ subsolution). 

To wit, 
\[
|(\check{\phi}_{t}^{s,\varepsilon}(y)-\check{\phi}_{t}^{s,\varepsilon}(x))|-|\nabla\check{\phi}_{t}^{s,\varepsilon}(x)\cdot(x-y)|\leq|(\check{\phi}_{t}^{s,\varepsilon}(y)-\check{\phi}_{t}^{s,\varepsilon}(x))-\nabla\check{\phi}_{t}^{s,\varepsilon}(x)\cdot(x-y)|\leq\frac{1}{2}\Vert D^{2}\check{\phi}_{t}^{s,\varepsilon}\Vert_{\infty}|x-y|^{2}
\]
and so
\[
(\check{\phi}_{t}^{s,\varepsilon}(y)-\check{\phi}_{t}^{s,\varepsilon}(x))^{2}\leq\frac{1}{4}\Vert D^{2}\check{\phi}_{t}^{s,\varepsilon}\Vert_{\infty}^{2}|x-y|^{4}+\Vert D^{2}\check{\phi}_{t}^{s,\varepsilon}\Vert_{\infty}|x-y|^{2}|\nabla\check{\phi}_{t}^{s,\varepsilon}(x)\cdot(x-y)|+|\nabla\check{\phi}_{t}^{s,\varepsilon}(x)\cdot(x-y)|^{2}.
\]
And, since $|\nabla\check{\phi}_{t}^{s,\varepsilon}(x)\cdot(x-y)|\leq|\nabla\check{\phi}_{t}^{s,\varepsilon}(x)||x-y|$,
it follows that 
\begin{multline*}
\frac{1}{4}\iint(\check{\phi}_{t}^{s,\varepsilon}(y)-\check{\phi}_{t}^{s,\varepsilon}(x))^{2}\left(1+\frac{3}{2s}|y-x|\right)\rho(x)\eta_{\varepsilon}(x,y)dxdy\\
\begin{aligned} & \leq\frac{1}{4}\iint|\nabla\check{\phi}_{t}^{s,\varepsilon}(x)\cdot(x-y)|^{2}\left(1+\frac{3}{2s}|y-x|\right)\rho(x)\eta_{\varepsilon}(x,y)dxdy\\
 & \quad+\frac{1}{4}\iint\Vert D^{2}\check{\phi}_{t}^{s,\varepsilon}\Vert_{\infty}|x-y|^{2}|\nabla\check{\phi}_{t}^{s,\varepsilon}(x)\cdot(x-y)|\left(1+\frac{3}{2s}|y-x|\right)\rho(x)\eta_{\varepsilon}(x,y)dxdy\\
 & \quad+\frac{1}{16}\iint\Vert D^{2}\check{\phi}_{t}^{s,\varepsilon}\Vert_{\infty}^{2}|x-y|^{4}\left(1+\frac{3}{2s}|y-x|\right)\rho(x)\eta_{\varepsilon}(x,y)dxdy\\
 & :=I_{\varepsilon}+II_{\varepsilon}+III_{\varepsilon}.
\end{aligned}
\end{multline*}
Since (as we will explicitly show below) the latter two terms are
higher order in $\varepsilon$, we initially focus our attention on
the first term $I_{\varepsilon}$.

In what follows, recall our notation $M_{p}(\eta):=\int|x-y|^{p}\eta(|x-y|)dy$;
note that this does not depend on the choice of $x\in\mathbb{R}^{d}$.
Additionally, observe that $M_{p}(\eta_{\varepsilon})=\varepsilon^{p}M_{p}(\eta).$

From \cite[Corollary 2.16]{garciatrillos2020gromov} (together with
Lemma \ref{lem:moment versus profile}), we know that 
\[
\frac{1}{4}\iint|\nabla\check{\phi}_{t}^{s,\varepsilon}(x)\cdot(x-y)|^{2}\rho(x)\eta_{\varepsilon}(x,y)dxdy=\frac{1}{4}\varepsilon^{2}\frac{M_{2}(\eta)}{d}\int|\nabla\check{\phi}_{t}^{s,\varepsilon}(x)|^{2}\rho(x)dx.
\]
Similarly, compute that 
\begin{align*}
\iint|\nabla\check{\phi}_{t}^{s,\varepsilon}(x)\cdot(x-y)|^{2}\left(\frac{3}{2s}|x-y|\right)\rho(x)\eta_{\varepsilon}(x,y)dxdy & \leq\frac{3}{2s}\iint|\nabla\check{\phi}_{t}^{s,\varepsilon}(x)|^{2}|x-y|^{3}\rho(x)\eta_{\varepsilon}(x,y)dxdy\\
 & =\frac{3}{2}\frac{\varepsilon^{3}}{s}M_{3}(\eta)\int|\nabla\check{\phi}_{t}^{s,\varepsilon}(x)|^{2}\rho(x)dx.
\end{align*}
Consequently,
\[
I_{\varepsilon}\leq\left(\frac{1}{4}\varepsilon^{2}\frac{M_{2}(\eta)}{d}+\frac{3}{8}\frac{\varepsilon^{3}}{s}M_{3}(\eta)\right)\int|\nabla\check{\phi}_{t}^{s,\varepsilon}(x)|^{2}\rho(x)dx
\]
and so
\begin{multline*}
\frac{1}{4}\iint(\check{\phi}_{t}^{s,\varepsilon}(y)-\breve{\phi}_{t}^{s,\varepsilon}(x))^{2}\left(1+\frac{3}{2s}|y-x|\right)\rho(x)\eta_{\varepsilon}(x,y)dxdy\\
\leq\left(\frac{1}{4}\frac{M_{2}(\eta)}{d}\varepsilon^{2}+\frac{3}{8}M_{3}(\eta)\frac{\varepsilon^{3}}{s}\right)\int|\nabla\check{\phi}_{t}^{s,\varepsilon}(x)|^{2}\rho(x)dx+II_{\varepsilon}+III_{\varepsilon}.
\end{multline*}

Therefore, we find that 

\begin{align*}
& \int\partial_{t}\check{\phi}_{t}^{s,\varepsilon}(x)d\rho(x)+\frac{1}{4}\int\int(\check{\phi}_{t}^{s,\varepsilon}(y)-\check{\phi}_{t}^{s,\varepsilon}(x))^{2}\theta(\rho(x),\rho(y))\eta_{\varepsilon}(x,y)dxdy \\
 & \leq -\frac{1}{4}\varepsilon^{2}\frac{M_{2}(\eta)}{d}\int|\nabla\check{\phi}_{t}^{s,\varepsilon}(x)|^{2}d\rho(x)-\frac{CA^{2}}{\varepsilon s}+\left(\frac{1}{4}\frac{M_{2}(\eta)}{d}\varepsilon^{2}+\frac{3}{8}M_{3}(\eta)\frac{\varepsilon^{3}}{s}\right)\int|\nabla\check{\phi}_{t}^{s,\varepsilon}(x)|^{2}\rho(x)dx \\
 & \phantom{\leq} \:\: +II_{\varepsilon}+III_{\varepsilon}\\
 & =-\frac{CA^{2}}{\varepsilon s}+\frac{3}{8}M_{3}(\eta)\frac{\varepsilon^{3}}{s}\int|\nabla\check{\phi}_{t}^{s,\varepsilon}(x)|^{2}\rho(x)dx+II_{\varepsilon}+III_{\varepsilon}.
\end{align*}
Hence, in order for $\check{\phi}_{t}^{s,\varepsilon}$ to be an $s$-smoothed
nonlocal HJ subsolution,
it suffices to show that for $s>\varepsilon$ and an appropriate choice
of $A$ and $C$,
\[
-\frac{CA^{2}}{\varepsilon s}+\frac{3}{8}M_{3}(\eta)\frac{\varepsilon^{3}}{s}\int|\nabla\check{\phi}_{t}^{s,\varepsilon}(x)|^{2}\rho(x)dx+II_{\varepsilon}+III_{\varepsilon}\leq0.
\]
To see when this occurs, we first use the estimate %
\begin{align*}
|\nabla\check{\phi}_{t}^{s,\varepsilon}(x)|^{2} & =|\nabla\tilde{\phi}_{t}^{s,\varepsilon}(x)|^{2}\\
 & =|\nabla\left(\frac{2d}{\varepsilon^{2}M_{2}(\eta)}\phi_{t}^{s}\right)(x)|^{2}\\
 & =\frac{4d^{2}}{\varepsilon^{4}M_{2}(\eta)^{2}}|\nabla\phi_{t}^{s}(x)|^{2}\\
\text{(Lemma \ref{lem:gradient convolution contraction})} & \leq\frac{4d^{2}}{\varepsilon^{4}M_{2}(\eta)^{2}}|\nabla\phi_{t}(x)|^{2}.\text{ \ensuremath{(x\text{-a.s})}}
\end{align*}
Consequently, for any probability density $\rho(x)dx$, it holds
that 
\begin{align*}
\frac{3}{8}M_{3}(\eta)\frac{\varepsilon^{3}}{s}\int|\nabla\check{\phi}_{t}^{s,\varepsilon}(x)|^{2}\rho(x)dx & \leq\frac{3}{2}\frac{d^{2}M_{3}(\eta)}{\varepsilon sM_{2}(\eta)^{2}}\Vert\nabla\phi_{t}\Vert_{\infty}^{2}.
\end{align*}
{} We also note that this implies that 
\begin{align}
I_{\varepsilon} & \leq\left(\frac{d}{\varepsilon^{2}M_{2}(\eta)}+\frac{3}{2}\frac{d^{2}M_{3}(\eta)}{\varepsilon sM_{2}(\eta)^{2}}\right)\Vert\nabla\phi_{t}\Vert_{\infty}^{2}.\label{eq:1epsilon upper bound}
\end{align}

Having analyzed the term $I_{\varepsilon}$ from the Taylor series
expansion, we consider the first terms $II_{\varepsilon}$ and $III_{\varepsilon}$,
both of which involve the quantity $\Vert D^{2}\check{\phi}_{t}^{s,\varepsilon}\Vert_{\infty}$.
For the term $III_{\varepsilon}$: using the fact that $\Vert D^{2}\check{\phi}_{t}^{s,\varepsilon}\Vert_{\infty}=\Vert D^{2}\tilde{\phi}_{t}^{s,\varepsilon}\Vert_{\infty}=\frac{2d}{\varepsilon^{2}M_{2}(\eta)}\Vert D^{2}\phi_{t}^{s}\Vert_{\infty}$,
and also, by Lemma \ref{lem:Folland Young's ineq} that $\Vert D^{2}\phi_{t}^{s}\Vert_{\infty}^{2}\leq s^{-2}\Vert\nabla\phi_{t}\Vert_{\infty}^{2}$,
we then see that
\begin{align} \label{eq:3epsilon upper bound}
\begin{split}
III_{\varepsilon} & =\frac{1}{16}\iint\Vert D^{2}\check{\phi}_{t}^{s,\varepsilon}\Vert_{\infty}^{2}|x-y|^{4}\left(1+\frac{3}{2s}|y-x|\right)\rho(x)\eta_{\varepsilon}(x,y)dxdy\\
 & \leq\frac{1}{4}\frac{d^{2}}{\left(\varepsilon^{2}M_{2}(\eta)\right)^{2}}\frac{1}{s^{2}}\Vert\nabla\phi_{t}\Vert_{\infty}^{2}\left(M_{4}(\eta)\varepsilon^{4}+\frac{3}{2}M_{5}(\eta)\frac{\varepsilon^{5}}{s}\right)\int\rho(x)dx \\
 & =\frac{1}{4}\frac{d^{2}}{M_{2}(\eta)^{2}}\frac{1}{s^{2}}\Vert\nabla\phi_{t}\Vert_{\infty}^{2}\left(M_{4}(\eta)+\frac{3}{2}M_{5}(\eta)\frac{\varepsilon}{s}\right).
\end{split}
\end{align}

It remains to consider the second ``mixed'' term in the Taylor series
expansion. Using Holder's inequality, we observe that 
\begin{align*}
II_{\varepsilon} & :=\frac{1}{4}\iint\Vert D^{2}\check{\phi}_{t}^{s,\varepsilon}\Vert_{\infty}|x-y|^{2}|\nabla\check{\phi}_{t}^{s,\varepsilon}(x)\cdot(x-y)|\left(1+\frac{3}{2s}|y-x|\right)\rho(x)\eta_{\varepsilon}(x,y)dxdy\\
 & \leq\left(\frac{1}{4}\iint\left(|\nabla\check{\phi}_{t}^{s,\varepsilon}(x)\cdot(x-y)|\left(1+\frac{3}{2s}|y-x|\right)^{1/2}\right)^{2}\rho(x)\eta_{\varepsilon}(x,y)dxdy\right)^{1/2}\\
 & \qquad\times\left(\frac{1}{4}\iint\left(\Vert D^{2}\check{\phi}_{t}^{s,\varepsilon}\Vert_{\infty}|x-y|^{2}\left(1+\frac{3}{2s}|y-x|\right)^{1/2}\right)^{2}\rho(x)\eta_{\varepsilon}(x,y)dxdy\right)^{1/2}\\
 & =(I_{\varepsilon}\times4\cdot III_{\varepsilon})^{1/2}.
\end{align*}
So, plugging in equations \ref{eq:1epsilon upper bound} and \ref{eq:3epsilon upper bound},
we see that
\begin{align*}
II_{\varepsilon} & \leq\left(\left(\left(\frac{d}{\varepsilon^{2}M_{2}(\eta)}+\frac{3}{2}\frac{d^{2}M_{3}(\eta)}{\varepsilon sM_{2}(\eta)^{2}}\right)\Vert\nabla\phi_{t}\Vert_{\infty}^{2}\right)\left(\frac{d^{2}}{M_{2}(\eta)^{2}}\frac{1}{s^{2}}\Vert\nabla\phi_{t}\Vert_{\infty}^{2}\left(M_{4}(\eta)+\frac{3}{2}M_{5}(\eta)\frac{\varepsilon}{s}\right)\right)\right)^{1/2}\\
 & =\frac{d^{2}}{\varepsilon sM_{2}(\eta)^{2}}\sqrt{\left(\frac{M_{2}(\eta)}{d}+\frac{3}{2}M_{3}(\eta)\frac{\varepsilon}{s}\right)\left(M_{4}(\eta)+\frac{3}{2}M_{5}(\eta)\frac{\varepsilon}{s}\right)}\Vert\nabla\phi_{t}\Vert_{\infty}^{2}.
\end{align*}

In sum: we find that it suffices to pick $C$ in such a way that 
\begin{multline*}
-\frac{CA^{2}}{\varepsilon s}+\frac{3}{8}\frac{d^{2}M_{3}(\eta)}{\varepsilon sM_{2}(\eta)^{2}}\Vert\nabla\phi_{t}\Vert_{\infty}^{2}+\frac{d^{2}}{\varepsilon sM_{2}(\eta)^{2}}\sqrt{\left(\frac{M_{2}(\eta)}{d}+\frac{3}{2}M_{3}(\eta)\frac{\varepsilon}{s}\right)\left(M_{4}(\eta)+\frac{3}{2}M_{5}(\eta)\frac{\varepsilon}{s}\right)}\Vert\nabla\phi_{t}\Vert_{\infty}^{2}\\
+\frac{1}{4}\frac{d^{2}}{M_{2}(\eta)^{2}}\frac{1}{s^{2}}\left(M_{4}(\eta)+\frac{3}{2}M_{5}(\eta)\frac{\varepsilon}{s}\right)\Vert\nabla\phi_{t}\Vert_{\infty}^{2}\leq0.
\end{multline*}
Using the fact that $s\geq\varepsilon$ and $\varepsilon\in(0,1]$,
we see that it suffices to pick
\[
C=\frac{d^{2}}{M_{2}(\eta)^{2}}\left[\frac{3}{8}M_{3}(\eta)+\sqrt{\left(\frac{M_{2}(\eta)}{d}+\frac{3}{2}M_{3}(\eta)\right)\left(M_{4}(\eta)+\frac{3}{2}M_{5}(\eta)\right)}+\frac{1}{4}\left(M_{4}(\eta)+\frac{3}{2}M_{5}(\eta)\right)\right]
\]
and 
$
A^{2}\geq\sup_{t}\Vert\nabla\phi_{t}\Vert_{\infty}^{2}.
$
\end{proof}

\begin{cor}
\label{lower-bound} Suppose that $\eta$ and $\theta$ satisfy
Assumptions \ref{assu:eta properties} and \ref{assu:theta properties}
respectively.
Let $\varepsilon\in(0,1]$. Let $\mu_{0},\mu_{1}\in\mathcal{P}(\mathbb{R}^{d})$
and assume both $\mu_{0}$ and $\mu_{1}$ are supported inside some
set with diameter at most $R$. Then,
we have that
\[
W_{2}^{2}(\mu_{0},\mu_{1})\leq\varepsilon^{2}\frac{M_{2}(\eta)}{2d}\mathcal{W}_{\eta,\varepsilon}^{2}(\mu_{0},\mu_{1})+\left(\frac{7}{4}dR^{2}+8dR\right)\sqrt{\varepsilon}.
\]
\end{cor}

\begin{proof}
Suppose by Corollary \ref{cor:HJ Lipschitz diameter bound} that $(\phi_{t})_{t\in[0,1]}$
is the optimal HJ subsolution for $(\mu_{0},\mu_{1})$, that is, $(\phi_{t})_{t\in[0,1]}$
satisfies
\[
\frac{1}{2}W_{2}^{2}(\mu_{0},\mu_{1})=\underset{\phi\in BL([0,1]\times\mathbb{R}^{d})}{\text{argmax}}\left\{ \int\phi_{1}d\mu_{1}-\int\phi_{0}d\mu_{0}:\partial_{t}\phi_{t}+\frac{1}{2}|\nabla\phi_{t}|^{2}=0\text{ in viscosity sense}\right\} .
\]
Note that $(\phi_{t})_{t\in[0,1]}$ is also a (not necessarily optimal)
Hamilton-Jacobi solution for $(\boldsymbol{K_{2s}}*\mu_{0},\boldsymbol{K_{2s}}*\mu_{1})$,
and so
\[
\frac{1}{2}W_{2}^{2}(\mu_{0},\mu_{1})\geq\frac{1}{2}W_{2}^{2}(\boldsymbol{K_{2s}}*\mu_{0},\boldsymbol{K_{2s}}*\mu_{1})\geq\int\phi_{1}d(\boldsymbol{K_{2s}}*\mu_{1})-\int\phi_{0}d(\boldsymbol{K_{2s}}*\mu_{0}).
\]
Furthermore, 
\begin{align*}
\int\phi_{1}d(\boldsymbol{K_{2s}}*\mu_{1})-\int\phi_{0}d(\boldsymbol{K_{2s}}*\mu_{0}) & =\int\phi_{1}^{s}d(\boldsymbol{K_{s}}*\mu_{1})-\int\phi_{0}^{s}d(\boldsymbol{K_{s}}*\mu_{0}).\\
 & =\int\phi_{1}^{2s}d\mu_{1}-\int\phi_{0}^{2s}d\mu_{0}
\end{align*}
At the same time, $\phi_{1}^{s}\rightarrow\phi_{1}$ and $\phi_{0}^{s}\rightarrow\phi_{0}$
uniformly on compact sets as $s\rightarrow0$. Quantitatively, by
Corollary \ref{cor:HJ Lipschitz diameter bound} one has the estimate
\[
\mid\phi-K_{s}*\phi|\leq s\cdot\text{Lip}\phi\leq sRM_{1}(K)
\]
where $M_{1}(K):=\int_{\mathbb{R}^{d}}c_{K}|x|e^{-|x|}dx$. Hence
\[
\int|\phi_{i}-\phi_{i}^{2s}|d\mu_{i}\leq2sRM_{1}(K)\qquad i=0,1
\]
and therefore 
\[
\int\phi_{1}^{2s}d\mu_{1}-\int\phi_{0}^{2s}d\mu_{0}\geq\int\phi_{1}d\mu_{1}-\int\phi_{0}d\mu_{0}-4sRM_{1}(K).
\]
Together, this implies that 
\[
\frac{1}{2}W_{2}^{2}(\mu_{0},\mu_{1})\geq\int\phi_{1}^{s}d(\boldsymbol{K_{s}}*\mu_{1})-\int\phi_{0}^{s}d(\boldsymbol{K_{s}}*\mu_{0})\geq\frac{1}{2}W_{2}^{2}(\mu_{0},\mu_{1})-4sRM_{1}(K).
\]

Since $\phi_{t}$ is a viscosity solution of the Hamilton-Jacobi equation
and is Lipschitz, it is also a Lebesgue $(x,t)$-almost everywhere
solution, by Rademacher's theorem and \cite[Theorem 10.1.1]{evans2022partial}.
Therefore, we can apply Proposition \ref{prop:local-to-nonlocal subsolution},
and deduce (defining $\tilde{\phi}_{t}^{s,\varepsilon}:=\frac{2d}{\varepsilon^{2}M_{2}(\eta)}\phi_{t}^{s}$,
as in Proposition \ref{prop:local-to-nonlocal subsolution}) that
\[
\frac{d}{\varepsilon^{2}M_{2}(\eta)}W_{2}^{2}(\mu_{0},\mu_{1})\geq\int\tilde{\phi}_{1}^{s,\varepsilon}d(\boldsymbol{K_{s}}*\mu_{1})-\int\tilde{\phi}_{0}^{s,\varepsilon}d(\boldsymbol{K_{s}}*\mu_{0})\geq\frac{d}{\varepsilon^{2}M_{2}(\eta)}\left(W_{2}^{2}(\mu_{0},\mu_{1})-8sRM_{1}(K)\right)
\]
and likewise (defining $\check{\phi}_{t}^{s,\varepsilon}:=\tilde{\phi}_{t}^{s,\varepsilon}-\frac{CA^{2}}{\varepsilon s}t$,
as in Proposition \ref{prop:local-to-nonlocal subsolution}) 
\begin{align*}
\int\check{\phi}_{1}^{s,\varepsilon}d(\boldsymbol{K_{s}}*\mu_{1})-\int\check{\phi}_{0}^{s,\varepsilon}d(\boldsymbol{K_{s}}*\mu_{0}) & =\int\left(\tilde{\phi}_{t}^{s,\varepsilon}-\frac{CA^{2}}{\varepsilon s}\right)d(\boldsymbol{K_{s}}*\mu_{1})-\int\tilde{\phi}_{t}^{s,\varepsilon}d(\boldsymbol{K_{s}}*\mu_{0})\\
 & \geq\frac{d}{\varepsilon^{2}M_{2}(\eta)}\left(W_{2}^{2}(\mu_{0},\mu_{1})-8sRM_{1}(K)\right)-\frac{CA^{2}}{\varepsilon s}.
\end{align*}
Observe that $\check{\phi}_{t}^{s,\varepsilon}\in BL([0,1]\times\mathbb{R}^{d})$.
At the same time, we know, by Lemma \ref{cor:HJ Lipschitz diameter bound},
that $\sup_{t}\Vert\nabla\phi_{t}\Vert_{\infty}^{2}\leq R^{2}$; so
if we put $A=R$ and $C$ as in Proposition \ref{prop:local-to-nonlocal subsolution},
then by Proposition \ref{prop:local-to-nonlocal subsolution}, $\check{\phi}_{t}^{s}$
is an $s$-smooth nonlocal HJ subsolution. Therefore, we have that
\begin{align*}
\int\check{\phi}_{1}^{s,\varepsilon}d(\boldsymbol{K_{s}}*\mu_{1})-\int\check{\phi}_{0}^{s,\varepsilon}d(\boldsymbol{K_{s}}*\mu_{0}) & \leq \! \sup_{\phi_{t}(x)\in\text{HJ}_{\text{NL}}^{1,s}\cap BL([0,1]\times\mathbb{R}^{d})} \int\phi_{1}d(\boldsymbol{K_{s}}*\mu_{1})-\int\phi_{0}d(\boldsymbol{K_{s}}*\mu_{0})\\
\text{(Proposition \ref{prop:s-smooth duality inequality})} & \leq\frac{1}{2}\mathcal{W}_{\eta,\varepsilon,s}^{2}(\boldsymbol{K_{s}}*\mu_{0},\boldsymbol{K_{s}}*\mu_{1}).
\end{align*}
This implies that 
\[
\frac{d}{\varepsilon^{2}M_{2}(\eta)}\left(W_{2}^{2}(\mu_{0},\mu_{1})-8sRM_{1}(K)\right)-\frac{CR^{2}}{\varepsilon s}\leq\frac{1}{2}\mathcal{W}_{\eta,\varepsilon,s}^{2}(\boldsymbol{K_{s}}*\mu_{0},\boldsymbol{K_{s}}*\mu_{1})
\]
and thus 
\[
W_{2}^{2}(\mu_{0},\mu_{1})-8sRM_{1}(K) \leq\varepsilon^{2}\frac{M_{2}(\eta)}{2d}\left(\mathcal{W}_{\eta,\varepsilon,s}^{2}(\boldsymbol{K_{s}}*\mu_{0},\boldsymbol{K_{s}}*\mu_{1})+\frac{CR^{2}}{\varepsilon s}\right).
\]
Now, since $\eta$ is supported on $B(0,1)$, all of the higher moments
$M_{3}(\eta)$, $M_{4}(\eta)$, and $M_{5}(\eta)$ are bounded above
by $M_{2}(\eta)$, we find that 
\[
C\leq\frac{d^{2}}{M_{2}(\eta)}\left[\frac{3}{8}+\sqrt{\left(\frac{1}{d}+\frac{3}{2}\right)\frac{5}{2}}+\frac{5}{16}\right]
\]
so that
\[
W_{2}^{2}(\mu_{0},\mu_{1})\leq  \varepsilon^{2}\frac{M_{2}(\eta)}{2d}\mathcal{W}_{\eta,\varepsilon,s}^{2}(\boldsymbol{K_{s}}*\mu_{0},\boldsymbol{K_{s}}*\mu_{1})
  +\left[\frac{3}{8}+\sqrt{\left(\frac{1}{d}+\frac{3}{2}\right)\frac{5}{2}}+\frac{5}{16}\right]R^{2}\frac{d\varepsilon}{2s}+8sRM_{1}(K).
\]
 Since $\mathcal{W}_{\eta,\varepsilon,s}^{2}(\boldsymbol{K_{s}}*\mu_{0},\boldsymbol{K_{s}}*\mu_{1})\leq\mathcal{W}_{\eta,\varepsilon}^{2}(\mu_{0},\mu_{1})$
by Lemma \ref{lem:s-smoothed nonlocal comparison}, we deduce that
\begin{align*}
W_{2}^{2}(\mu_{0},\mu_{1})\leq & \varepsilon^{2}\frac{M_{2}(\eta)}{2d}\mathcal{W}_{\eta,\varepsilon}^{2}(\mu_{0},\mu_{1})+\frac{7}{4}dR^{2}\frac{\varepsilon}{s}++8sRM_{1}(K).
\end{align*}
Finally, we set $s=\sqrt{\varepsilon}$, and use the fact that $M_1(K)=d$, by Lemma \ref{lem:Laplace kernel moments}.
\end{proof}
\medskip

\textbf{Acknowledgements:} The authors thank Giuseppe Buttazzo, Matthias Erbar, Wilfrid
Gangbo, Nicol\'as Garc\'ia Trillos, Giovanni Leoni, and Jan Maas for helpful discussions.
The authors are grateful to NSF for support via grants DMS 1814991 and DMS 2206069. 
 Part of this work was done while the authors were visiting the Simons Institute for the Theory of Computing. The authors are thank the institute for hospitality. Furthermore they are thankful to CNA of CMU for support.

\bibliographystyle{siam}
\bibliography{otrefs} \label{refer}

\appendix

\section{Lower semicontinuity of integral functionals on $\mathcal{M}_{loc}$}
\label{sec:lsc}

In this appendix, we establish the following variant of Reshetnyak's
theorem as well as some direct consequences of this theorem which
are used in the main body of the article (for other variants on this
theorem, we refer the reader to \cite{ambrosio1987new,ambrosio2000functions,buttazzo1989semicontinuity,reshetnyak1967general,spector2011simple}).

\begin{thm}
\label{thm:locally finite reshetnyak}(locally finite, topological
and sequential ``Reshetnyak's theorem'') Let $\Omega$ be a locally
compact Polish space, and let $f:\Omega\times\mathbb{R}^{n}\rightarrow[0,\infty]$
be a (topologically) lower semicontinuous function such that for every
$\omega\in\Omega$, the function $f(\omega,\cdot)$ is convex and
positively 1-homogeneous. Then the functional 
\[
F:\mathcal{M}_{loc}(\Omega,\mathbb{R}^{n})\rightarrow[0,\infty]
\]
\[
F(\lambda)=\int_{\Omega}f\left(\omega,\frac{d\lambda}{d|\lambda|}(\omega)\right)d|\lambda|(\omega)
\]
is convex, and both topologically and sequentially weak{*} lower semicontinuous.
\end{thm}

In the statement of this theorem, the ``weak{*} topology on $\mathcal{M}_{loc}(\Omega,\mathbb{R}^{d})$''
has the following sense: we consider $\mathcal{M}_{loc}(\Omega,\mathbb{R}^{n})$
as the dual space of $C_{c}(\Omega,\mathbb{R}^{n})$, where $C_{c}(\Omega,\mathbb{R}^{n})$
is understood as a locally convex t.v.s. which is a direct limit of
the Banach spaces $C_{c}(K,\mathbb{R}^{n})$ for every compact $K\subset\Omega$.
(See also discussion on this point in discussion in \cite[Chapter 1]{ambrosio2000functions}.)
The reason why the statement of the theorem carefully specifies that
$F$ is \emph{both} topologicallly and sequentially weak{*} l.s.c.
is that in this setting, it is not known to the authors whether the
two notions coincide. (In particular, $\mathcal{M}_{loc}(\Omega,\mathbb{R}^{d})$
is not the weak{*} dual of a separable Banach space, so we cannot
apply \cite[Proposition 1.1.6 (iii)]{buttazzo1989semicontinuity}.)
For our immediate purposes, the sequential weak{*} l.s.c. property
is used in variational arguments, while $F$ being topologically weak{*}
l.s.c. implies directly that $F$ is weak{*} Borel measurable.

\begin{proof}
Convexity of $F$ follows directly, regardless of the underlying domain,
from the convexity and $1$-homogeneity of $f$.

Let $K\subset\Omega$ be compact. Consider the functional
\[
\mathcal{M}_{loc}(\Omega,\mathbb{R}^{n})\rightarrow[0,\infty]
\]
\[
\lambda\mapsto\int_{K}f\left(\omega,\frac{d\pi_{K}\lambda}{d|\pi_{K}\lambda|}(\omega)\right)d|\pi_{K}\lambda|(\omega)
\]
namely precomposition of the continuous projection $\pi_{K}$ with
the functional $F(\cdot\upharpoonright K)$ (which has domain $\mathcal{M}(K,\mathbb{R}^{n})$).
The functional $F(\cdot\upharpoonright K)$ is known, by Reshetnyak's
theorem (more precisely, \cite[Theorem 3.4.3]{buttazzo1989semicontinuity}),
to be sequentially weak{*} lower semicontinuous on $\mathcal{M}(K,\mathbb{R}^{n})$.
Since $\mathcal{M}(K,\mathbb{R}^{n})$ is the dual of a separable
Banach space, and $F(\cdot\upharpoonright K)$ is convex, $F(\cdot\upharpoonright K)$
is also topologically weak{*} lower semicontinuous on $\mathcal{M}(K,\mathbb{R}^{n})$.
Since the precomposition of a continuous map with a (topologically)
l.s.c. map is again l.s.c., it follows that $F(\cdot\upharpoonright K)\circ\pi_{K}$
is a topologically lower semicontinuous map from $\mathcal{M}_{loc}(\Omega,,\mathbb{R}^{n})$
to $[0,\infty]$. 

Now, consider a sequence $(K_{n})$ of compact sets which exhaust
$\Omega$. Observe that 
\begin{align*}
\int_{\Omega}f\left(\omega,\frac{d\lambda}{d|\lambda|}(\omega)\right)d|\lambda|(\omega) & =\sup_{n}\int_{K_{n}}f\left(\omega,\frac{d\lambda}{d|\lambda|}(\omega)\right)d|\lambda|(\omega) \\
& =\sup_{n}\int_{K_{n}}f\left(\omega,\frac{d\pi_{K_{n}}\lambda}{d|\pi_{K_{n}}\lambda|}(\omega)\right)d|\pi_{K_{n}}\lambda|(\omega).
\end{align*}
Consequently, $\int_{\Omega}f\left(\omega,\frac{d\lambda}{d|\lambda|}(\omega)\right)d|\lambda|(\omega)$
is a supremum of topologically l.s.c. functionals, and therefore is
topologically l.s.c. on $\mathcal{M}_{loc}(\Omega,\mathbb{R}^{n})$
as well \cite[Proposition 1.1.2 (ii)]{buttazzo1989semicontinuity}. 

In fact, from this one can also deduce that $\int_{\Omega}f\left(\omega,\frac{d\lambda}{d|\lambda|}(\omega)\right)d|\lambda|(\omega)$
is also sequentially weak{*} l.s.c.: this follows from the fact that
a function is sequentially l.s.c. with respect to some topology $\tau$
iff it is l.s.c. with respect to the sequential topology induced by
$\tau$, and this topology is at least as fine as $\tau$ itself (see
discussion in \cite{buttazzo1989semicontinuity}, especially \cite[Proposition 1.1.5]{buttazzo1989semicontinuity}).
\end{proof}
\begin{cor}[{compare \cite[Lemma 2.4]{erbar2014gradient} and \cite[Lemma 2.9]{esposito2019nonlocal}}]
\label{cor:action convex l.s.c.} The action $\mathcal{A}(\mu,\mathbf{j};m)$
is jointly convex, and jointly topologically and sequentially weak{*}
l.s.c., on $\mathcal{P}(\mathbb{R}^{d})\times\mathcal{M}_{loc}(G)\times\mathcal{M}_{loc}^{+}(\mathbb{R}^{d})$,
also more generally on the larger space $\mathcal{M}_{loc}^{+}(\mathbb{R}^{d})\times\mathcal{M}_{loc}(G)\times\mathcal{M}_{loc}^{+}(\mathbb{R}^{d})$. 
\end{cor}

\begin{proof}
Joint convexity follows exactly as in \cite[Lemma 2.7]{erbar2014gradient}
and \cite[Lemma 2.12]{esposito2019nonlocal}.

We apply Theorem \ref{thm:locally finite reshetnyak} in the case
where the space of vector measures is $\mathcal{M}_{loc}(\mathbb{R}^{2d};\mathbb{R}^{3})$
and 
\[
f\left((x,y),\frac{d\lambda}{d|\lambda|}(x,y)\right)=\frac{\left(\frac{d\lambda_{3}}{d|\lambda|}\right)^{2}}{2\theta\left(\frac{d\lambda_{1}}{d|\lambda|},\frac{d\lambda_{2}}{d|\lambda|}\right)}\eta(x,y)\mathbf{1}_{G}(x,y)
\]
with the convention that $\frac{0}{0}=0$, and that $\theta(r,s)=0$
if either $r$ or $s$ is negative. Precomposing this $\int_{\mathbb{R}^{2d}}f\left((x,y),\frac{d\lambda}{d|\lambda|}(x,y)\right)d|\lambda|(x,y)$
with the continuous embedding 
\[
\mathcal{M}_{loc}^{+}(\mathbb{R}^{d})\times\mathcal{M}_{loc}(G)\times\mathcal{M}_{loc}^{+}(\mathbb{R}^{d})\hookrightarrow\mathcal{M}_{loc}(\mathbb{R}^{2d};\mathbb{R}^{3})
\]
\[
(\mu,\mathbf{j},m)\mapsto(\mu\times m,m\times\mu,\mathbf{j})
\]
demonstrates that $\mathcal{A}(\mu,\mathbf{j};m)$ is topologically
l.s.c. on $\mathcal{M}_{loc}^{+}(\mathbb{R}^{d})\times\mathcal{M}_{loc}(G)\times\mathcal{M}_{loc}^{+}(\mathbb{R}^{d})$
as desired, and sequential lower semicontinuity follows exactly as
before.
\end{proof}

As an application, we give a proof of the following convolution inequality.
\begin{lemma}
\label{lem:mass-flux convolution inequality}Let $\mu\in\mathcal{P}(\mathbb{R}^{d})$
and $\mathbf{j}\in\mathcal{M}_{loc}(G)$, and suppose that $\mathcal{A}_{\eta,\theta}(\mu,\mathbf{j})<\infty$.
Let $k$ be a convolution kernel, and for each $z\in\mathbb{R}^{d}$,
let $\mu_{z}$ and $\mathbf{j}_{z}$ denote the $z$-translates of
$\mu$ and $\mathbf{j}$ respectively: namely for all $\psi\in C_{C}^{\infty}(\mathbb{R}^{d})$
and $\varphi\in C_{C}^{\infty}(G)$,
\[
\int_{\mathbb{R}^{d}}\psi(x)d\mu_{z}(x):=\int_{\mathbb{R}^{d}}\psi(x+z)d\mu(x);\quad\iint_{G}\varphi(x,y)d\mathbf{j}_{z}(x,y)=\iint_{G}\varphi(x+z,y+z)d\mathbf{j}(x,y).
\]
Then, 
\[
\mathcal{A}_{\eta,\theta}(\boldsymbol{k}*\mu,\boldsymbol{k}*\mathbf{j})\leq\int_{\mathbb{R}^{d}}\mathcal{A}_{\eta,\theta}(\mu_{z},\mathbf{j}_{z})k(z)dz.
\]
\end{lemma}

We note that morally, this is just an application of Jensen's inequality,
but we are unaware of a version of Jensen's inequality in the literature
that applies to this setting. Therefore we give an ad hoc proof, essentially
the same (albeit with addition complications) to the proof of Proposition
\ref{thm:disintegration-inequality}.
\begin{proof}
Let $(\Omega,\mathcal{A},\mathbb{P})$ be a probability space, and
let $Z:\Omega\rightarrow\mathbb{R}^{d}$ be a random variable with
distribution $k(|z|)dz$. Furthermore, let $Z_{1},Z_{2},\ldots$ be
i.i.d. copies of $Z$. Then, $(\mu_{Z},\mathbf{j}_{Z})$ is a $\mathcal{P}(\mathbb{R}^{d})\times\mathcal{M}_{loc}(G)$-valued
random variable, and since $\mathcal{A}_{\eta,\theta}$ is topologically
l.s.c. and hence Borel, it follows that $\mathcal{A}_{\eta,\theta}(\mu_{Z},\mathbf{j}_{Z})$
is a $[0,\infty]$-valued random variable with distribution $\mathcal{A}_{\eta,\theta}(\mu_{z},\mathbf{j}_{z})k(|z|)dz$. 

Since $\mathcal{A}_{\eta,\theta}$ is jointly convex, we observe that
\[
\mathcal{A}_{\eta,\theta}\left(\frac{1}{n}\sum_{i=1}^{n}\mu_{Z_{i}},\frac{1}{n}\sum_{i=1}^{n}\mathbf{j}_{Z_{i}}\right)\leq\frac{1}{n}\sum_{i=1}^{n}\mathcal{A}_{\eta,\theta}(\mu_{Z_{i}},\mathbf{j}_{Z_{i}}).
\]
Now, suppose that $\int_{\mathbb{R}^{d}}\mathcal{A}_{\eta,\theta}(\mu_{z},\mathbf{j}_{z})k(|z|)dz<\infty$,
since otherwise the lemma holds trivially. Since the action is nonnegative,
the strong law of large numbers shows that 
\[
\sum_{i=1}^{n}\mathcal{A}_{\eta,\theta}(\mu_{Z_{i}},\mathbf{j}_{Z_{i}})\rightarrow\int_{\mathbb{R}^{d}}\mathcal{A}_{\eta,\theta}(\mu_{z},\mathbf{j}_{z})k(|z|)dz
\]
almost surely. Therefore, if we can verify that $\frac{1}{n}\sum_{i=1}^{n}\mu_{Z_{i}}\stackrel{*}{\rightharpoonup}\boldsymbol{k}*\mu$
and $\frac{1}{n}\sum_{i=1}^{n}\mathbf{j}_{Z_{i}}\stackrel{*}{\rightharpoonup}\boldsymbol{k}*\mathbf{j}$
almost surely, it then follows from the joint sequential lower semicontinuity
of $\mathcal{A}_{\eta,\theta}$ that 
\[
\mathcal{A}_{\eta,\theta}(\boldsymbol{k}*\mu,\boldsymbol{k}*\mathbf{j})\leq\liminf_{n\rightarrow\infty}\mathcal{A}_{\eta,\theta}\left(\frac{1}{n}\sum_{i=1}^{n}\mu_{Z_{i}},\frac{1}{n}\sum_{i=1}^{n}\mathbf{j}_{Z_{i}}\right)
\]
and the lemma is proved.  

We therefore verify that $\frac{1}{n}\sum_{i=1}^{n}\mu_{Z_{i}}\stackrel{*}{\rightharpoonup}\boldsymbol{k}*\mu$
and $\frac{1}{n}\sum_{i=1}^{n}\mathbf{j}_{Z_{i}}\stackrel{*}{\rightharpoonup}\boldsymbol{k}*\mathbf{j}$
almost surely; note that this is a similar result to the Glivenko-Cantelli
theorem. 

Let $\psi\in C_{0}(\mathbb{R}^{d})$. Then $\int_{\mathbb{R}^{d}}\psi(x)d\mu_{Z}(x)$
is an $L^{1}$ random variable, in fact $L^{\infty}$ since for all
$\omega\in\Omega$
\[
\left|\int_{\mathbb{R}^{d}}\psi(x)d\mu_{Z(\omega)}(x)\right|\leq\Vert\psi\Vert_{\infty}.
\]
It follows from the strong law of large numbers that 
\begin{align*}
\int_{\mathbb{R}^{d}}\psi(x)d\left(\frac{1}{n}\sum_{i=1}^{n}\mu_{Z_{i}}(x)\right) & =\frac{1}{n}\sum_{i=1}^{n}\int_{\mathbb{R}^{d}}\psi(x)d\mu_{Z_{i}}(x)\\
\text{(a.s.)} & \rightarrow\int_{\mathbb{R}^{d}}\left(\int_{\mathbb{R}^{d}}\psi(x)d\mu_{z}(x)\right)k(|z|)dz\\
 & =\int_{\mathbb{R}^{d}}\psi(x)d\left(\int_{\mathbb{R}^{d}}\mu_{z}k(|z|)dz\right)(x)\\
 & =\int_{\mathbb{R}^{d}}\psi(x)d(\boldsymbol{k}*\mu)(x).
\end{align*}
By \cite[Lemma 4.8 (i)]{kallenberg2017random} this implies that $\frac{1}{n}\sum_{i=1}^{n}\mu_{Z_{i}}\stackrel{*}{\rightharpoonup}\boldsymbol{k}*\mu$
almost surely (more specifically, in duality with $C_{0}(\mathbb{R}^{d})$).
We will use similar, albeit more involved, reasoning to show that
$\frac{1}{n}\sum_{i=1}^{n}\mathbf{j}_{Z_{i}}\stackrel{*}{\rightharpoonup}\boldsymbol{k}*\mathbf{j}$
almost surely as well. 

Consider the Jordan decomposition $\mathbf{j}=\mathbf{j}^{+}-\mathbf{j}^{-}$.
By %
{} globalization of \cite[Theorem 2.8]{dubins1964measurable} this is
a measurable operation, and so both $\mathbf{j}_{Z}^{+}$ and $\mathbf{j}_{Z}^{-}$
are $\mathcal{M}_{loc}^{+}(G)$-valued random variables. Note also
that 
\[
\boldsymbol{k}*\mathbf{j}=\int_{\mathbb{R}^{d}}\mathbf{j}_{z}k(|z|)dz=\int_{\mathbb{R}^{d}}\left(\mathbf{j}_{z}^{+}-\mathbf{j}_{z}^{-}\right)k(|z|)dz=\boldsymbol{k}*\mathbf{j}^{+}-\boldsymbol{k}*\mathbf{j}^{-}.
\]
Let $K\subset G$ be any compact set, and let $\varphi\in C_{c}(K)$.
Then, we claim that both of 
\[
\iint_{K}\varphi(x,y)\eta(|x-y|)d\mathbf{j}_{z}^{+}(x,y)\text{ and }\iint_{K}\varphi(x,y)\eta(|x-y|)d\mathbf{j}_{z}^{-}(x,y)
\]
are $L^{1}$ random variables, in fact $L^{\infty}$. It suffices
to show that $|\eta(|x-y|)d\mathbf{j}_{z}(x,y)|(K)$ is uniformly
bounded.

In \cite[Lemma 2.6]{erbar2014gradient} it is shown that for some
uniform constant $C$, the measure $\mathbf{v}=\eta(|x-y|)d\mathbf{j}(x,y)$
satisfies the property 
\[
|\mathbf{v}|(K)\leq\frac{1}{a}\int_{G}(1\wedge|x-y|)d|\mathbf{v}|(x,y)\leq C\mathcal{A}(\mu,\mathbf{j})
\]
where $a=\min\{|x-y|:(x,y)\in K\}$. Denoting 
\[
K-\binom{z}{z}:=\{(x-z,y-z):(x,y)\in K\}
\]
we see that $a$ is invariant under translation of $K$; and $\eta(|x-y|)$
is invariant under translation of $(x,y)$ by $z$, so in fact for
\emph{all }$z\in\mathbb{R}^{d}$ uniformly, 
\[
|\eta(|x-y|)d\mathbf{j}_{z}(x,y)|(K)\leq C\mathcal{A}(\mu_{z},\mathbf{j}_{z}).
\]
Furthermore, as observed in \cite{erbar2014gradient}, $\mathcal{A}(\mu,\mathbf{j})$
is invariant under translation of $(\mu,\mathbf{j})$ since we have
taken the reference measure on $\mathbb{R}^{d}$ to be the Lebesgue
measure. Therefore, it holds (uniformly in $z$) that
\[
\left|\iint_{K}\varphi(x,y)\eta(|x-y|)d\mathbf{j}_{z}^{+}(x,y)\right|\leq\Vert\varphi\Vert_{L^{\infty}(K)}|\eta(|x-y|)d\mathbf{j}(x,y)|(K)
\]
and the same holds for $\mathbf{j}_{z}^{-}$ also. 

Now, it follows from the strong law of large numbers that 
\begin{align*}
\int_{K}\varphi(x,y)\eta(|x-y|)d\left(\frac{1}{n}\sum_{i=1}^{n}\mathbf{j}_{Z_{i}}^{+}(x,y)\right) & =\frac{1}{n}\sum_{i=1}^{n}\int_{K}\varphi(x,y)\eta(|x-y|)d\mathbf{j}_{Z_{i}}^{+}(x)\\
\text{(a.s.)} & \rightarrow\int_{\mathbb{R}^{d}}\left(\int_{K}\varphi(x,y)\eta(|x-y|)d\mathbf{j}_{z}^{+}(x,y)\right)k(|z|)dz\\
 & =\int_{K}\varphi(x,y)\eta(|x-y|)d\left(\int_{\mathbb{R}^{d}}\mathbf{j}_{z}^{+}k(|z|)dz\right)(x,y)\\
 & =\int_{K}\varphi(x,y)\eta(|x-y|)d(\boldsymbol{k}*\mathbf{j}^{+})(x,y).
\end{align*}
By \cite[Lemma 4.8 (i)]{kallenberg2017random} this implies that almost
surely, 
\[
\eta(|x-y|)d\left(\frac{1}{n}\sum_{i=1}^{n}\mathbf{j}_{Z_{i}}^{+}(x,y)\right)\upharpoonright K\stackrel{*}{\rightharpoonup}\eta(|x-y|)d(\boldsymbol{k}*\mathbf{j}^{+})(x,y)\upharpoonright K
\]
(where the convergence is in duality with $C_{C}(K)$); since $\mathbf{j}$
is supported on $G:=\{(x,y):\eta(|x-y|)>0\}$ (and therefore, so is
$\mathbf{j}_{z}$ for every $z$, since $G$ is translation-invariant)
and $\eta$ is a continuous function of compact support when restricted
to any $K\subset G$, this implies that that almost surely, 
\[
\left(\frac{1}{n}\sum_{i=1}^{n}\mathbf{j}_{Z_{i}}^{+}\right)\upharpoonright K\stackrel{*}{\rightharpoonup}\boldsymbol{k}*\mathbf{j}^{+}\upharpoonright K
\]
again in duality with $C_{C}(K)$. Since $K$ was arbitrary, by the
characterization of weak{*} convergence in $\mathcal{M}_{loc}(G)$
we have that almost surely, 
\[
\left(\frac{1}{n}\sum_{i=1}^{n}\mathbf{j}_{Z_{i}}^{+}\right)\stackrel{*}{\rightharpoonup}\boldsymbol{k}*\mathbf{j}^{+}
\]
in duality with $C_{c}(G)$. By identical reasoning, we have that
$\left(\frac{1}{n}\sum_{i=1}^{n}\mathbf{j}_{Z_{i}}^{-}\right)\stackrel{*}{\rightharpoonup}\boldsymbol{k}*\mathbf{j}^{-}$
also. Finally, this implies that $\left(\frac{1}{n}\sum_{i=1}^{n}\mathbf{j}_{Z_{i}}\right)\stackrel{*}{\rightharpoonup}\boldsymbol{k}*\mathbf{j}$
almost surely. 
\end{proof}

\section{Additional lemmas} \label{sec:additional}

\begin{lemma}
[estimate adapted from the two point $\mathcal{W}$ space]\label{lem:2 point space estimate}Let
$\varepsilon>0$. Let $A$ and $B$ be disjoint, bounded subsets of
$\mathbb{R}^{d}$ of positive Lebesgue measure, such that $\sup_{x \in A, y\in B} |x-y| <\varepsilon$.
Suppose without loss of generality that $|A|\leq|B|$. Let $\mathfrak{m}_{A}$
and $\mathfrak{m}_{B}$ denote the uniform probability measures supported
on $A$ and $B$, respectively. Then, 
\[
\mathcal{W}_{\eta,\varepsilon}(\mathfrak{m}_{A},\mathfrak{m}_{B})\leq\frac{C_{\theta}}{4\sqrt{|A|\eta_{\varepsilon}(\sup_{x \in A, y\in B} |x-y|)}}
\]
where $C_{\theta}:=\int_{0}^{1}\frac{1}{\sqrt{\theta(1-r,1+r)}}dr$.
\end{lemma}

\begin{proof}
We use a construction (as suggested by \cite{garciatrillos2020gromov})
from Lemma 2.3 (and Theorem 2.4) from \cite{maas2011gradient}. If
the 2-point $\mathcal{W}_{\theta}$ space is connected, then the proof
Lemma 2.3 therein indicates that for any $\varepsilon_{0}>0$, there
exist differentiable functions $\gamma:[0,1]\rightarrow[-1,1]$ and
$\chi:[0,1]\rightarrow\mathbb{R}$ such that $\gamma(0)=-1$, $\gamma(1)=1$,
and 
\[
\forall t\in(0,1)\quad\gamma_{t}^{\prime}=\frac{1}{4}\theta(1+\gamma_{t},1-\gamma_{t})\chi_{t},
\]
and, 
\[
\int_{0}^{1}\theta(1+\gamma_{t},1-\gamma_{t})\chi_{t}^{2}dt<\mathcal{W}_{\theta}^{2}(\delta_{0},\delta_{1})+\varepsilon_{0}.
\]
(Here $\mathcal{W}_{\theta}(\delta_{0},\delta_{1})$ denotes the $\mathcal{W}_{\theta}$
distance between two Dirac masses on the two-point graph, with edge
weight $1/2$.)

Consider the curve $\rho_{t}:[0,1]\rightarrow\mathcal{P}(\mathbb{R}^d)$ defined
by 
\[
\frac{d\rho_{t}}{d\text{Leb}}(x)=\begin{cases}
\frac{1-\gamma_{t}}{2|A|} & x\in A\\
\frac{1+\gamma_{t}}{2|B|} & x\in B\\
0 & \text{else}.
\end{cases}
\]
Note also that by construction, 
\[
\frac{d}{dt}\frac{d\rho_{t}}{d\text{Leb}}(x)=\begin{cases}
-\frac{\theta(1+\gamma_{t},1-\gamma_{t})\chi_{t}}{8|A|} & x\in A\\
\frac{\theta(1+\gamma_{t},1-\gamma_{t})\chi_{t}}{8|B|} & x\in B\\
0 & \text{else}.
\end{cases}
\]
 Let $\mathbf{j}_{t}(x,y)$ be a flux so that $(\rho_{t},\mathbf{j}_{t})$
solves the nonlocal continuity equation; in particular we set
\[
\frac{d\mathbf{j}_{t}}{d(\text{Leb}\otimes\text{Leb})}(x,y)=\begin{cases}
\frac{\theta(1+\gamma_{t},1-\gamma_{t})\chi_{t}}{16\eta_{\varepsilon}(x,y)|A||B|} & (x,y)\in A\times B\\
-\frac{\theta(1+\gamma_{t},1-\gamma_{t})\chi_{t}}{16\eta_{\varepsilon}(x,y)|A||B|} & (x,y)\in B\times A\\
0 & \text{else}.
\end{cases}
\]
Now, observe that $\text{Leb}\otimes\text{Leb}$ dominates all of
$\mathbf{j}_{t}$, $\rho_{t}\otimes\text{Leb}$, and $\text{Leb}\otimes\rho_{t}$.
Furthermore, note that that $\frac{d(\rho_{t}\otimes\text{Leb})}{d(\text{Leb}\otimes\text{Leb})}(x,y)=\frac{d\rho_{t}}{d\text{Leb}}(x)$,
and similarly $\frac{d(\text{Leb}\otimes\rho_{t})}{d(\text{Leb}\otimes\text{Leb})}(x,y)=\frac{d\rho_{t}}{d\text{Leb}}(y)$.
Together with the fact that $\mathbf{j}_{t}$ is antisymmetric, it
follows that
\begin{align*}
\mathcal{A}(\rho_{t},\mathbf{j}_{t}) & =\int_{A}\int_{B}\frac{\left(\frac{\theta(1+\gamma_{t},1-\gamma_{t})\chi_{t}}{16\eta_{\varepsilon}(x,y)|A||B|}\right)^{2}}{\theta\left(\frac{1-\gamma_{t}}{2|A|}(x),\frac{1+\gamma_{t}}{2|B|}(y)\right)}\eta_{\varepsilon}(x,y)dxdy.
\end{align*}
Moreover, from the homogeneity and monotonicity of $\theta$, and
the fact that $|A|\leq|B|$,
\begin{align*}
\mathcal{A}(\rho_{t},\mathbf{j}_{t}) & =\int_{A}\int_{B}\frac{\left(\frac{\theta(1+\gamma_{t},1-\gamma_{t})\chi_{t}}{16\eta_{\varepsilon}(x,y)|A||B|}\right)^{2}}{\frac{1}{2|B|}\cdot\theta\left(\frac{1-\gamma_{t}}{\frac{|A|}{|B|}},1+\gamma_{t}\right)}\eta_{\varepsilon}(x,y)dxdy\\
 & \leq\int_{A}\int_{B}\frac{\left(\frac{\theta(1+\gamma_{t},1-\gamma_{t})\chi_{t}}{16\eta_{\varepsilon}(x,y)|A||B|}\right)^{2}}{\frac{1}{2|B|}\cdot\theta\left(1-\gamma_{t},1+\gamma_{t}\right)}\eta_{\varepsilon}(x,y)dxdy
\end{align*}
and simplifying, we see that 
\begin{align*}
\mathcal{A}(\rho_{t},\mathbf{j}_{t}) & \leq\int_{A}\int_{B}\frac{\theta(1+\gamma_{t},1-\gamma_{t})\chi_{t}^{2}}{2^{7}\eta_{\varepsilon}(x,y)|A|^{2}|B|}dxdy.
\end{align*}
On the other hand, the action computation from the 2-point space in
\cite{maas2011gradient} tells us that
\[
\frac{1}{4}\int_{0}^{1}\theta(1+\gamma_{t},1-\gamma_{t})\chi_{t}^{2}dt<\mathcal{W}_{\theta}^{2}(\delta_{0},\delta_{1})+\varepsilon_{0}.
\]
Thus, 
\begin{align*}
\int_{0}^{1}\mathcal{A}(\rho_{t},\mathbf{j}_{t})dt & \leq\frac{1}{2^{7}|A|^{2}|B|\eta_{\varepsilon}(\sup_{x\in A,y\in B}|x-y|)}\int_{0}^{1}\int_{A}\int_{B}\theta(1+\gamma_{t},1-\gamma_{t})\chi_{t}^{2}dxdydt\\
 & =\frac{1}{2^{7}|A|\eta_{\varepsilon}(\sup_{x\in A,y\in B}|x-y|)}\int_{0}^{1}\theta(1+\gamma_{t},1-\gamma_{t})\chi_{t}^{2}dt\\
 & <\frac{\mathcal{W}^{2}(\delta_{0},\delta_{1})+\varepsilon_{0}}{2^{5}|A|\eta_{\varepsilon}(\sup_{x\in A,y\in B}|x-y|)}.
\end{align*}
But $\varepsilon_{0}$ is arbitrary, so 
\[
\mathcal{W}_{\eta,\varepsilon}(\mathfrak{m}_{A},\mathfrak{m}_{B})=\sqrt{\int_{0}^{1}\mathcal{A}(\rho_{t},j_{t})dt}\leq\frac{\mathcal{W}(\delta_{0},\delta_{1})}{2^{2.5}\sqrt{|A|\eta_{\varepsilon}(\sup_{x\in A,y\in B}|x-y|)}}.
\]
We can rephrase this in terms of the constant $C_{\theta}$, which
is defined by $C_{\theta}:=\int_{0}^{1}\frac{1}{\sqrt{\theta(1-r,1+r)}}dr$
or, equivalently (see \cite{maas2011gradient}), $C_{\theta}$ is
$\sqrt{2}$ times the $\mathcal{W}_{\theta}$ distance between a Dirac
mass and the uniform distribution on the two-point space, in other
words, $C_{\theta}=\frac{\sqrt{2}}{2}\mathcal{W}(\delta_{0},\delta_{1})$;
so that
\[
\mathcal{W}_{\eta,\varepsilon}(\mathfrak{m}_{A},\mathfrak{m}_{B})\leq\frac{C_{\theta}}{4\sqrt{|A|\eta_{\varepsilon}(\sup_{x\in A,y\in B}|x-y|)}}.
\]
\end{proof}

\subsection{Lemmas for Section \ref{exact nonlocalization}}

\begin{lemma}
\label{lem:moment versus profile}Let $\alpha_{d}:=Vol(B(0,1))$ in
$\mathbb{R}^{d}$. Then, $M_{p}(\eta):=\int_{\mathbb{R}^{d}}|y|^{p}\eta(|y|)dy=d\alpha_{d}\int_{0}^{\infty}r^{d+p-1}\eta(r)dr.$
\end{lemma}

\begin{proof}
Recall that the Hausdorff measure of $\partial B(0,r)$ is $\frac{d}{r}Vol(B(0,r))$;
in turn, $Vol(B(0,r))=\alpha_{d}r^{d}$. Now, simply compute that
\begin{align*}
\int_{\mathbb{R}^{d}}|y|^{p}\eta(|y|)dy & =\int_{0}^{\infty}\int_{\partial B(0,r)}r^{p}\eta(r)d\mathcal{H}^{d-1}dr=d\alpha_{d}\int_{0}^{\infty}r^{p+d-1}\eta(r)dr.
\end{align*}
\end{proof}

\begin{lemma}
\label{lem:zeta kernel mass}Given any kernel $\eta(|x-y|)$ satisfying
Assumption \ref{assu:eta properties} (i-iv), let $\zeta_{\eta}(x,y):=\int_{|x-y|}^{\infty}s\eta(s)ds$.
Then, for all $x\in\mathbb{R}^{d}$, 
\[
\int_{\mathbb{R}^{d}}\zeta_{\eta}(x,y)dy=\frac{M_{2}(\eta)}{d}
\]
and, concerning the rescaled kernel $\eta_{\varepsilon}(|x-y|)$,
we have moreover that 
\[
\int_{\mathbb{R}^{d}}\zeta_{(\eta_{\varepsilon})}(x,y)dy=\varepsilon^{2}\frac{M_{2}(\eta)}{d}.
\]
\end{lemma}

\begin{proof}
(i) Compute that 
\begin{align*}
\int_{\mathbb{R}^{d}}\zeta_{\eta}(x,y)dy=\int_{\mathbb{R}^{d}}\left[\int_{|x-y|}^{\infty}s\eta(s)ds\right]dy & =\int_{\mathbb{R}^{d}}\left[\int_{|y|}s\eta(s)ds\right]dy\\
 & =\int_{0}^{\infty}\int_{\partial B(0,r)}\left[\int_{r}^{\infty}s\eta(s)ds\right]d\mathcal{H}^{d-1}dr.
\end{align*}
Notice that $\left[\int_{r}^{\infty}s\eta(s)ds\right]$ is rotationally
invariant, and recall that the Hausdorff measure of $\partial B(0,r)$
is $\frac{d}{r}Vol(B(0,r))$; in turn, $Vol(B(0,r))=\alpha_{d}r^{d}$.
Hence, 
\[
\int_{0}^{\infty}\int_{\partial B(0,r)}\left[\int_{r}^{\infty}s\eta(s)ds\right]d\mathcal{H}^{d-1}dr=d\alpha_{d}\int_{0}^{\infty}\left[\int_{r}^{\infty}s\eta(s)ds\right]r^{d-1}dr.
\]
In turn, using integration by parts, we observe that 
\begin{align*}
\int_{0}^{\infty}\left[\int_{r}^{\infty}s\eta(s)ds\right]r^{d-1}dr & =\left[\left[\int_{r}^{\infty}s\eta(s)ds\right]\frac{r^{d}}{d}\right]_{0}^{\infty}-\int_{0}^{\infty}\left(-r\eta(r)\right)\frac{r^{d}}{d}dr\\
 & =\frac{1}{d}\int_{0}^{\infty}r^{d+1}\eta(r)dr.
\end{align*}
The claim now follows by way of Lemma \ref{lem:moment versus profile}.

(ii) By replacing $\eta$ with $\eta_{\varepsilon}$, the previous
part shows that 
\[
\int_{\mathbb{R}^{d}}\zeta_{(\eta_{\varepsilon})}(x,y)dy=\frac{M_{2}(\eta_{\varepsilon})}{d}.
\]
So we conclude by computing that 
\[
M_{2}(\eta_{\varepsilon})=\int_{\mathbb{R}^{d}}|y|^{2}\frac{1}{\varepsilon^{d}}\eta\left(\frac{|y|}{\varepsilon}\right)dy=\varepsilon^{2}\int_{\mathbb{R}^{d}}|y|^{2}\eta(|y|)dy=\varepsilon^{2}\frac{M_{2}(\eta)}{d}.
\]
\end{proof}

The preceding lemma indicates that $\zeta_{\eta}$ and $\zeta_{(\eta_{\varepsilon})}$
are not convolution kernels (since they are not appropriately normalized).
We therefore introduce their normalizations
\[
\bar{\zeta}_{\eta}:=\frac{d}{M_{2}(\eta)}\zeta_{\eta};\qquad\bar{\zeta}_{(\eta_{\varepsilon})}:=\frac{d}{\varepsilon^{2}M_{2}(\eta)}\zeta_{(\eta_{\varepsilon})}.
\]
Note that under Assumption \ref{assu:eta properties}, $\bar{\zeta}_{\eta}$ is a convolution kernel supported
on the unit ball, while $\bar{\zeta}_{(\eta_{\varepsilon})}$ is a
convolution kernel supported on the ball of radius $\varepsilon$.

\begin{lemma}
\label{lem:zeta-rel-lipschitz-regularity}Assume that $\eta$ satisfies Assumption \ref{assu:eta properties}. Let $K(x)=c_{K}e^{-|x|}$,
where $c_{K}$ is a normalizing constant; and let $\overline{\zeta}_{(\eta_{\varepsilon})}(x)=\frac{1}{\varepsilon^{2}\alpha_{d}\sigma_{\eta}}\int_{|x|}^{\infty}t\eta_{\varepsilon}(t)dt$.
Let $\mu\in\mathcal{P}(\mathbb{R}^d)$. Then, for $0<\varepsilon<\delta$, and
$|x-y|<\varepsilon$, 
\[
K_{\delta}*\mu(x)\left(\frac{1}{1+\frac{3}{\delta}\varepsilon}\right)\leq\overline{\zeta}_{(\eta_{\varepsilon})}*(K_{\delta}*\mu)(x)\leq K_{\delta}*\mu(x)\left(1+\frac{3}{\delta}\varepsilon\right)
\]
and
\[
\frac{1}{\left(1+\frac{3}{\delta}\varepsilon\right)^{2}}\frac{1}{1+\frac{3}{\delta}|x-y|}\leq\frac{\overline{\zeta}_{(\eta_{\varepsilon})}*(K_{\delta}*\mu)(y)}{\overline{\zeta}_{(\eta_{\varepsilon})}*(K_{\delta}*\mu)(x)}\leq\left(1+\frac{3}{\delta}\varepsilon\right)^{2}\left(1+\frac{3}{\delta}|x-y|\right).
\]
\end{lemma}

\begin{proof}
We know that for $|x-y|<\delta$,
\[
\frac{K_{\delta}*\mu(y)}{K_{\delta}*\mu(x)}\leq1+\frac{3}{\delta}|x-y|.
\]
Since (for all $|x-y|<\delta$)
\[
K_{\delta}*\mu(x)\left(\frac{1}{1+\frac{3}{\delta}|x-y|}\right)\leq K_{\delta}*\mu(y)\leq K_{\delta}*\mu(x)\left(1+\frac{3}{\delta}|x-y|\right)
\]
and 
$\overline{\zeta}_{(\eta_{\varepsilon})}$ is a convolution kernel
supported on the ball of radius $\varepsilon$, and $\varepsilon<\delta$,
it follows that for all $x\in X$,
\[
K_{\delta}*\mu(x)\left(\frac{1}{1+\frac{3}{\delta}\varepsilon}\right)\leq\overline{\zeta}_{(\eta_{\varepsilon})}*(K_{\delta}*\mu)(x)\leq K_{\delta}*\mu(x)\left(1+\frac{3}{\delta}\varepsilon\right).
\]
Therefore, if $|x-y|<\delta$,
\[
\frac{K_{\delta}*\mu(y)\left(\frac{1}{1+\frac{3}{\delta}\varepsilon}\right)}{K_{\delta}*\mu(x)\left(1+\frac{3}{\delta}\varepsilon\right)}\leq\frac{\overline{\zeta}_{(\eta_{\varepsilon})}*(K_{\delta}*\mu)(y)}{\overline{\zeta}_{(\eta_{\varepsilon})}*(K_{\delta}*\mu)(x)}\leq\frac{K_{\delta}*\mu(y)\left(1+\frac{3}{\delta}\varepsilon\right)}{K_{\delta}*\mu(x)\left(\frac{1}{1+\frac{3}{\delta}\varepsilon}\right)}
\]
and so 
\[
\frac{1}{\left(1+\frac{3}{\delta}\varepsilon\right)^{2}}\frac{1}{1+\frac{3}{\delta}|x-y|}\leq\frac{\overline{\zeta}_{(\eta_{\varepsilon})}*(K_{\delta}*\mu)(y)}{\overline{\zeta}_{(\eta_{\varepsilon})}*(K_{\delta}*\mu)(x)}\leq\left(1+\frac{3}{\delta}\varepsilon\right)^{2}\left(1+\frac{3}{\delta}|x-y|\right).
\]
\end{proof}
\begin{lemma}[$W_{2}$ convolution estimates]
\label{lem:more convolution estimates} (i) Let $k$ be any radially
symmetric convolution kernel. Then, for any $\mu\in\mathcal{P}_{2}(\mathbb{R}^{d})$,
\[
W_{2}(\mu,\boldsymbol{k_{s}}*\mu)\leq\left(M_2(k)\right)^{1/2}s.
\]

(ii) In the special case of $\bar{\zeta}_{\eta_{\varepsilon}}$, under Assumption \ref{assu:eta properties} (i-iv) we
have that 
\[
W_{2}(\mu,\boldsymbol{\bar{\zeta}_{(\eta_{\varepsilon})}}*\mu)\leq\left(\frac{d}{d+2}\frac{M_{4}(\eta)}{M_{2}(\eta)}\right)^{1/2}\varepsilon.
\]
\end{lemma}

\begin{proof}
(i) Consider the coupling defined by $k_{s}(|x-y|)d\mu(x)dy$. Let
us first check that this is indeed a coupling between $\mu$ and $\boldsymbol{k_{s}}*\mu$.

To see that the first marginal is $\mu$, simply compute that for
any $A\subset\mathbb{R}^{d}$, 
\begin{align*}
\int_{\mathbb{R}^{d}}\int_{A}k_{s}(|x-y|)d\mu(x)dy & =\int_{A}\int_{\mathbb{R}^{d}}k_{s}(|x-y|)dyd\mu(x)\\
 & =\int_{A}d\mu(x)
\end{align*}
since $k_{s}(|x-y|)$ is normalized. On the other hand, $\int_{\mathbb{R}^{d}}k_{s}(|x-y|)d\mu(x)$
is definitionally equal to the density of $\boldsymbol{k_{s}}*\mu$.
So the second marginal of $k_{s}(|x-y|)d\mu(x)dy$ is indeed $\boldsymbol{k_{s}}*\mu$:
for any $A\subset\mathbb{R}^{d}$,
\[
\int_{A}\int_{\mathbb{R}^{d}}k_{s}(|x-y|)d\mu(x)dy=\int_{A}(k_{s}*\mu)(y)dy=(\boldsymbol{k_{s}}*\mu)(A).
\]

Therefore, 
\[
W_{2}^{2}(\mu,\boldsymbol{k_{s}}*\mu)\leq\int\int|x-y|^{2}k_{s}(|x-y|)d\mu(x)dy.
\]
Since for all $x\in\mathbb{R}^{d}$, 
\begin{align*}
\int|x-y|^{2}k_{s}(|x-y|)dy & =\int|y|^{2}k_{s}(|y|)dy=s^{2}\int|y|^{2}k(|y|)dy
\end{align*}
we find that 
\[
W_{2}^{2}(\mu,\boldsymbol{k_{s}}*\mu)\leq\int\left(s^{2}\int|y|^{2}k(|y|)dy\right)d\mu(x)=s^{2}\int|y|^{2}k(|y|)dy.
\]

(ii) By identical reasoning to part (i), we find that $\bar{\zeta}_{(\eta_{\varepsilon})}(|x-y|)d\mu(x)dy$
is indeed a coupling between $\mu$ and $\boldsymbol{\bar{\zeta}_{(\eta_{\varepsilon})}}*\mu$;
and furthermore, 
\[
W_{2}^{2}(\mu,\boldsymbol{\bar{\zeta}_{(\eta_{\varepsilon})}}*\mu)\leq\int|y|^{2}\bar{\zeta}_{(\eta_{\varepsilon})}(|y|)dy.
\]
So, observe that 
\[
\int|y|^{2}\bar{\zeta}_{(\eta_{\varepsilon})}(|y|)dy=\int|y|^{2}\left(\frac{d}{\varepsilon^{2}M_{2}(\eta)}\int_{|y|}^{\infty}s\eta_{\varepsilon}(s)ds\right)dy;
\]
by performing a similar computation to Lemma \ref{lem:zeta kernel mass},
we see that 
\begin{align*}
\int|y|^{2}\left(\int_{|y|}^{\infty}s\eta_{\varepsilon}(s)ds\right)dy & =\int_{0}^{\infty}\int_{\partial B(0,r)}\left[r^{2}\int_{r}^{\infty}s\eta_{\varepsilon}(s)ds\right]d\mathcal{H}^{d-1}dr\\
 & =d\alpha_{d}\int_{0}^{\infty}\left[r^{2}\int_{r}^{\infty}s\eta_{\varepsilon}(s)ds\right]r^{d-1}dr\\
 & =d\alpha_{d}\left(\frac{1}{d+2}\int_{0}^{\infty}r^{d+3}\eta_{\varepsilon}(r)dr\right)\\
 & =\frac{d}{d+2}\alpha_{d}\varepsilon^{4}\int_{0}^{\infty}r^{d+3}\eta(r)dr\\
\text{(Lemma \ref{lem:moment versus profile})} & =\frac{\varepsilon^{4}}{d+2}M_{4}(\eta).
\end{align*}
Hence, $W_{2}^{2}(\mu,\bar{\zeta}_{(\eta_{\varepsilon})}*\mu)\leq\frac{d}{d+2}\frac{M_{4}(\eta)}{M_{2}(\eta)}\varepsilon^{2}$
as desired. 
\end{proof}

\begin{lemma}[Moments and normalizing constant for the Laplace kernel]
\label{lem:Laplace kernel moments} Let $N\in\mathbb{N}$. Concerning
the kernel $K=c_{K}e^{-|x|}$, we have that $c_{K}=(\alpha_{d}d!)^{-1}$
and 
\[
M_{N}(K)=(N+d-1)!/(d-1)!.
\]
\end{lemma}

\begin{proof}
Simply compute that 
\[
c_{K}^{-1}=\int_{\mathbb{R}^{d}}e^{-|x|}dx=\int_{0}^{\infty}\int_{\partial B(0,r)}e^{-r}d\mathcal{H}^{d-1}dr=d\alpha_{d}\int_{0}^{\infty}e^{-r}r^{d-1}dr=d\alpha_{d}(d-1)!
\]
and similarly
\[
M_{N}(K)=c_{K}\int_{\mathbb{R}^{d}}|x|^{N}e^{-|x|}dx=c_{K}d\alpha_{d}\int_{0}^{\infty}e^{-r}r^{N+d-1}dr=c_{K}d\alpha_{d}(N+d-1)!=\frac{(N+d-1)!}{(d-1)!}.
\]
\end{proof}
\subsection{Lemmas for Section \ref{sec:nonlocal hj}}

\begin{lemma}[Hamilton-Jacobi subsolutions closed under convolution in $x$]
\label{lem:convolution of HJ subsoln} Let $\phi_{t}(x)\in BL([0,1]\times\mathbb{R}^{d})$.
Suppose that for almost all $(t,x)\in[0,1]\times\mathbb{R}^{d}$,
\[
\partial_{t}\phi_{t}(x)+H(\nabla\phi_{t}(x))\leq0
\]
where $H:\mathbb{R}^{d}\rightarrow\mathbb{R}_{+}$ is convex. Then,
if $k$ is any convolution kernel, it also holds $(t,x)$-a.e. that
\[
\forall(t,x)\in[0,1]\times\mathbb{R}^{d}\qquad\partial_{t}(k*\phi_{t}(x))+H(\nabla(k*\phi_{t}(x)))\leq0
\]
where $k*\phi_{t}(x):=\int_{\mathbb{R}^{d}}k(x-y)\phi(y)dy$.
\end{lemma}

\begin{proof}
First, since $\phi_{t}(x)\in BL([0,1]\times\mathbb{R}^{d})$, we have
by Rademacher's theorem that $\phi_{t}(x)$ is (jointly) differentiable
almost everywhere, and in particular 
\[
|\partial_{t}\phi_{t}(x)|\leq|\partial_{t}\phi_{t}(x)+\nabla\phi_{t}(x)|\leq\text{Lip}_{[0,1]\times\mathbb{R}^{d}}\phi\qquad(t,x)-\text{a.e.},
\]
so by \cite[Ex. 5.8.135]{bogachev2007measure1} we are justified in
differentiating under the integral sign with respect to the measure
$k(x-y)dy$:
\begin{align*}
\partial_{t}(k*\phi_{t}(x)) & =\partial_{t}\int_{\mathbb{R}^{d}}\phi_{t}(y)k(x-y)dy=\int_{\mathbb{R}^{d}}\partial_{t}\phi_{t}(y)k(x-y)dy=k*(\partial_{t}\phi_{t})(x)\qquad(t,x)-\text{a.e..}
\end{align*}

Second, using the convexity of $H$, we compute that 
\begin{align*}
H(\nabla(k*\phi_{t}(x)) & =H(k*\nabla\phi_{t}(x))\\
 & =H\left(\int_{\mathbb{R}^{d}}\nabla\phi_{t}(y)k(x-y)dy\right)\\
\text{(Jensen)} \;\; & \leq\int_{\mathbb{R}^{d}}H(\nabla\phi_{t}(y))k(x-y)dy.
\end{align*}
Using the fact that $\partial_{t}\phi_{t}(x)+H(\nabla\phi_{t}(x))\leq0$,
we see that 
\begin{align*}
\int_{\mathbb{R}^{d}}H(\nabla\phi_{t}(y))k(x-y)dy & \leq-\int_{\mathbb{R}^{d}}\left(\partial_{t}\phi_{t}(y)\right)k(x-y)dy\\
\text{(a.e.)} & =-\partial_{t}(k*\phi_{t}(x)).
\end{align*}
Therefore, $H(\nabla(k*\phi_{t}(x))\leq-\partial_{t}(k*\phi_{t}(x))$
a.e. as desired.
\end{proof}
\begin{lemma}[{Cf. \cite[(0.10) \textquotedblleft generalized Young's inequality\textquotedblright]{folland1995introduction}}]
\label{lem:Folland Young's ineq} Let $K=c_{K}e^{-|x|}$, and let
$\phi\in L^{\infty}(\mathbb{R}^{d})$. Then 
\[
\Vert\nabla(K_{s}*\phi)\Vert_{\infty}\leq s^{-1}\Vert\phi\Vert_{\infty}.
\]
Moreover, for any $\phi\in BL(\mathbb{R}^{d})$,
\[
\Vert D^{2}(K_{s}*\phi)\Vert_{\infty}\leq s^{-1}\Vert\nabla\phi\Vert_{\infty}.
\]
\end{lemma}

\begin{proof}
Simply compute as follows:
\begin{align*}
|\nabla(K_{s}*\phi)(x)| & \leq\int_{\mathbb{R}^{d}}|\nabla K_{s}(x-y)\phi(y)|dy\\
 & =s^{-(d+1)}\int_{\mathbb{R}^{d}}\left|\nabla K\left(\frac{|x-y|}{s}\right)\phi(y)\right|dy\\
 & \leq s^{-(d+1)}\Vert\phi\Vert_{\infty}\int_{\mathbb{R}^{d}}\left|\nabla K\left(\frac{|y|}{s}\right)\right|dy\\
 & =s^{-1}\Vert\phi\Vert_{\infty}\int_{\mathbb{R}^{d}}|\nabla K(|y|)|dy.
\end{align*}
Note that $c_{K}$ is chosen to normalize the integral of $e^{-|x|}$;
this shows that $\Vert\nabla K(x)\Vert_{1}=1$. Hence, $\Vert\nabla(K_{s}*\phi)\Vert_{\infty}\leq s^{-1}\Vert\phi\Vert_{\infty}$.
The proof that $\Vert D^{2}(K_{s}*\phi)\Vert_{\infty}\leq s^{-1}\Vert\nabla\phi\Vert_{\infty}$
is similar. 
\end{proof}

\begin{lemma}[Convolution decreases
size of gradient]
\label{lem:gradient convolution contraction} Let k be any convolution kernel, and let $\phi\in BL(\mathbb{R}^{d})$.
\end{lemma}
\begin{proof}
First, compute that for any unit vector $\vec{v}$, 
\[
\Vert\nabla_{\vec{v}}(k*\phi)\Vert_{\infty}=\Vert k*(\nabla_{\vec{v}}\phi)\Vert_{\infty}\leq\Vert\nabla_{\vec{v}}\phi\Vert_{\infty} \leq  \Vert\nabla \phi\Vert_{\infty} .
\]
Taking the supremum over  unit vectors $v$ establishes the claim .
\end{proof}
\begin{lemma}
\label{lem:BL stable convolution in space}Let $\phi_{t}(x)\in$$BL([0,1]\times\mathbb{R}^{d})$.
Let $k$ be a convolution kernel. Then, 
$k*\phi_{t}(x):=\int_{\mathbb{R}^{d}}k(x-y)\phi_t(y)dy$ also belongs to $BL([0,1]\times\mathbb{R}^{d})$.
\end{lemma}

\begin{rem}
If, instead, we were smoothing $\phi_{t}(x)$ in $[0,1]\times\mathbb{R}^{d}$
rather than just $\mathbb{R}^{d}$, then this would follow immediately
from the previous two lemmas.
\end{rem}

\begin{proof}
For each $t$, we know that $\Vert k*\phi_{t}(x)\Vert_{L^{\infty}(\mathbb{R}^{d})}\leq\Vert\phi_{t}(x)\Vert_{L^{\infty}(\mathbb{R}^{d})}$.
Therefore, $\Vert k*\phi_{t}(x)\Vert_{L^{\infty}([0,1]\times\mathbb{R}^{d})}\leq\Vert\phi_{t}(x)\Vert_{L^{\infty}([0,1]\times\mathbb{R}^{d})}$. 

In order to check that $\text{Lip}_{[0,1]\times\mathbb{R}^{d}}\phi_{t}(x)<\infty$,
it suffices to check that $\phi_{t}(x)$ is Lipschitz in $x$ for
uniformly for fixed $t$, and vice versa. By Lemma \ref{lem:gradient convolution contraction}
(together with Rademacher's theorem), we know that $\text{Lip}_{\mathbb{R}^{d}}k_{s}*\phi_{t}(x)\leq\text{Lip}_{\mathbb{R}^{d}}\phi_{t}(x)$.
Therefore it remains only to check that $k*\phi_{t}(x)$ is Lipschitz
in $t$ for each fixed $x$.

So compute that 
\begin{align*}
|k*\phi_{t_{1}}(x)-k*\phi_{t_{2}}(x)| & =\left|\int_{\mathbb{R}^{d}}k(x-y)(\phi_{t_{1}}(y)-\phi_{t_{2}}(y))dy\right|\\
 & \leq\int_{\mathbb{R}^{d}}k(x-y)\text{Lip}_{[0,1]}\phi_{t}(y)\cdot|t_{1}-t_{2}|dy\\
 & \leq\sup_{y\in\mathbb{R}^{d}}\text{Lip}_{[0,1]}\phi_{t}(y)\cdot|t_{1}-t_{2}|\\
 & \leq\text{Lip}_{[0,1]\times\mathbb{R}^{d}}\phi_{t}(y)\cdot|t_{1}-t_{2}|.
\end{align*}
\end{proof}


\end{document}